\documentclass[11pt]{article}
\usepackage[T1]{fontenc}
\usepackage[latin9]{inputenc}
\usepackage{geometry}
\usepackage{amsmath}
\usepackage{amssymb}
\usepackage{amsthm}
\usepackage{comment}
\usepackage{xcolor}

\usepackage{mathrsfs}
\usepackage{enumitem}

\numberwithin{equation}{section} 
\newtheorem{Theorem}{Theorem}[section]
\newtheorem{Definition}[Theorem]{Definition}

\newtheorem{Proposition}[Theorem]{Proposition}
\newtheorem{Lemma}[Theorem]{Lemma}
\newtheorem{Remark}[Theorem]{Remark}

\DeclareMathOperator*{\dv}{div}
\newcommand{\dvm}{{\rm div}_m} 

\DeclareMathOperator*{\tr}{Tr}

\newcommand{\s}{s}

\newcommand{\eps}{\varepsilon}

\DeclareMathOperator{\Proj}{{\rm P}}
\newcommand{\XN}{\mathbf{X}_N} 

\title{Global weak solutions to the Quantum Navier-Stokes system in the whole space with a confining external potential}

\author{J\'er\'emy Faupin$^1$, Ingrid Lacroix-Violet$^2$, Julien Lequeurre$^3$}

\begin{document}
	
	\maketitle

\begin{abstract}

We consider a dissipative quantum fluid on the whole space $\mathbb{R}^d$ ($d\geq 1$) confined by an external harmonic potential. The dynamics of the quantum fluid is described by the Quantum Navier-Stokes (QNS) system which is a particular case of the Navier-Stokes-Korteweg systems. The goal of this paper is to prove the existence of global weak solutions to the QNS system. To this end, we write the evolution equations with respect to a Gaussian reference measure and follow the general strategy of previous works \cite{Bresch_Gisclon_Lacroix-Violet_ARMA19, Carles_Carrapatoso_Hillairet_AIF2022, Lacroix-Violet_Vasseur_JMPA18, Vasseur_Yu_SIMA2016}. Nevertheless, several substantial modifications have to be done due to our choice of the reference measure. 

\end{abstract}

	\tableofcontents
	
\section{Introduction}

Quantum fluid models have attracted lots of attention in the physics literature in the last few decades, due to the variety of applications. Among others, quantum fluid models can indeed be used to describe superfluids \cite{LoMo93}, quantum semiconductors \cite{FeZhou93}, weakly interacting Bose gases \cite{Grant73} and quantum trajectories of Bohmian mechanics \cite{Wyatt_book}.

In this paper, we will be interested in a dissipative quantum fluid model, the so-called Quantum Navier-Stokes (QNS) model, in the whole space $\mathbb{R}^d$ (in any dimension $d\ge1$) in the presence of a confining external potential. This model has been derived in \cite{BruMe09}, starting with a quantum system satisfying the Liouville-von Neumann equation, considering the associated Wigner equation and using a Chapman-Enskog expansion around a quantum local equilibrium. Roughly spea\-king, the QNS system corresponds to the classical Navier-Stokes system with a quantum correction term.  The main difficulties encountered in the mathematical analysis of this model lie in the highly nonlinear structure of the third order quantum term, and the proof of the positivity (or non-negativity) of the particle density. 

The system that we study in this paper reads, on the whole space $\mathbb{R}^d$ (for $d\geq 1$ an integer):
	\begin{subequations}\label{eq:CNSK_system}
		\begin{eqnarray}
			\partial_t\rho + \dv(\rho u) & = & 0, \label{eq:CNSK_system_density} \\
			\partial_t(\rho u) + \dv(\rho u \otimes u) & = & \dv(2\nu \rho D(u))  + 2\kappa^2 \rho\nabla\left(\frac{\Delta(\sqrt{\rho})}{\sqrt{\rho}}\right) - a \nabla \rho - \lambda \rho x, \label{eq:CNSK_system_velocity}
		\end{eqnarray}
	\end{subequations}
where, as usual, $\rho$ stands for the density of the fluid, $u$ its velocity, $\nu$,  $a$, $\kappa$ and $\lambda$ are positive constants, the matrix $D(u)$ is the symmetric part of $\nabla u = (\partial_{j}u_i)_{i,j}$, defined by $D(u) = \frac{\nabla u + \nabla u^\top}{2}$ (where, for every matrix $M$, the matrix $M^\top$ is the transpose of $M$). This explicit form of the system that we consider follows from particular choices for the pressure, viscosity, capillarity tensor and external potential. 

The  QNS model \eqref{eq:CNSK_system} indeed belongs to a more general class of models, the Navier-Stokes-Korteweg models. Let us mention that the latter are widely used in theoretical physics to describe capillarity phenomena in fluids with diffusing interfaces. The Navier-Stokes-Korteweg model are associated to the system of equations
\begin{subequations}\label{eq:CNSK_system_general}
\begin{eqnarray}
\partial_t\rho + \dv(\rho u) & = & 0, \label{eq:CNSK_system_general_density} \\
\partial_t(\rho u) + \dv(\rho u \otimes u) & = & \dv(2\mu_F(\rho)D(u) + \lambda_F(\rho)\tr(D(u)){\rm I}_d)   - \nabla p(\rho)  \label{eq:CNSK_system_general_velocity} \\
 & & \quad + \rho\nabla\left(K(\rho)\Delta\rho + \frac{1}{2}K'(\rho) |\nabla\rho|^2\right) - \rho\nabla V,  \notag
\end{eqnarray}
\end{subequations}
where $p(\rho) = a\rho^\gamma$ is the pressure ($a>0$ and $\gamma \geq 1$ are two constants), $K(\rho)$ is the function in the capillarity tensor (Korteweg tensor), ${\rm I}_d$ is the identity matrix, and $\mu_F(\rho)$ and $\lambda_F(\rho)$ are the fluid viscosities that satisfy the Bresch-Desjardins condition, see \cite{Bresch-Desjardins_CRAS04},
\begin{equation}\label{BD_condition}
\lambda_F(\rho) = 2(\rho\mu_F'(\rho)-\mu_F(\rho)).
\end{equation}
Finally the potential $V$ is a function of the space variable $x$ only, which represents the external forces applied on the fluid. 
The QNS model is then given by the equations \eqref{eq:CNSK_system_general} in the particular case where $\mu_F$, $\lambda_F$ and $p$ in \eqref{eq:CNSK_system_general} are given respectively by
\begin{equation}\label{def:mu_lambda_p}
\mu_F(\rho) = \nu\rho, \quad \lambda_F(\rho) = 0, \quad p(\rho) = a\rho,
\end{equation}
with $\nu>0$, $a>0$ (note that $\mu_F$, $\lambda_F$ then satisfy the Bresch-Desjardins condition \eqref{BD_condition}) and the external potential  is supposed to be harmonic,
\begin{equation}
V(x)=\frac{\lambda}{2}|x|^2, \quad x\in\mathbb{R}^d,
\end{equation}
with $\lambda>0$. The assumption that $V$ is harmonic leads to a simplification of several formulas below. Finally, we also assume that 
\begin{equation}\label{def:K(rho)quantum}
K(\rho) = \frac{\kappa^2}{\rho},
\end{equation}
with $\kappa >0$ a constant. Note that in this case the Korteweg tensor in \eqref{eq:CNSK_system_general} reduces indeed to $\rho\nabla\left(2\kappa^2\frac{\Delta \sqrt{\rho}}{\sqrt{\rho}}\right)$ (see Appendix \ref{appendixA}).

As mentioned before, the system of equations \eqref{eq:CNSK_system_general} with $K(\rho)$ given by \eqref{def:K(rho)quantum} was derived in \cite{BruMe09} as a `hydrodynamic limit' (more precisely, a Chapman-Enskog expansion up to order $1$ around a local quantum equilibrium) of the Wigner equation describing the evolution of the density of a quantum model. Note that the potential force associated to $V$ in  \eqref{eq:CNSK_system_general} originates from the external potential applied to the quantum system considered in \cite{BruMe09}. We refer the reader to \cite{BruMe09} and references therein for more details concerning the physics associated to the Quantum-Navier-Stokes model.

\vspace{1em}

General classes of Navier-Stokes-Korteweg systems have been studied since the last two decades, especially in the case of barotropic pressures (that is $p(\rho)$ proportional to $\rho^\gamma$ with $\gamma >1$), and in finite volume domains (generally tori in two or three dimensions), without an external force, for different expressions of the Korteweg function $K(\rho)$. Let us mention, among others, \cite{Jungel_SIMA10,Lacroix-Violet_Vasseur_JMPA18,Lu_Zhang_Zhong_JMP19} when $K(\rho) \propto \frac{1}{\rho}$, \cite{Antonelli_Spirito_AHP22} when $K(\rho) = 1$,  \cite{Bresch_Gisclon_Lacroix-Violet_ARMA19} when $K(\rho)$ and $\mu_F(\rho)$ satisfy an algebraic equation (namely $\sqrt{\rho K(\rho)} = \mu_F'(\rho)$) with $K(\rho) \propto \rho^s$, or \cite{Bresch_Gisclon_Lacroix-Violet_Vasseur_JMFM22} for more general expressions of $K(\rho)$.

In the whole space $\mathbb{R}^d$, fewer references exist. In \cite{Antonelli_Hientzsch_Spirito_JDE21}, in dimension $d=2$ or $3$, the Quantum Navier-Stokes system ($K(\rho) \propto \frac{1}{\rho}$, $\mu_F(\rho) \propto \rho$, $\lambda_F(\rho) = 0$) with non trivial far-field behavior $\rho(x) \underset{|x|\to \infty}{\longrightarrow} 1$ is considered.
In \cite{Carles_Carrapatoso_Hillairet_AIF2022}, the quantum `isothermal fluid' case ($p(\rho) \propto \rho$) without an external force is considered in $\mathbb{R}^d$ with $d\leq 3$, that is the system \eqref{eq:CNSK_system} with $\lambda = 0$,  by approaching $\mathbb{R}^d$ by finite volume tori after introducing auxiliary variables obtained thanks to a time-dependent rescaling (a similar rescaling was already used in \cite{Carles_Carrapatoso_Hillairet_AHL2018} hinted by the link between the Euler-Korteweg equations (when $\mu_F(\rho) = \lambda_F(\rho) = 0$ and $K(\rho) \neq 0$) and the corresponding nonlinear Schrödinger equation through the Madelung transform). Anticipating Section \ref{sect:CCH} below, we mention that the method developed in the present paper provides an alternative proof of the result established in \cite{Carles_Carrapatoso_Hillairet_AIF2022}, without relying on an approximation of $\mathbb{R}^d$ by finite volume tori.

In the previous references, either in tori or in the whole space, global existence of weak solutions is proved for initial data of bounded initial energy and bounded initial Bresch-Dejardins Entropy (BD Entropy). See Section \ref{subsect:Introduction_QNS_drag} below for the definitions of the energy and BD entropy in our context. With an added drag term $-r_2\rho u$ (for $r_2 >0$), the authors in \cite{Bresch_Gisclon_Lacroix-Violet_Vasseur_JMFM22} prove the convergence of the global weak solution to the equilibrium of the system with an exponential rate of convergence.

Note that most of, if not all, the previous references consider fluid viscosities $\mu_F$ and $\lambda_F$ satisfying the BD relation \eqref{BD_condition}.

\vspace{1em}

Our main goal in this paper is to prove the existence of global weak  solutions to \eqref{eq:CNSK_system}. Before stating our result in precise terms (see Theorem \ref{thm:main} below) we begin with stating some useful basic properties of the system \eqref{eq:CNSK_system} and reformulating it in a more appropriate setting.

It would be interesting to extend the analysis presented in this paper to a more general setting with $\mu_F$, $\lambda_F$ functions satisfying the Bresch-Desjardins condition, $p(\rho) = a\rho^\gamma$ with $\gamma\ge1$ and $V$ a suitable confining potential.

\subsection{Basic properties of the QNS system and change of reference measure}

\paragraph{Mass conservation.}

We will assume that the initial mass of the fluid, $m_F = \int_{\mathbb{R}^d} \rho^0$, is finite. Therefore, by the conservation of mass along the evolution, the following equality
\begin{equation}\label{eq:mass}
\int_{\mathbb{R}^d} \rho(t,x)\,{\rm d}x = m_F
\end{equation}
is satisfied for all time $t\geq0$. Indeed, under suitable assumptions on solution $(\rho,u)$, integrating equation \eqref{eq:CNSK_system_density} on $\mathbb{R}^d$ leads to $\frac{{\rm d}}{{\rm d}t}\int_{\mathbb{R}^d} \rho(t,x) \,{\rm d}x = 0$ for all time $t\geq 0$.

Moreover, because of the particular expression of $K(\rho)$, $\mu_F(\rho)$, $\lambda_F(\rho)$ and $p(\rho)$, see \eqref{def:K(rho)quantum} and \eqref{def:mu_lambda_p}, the system \eqref{eq:CNSK_system} is homogeneous in the density. In other words, changing $\rho$ to $\frac{\rho}{m_F}$, one sees that \eqref{eq:mass} simplifies to 
\begin{equation}\label{eq:mass_1}
\int_{\mathbb{R}^d} \rho(\cdot,x) \,{\rm d}x = 1
\end{equation}
and that $\left(\displaystyle \frac{\rho}{m_F}, u\right)$ satisfies  \eqref{eq:CNSK_system} whenever $(\rho, u)$ does.

We will then assume in the following that the initial density $\rho^0$ satisfies \eqref{eq:mass_1}.

\paragraph{Relative energy.}

It is convenient to introduce the relative energy  (with respect to the minimal energy) of the QNS system \eqref{eq:CNSK_system} as we describe now. First, the energy associated to the system \eqref{eq:CNSK_system} is given by
\begin{equation}\label{eq:def_energy_intro}
E(\rho,u) = \frac12 \int_{\mathbb{R}^d} \rho|u|^2 + a\int_{{\mathbb{R}^d}}\rho \ln(\rho) + \frac{\kappa^2}{2}\int_{{\mathbb{R}^d}} \rho|\nabla \ln(\rho)|^2 + \frac{\lambda}{2}\int_{{\mathbb{R}^d}} \rho|x|^2.
\end{equation}
A direct computation (see Appendix \ref{appendixA}) then shows that $E(\rho,u)$ has a unique global minimizer $(\rho_m,u_m)$ in the set $\displaystyle \left\{(\rho,u) \text{ measurable such that } \rho > 0 , \; \int_{\mathbb{R}^d}\rho=1, \; E(\rho,u)<\infty \right\}$, given by
\begin{equation}\label{def:rho_m_u_m}
\rho_m(x)=\frac{1}{(2\pi\sigma^2)^{\frac{d}2}} \exp\left(-\frac{1}{2\sigma^2}|x|^2\right) , \quad u_m(x)=0,\quad  x\in\mathbb{R}^d, 
\end{equation}
with $\sigma>0$ solution of the equation 
	\begin{equation}\label{def:sigma}
	\frac{a}{\sigma^2}+\frac{\kappa^2}{\sigma^4} = \lambda.
	\end{equation}
The minimal energy is 
\begin{equation}\label{eq:def_energy_minimal}
E(\rho_m,0) = d\left(\frac{\kappa^2}{\sigma^2}-\frac{a}{2}\ln(2\pi\sigma^2)\right)
\end{equation}
and one can then introduce the (relative) energy $\mathcal{E}(\rho,u)$ as
\begin{equation}\label{eq:CNSK_energy}
\mathcal{E}(\rho,u) = E(\rho,u) - E(\rho_m,0) = \frac12 \int_{\mathbb{R}^d} \rho|u|^2 + a\int_{{\mathbb{R}^d}}\rho \ln\left(\frac{\rho}{\rho_m}\right) + \frac{\kappa^2}{2}\int_{{\mathbb{R}^d}} \rho\left|\nabla \ln\left(\frac{\rho}{\rho_m}\right)\right|^2.
\end{equation}
For a justification of the second equality, see again Appendix \ref{appendixA}.

It should be observed that the system satisfies the energy estimate
\begin{equation*}\label{eq:CNSK_energy_estimate}
\frac{{\rm d}}{{\rm d}t}\mathcal{E}(\rho,u) + \mathcal{D}(\rho,u) = 0,
\end{equation*}
where, as in the classical Navier-Stokes system, the dissipation is given by
\begin{equation}\label{eq:CNSK_dissipation}
	\mathcal{D}(\rho,u)  =  2\nu \int_{\mathbb{R}^d} \rho|D(u)|^2.
\end{equation}

\paragraph{Mean position and velocity.}

By a simple change of unknown variables, we can assume without loss of generality that the mean position and mean velocity of the fluid vanish at time $t=0$, and therefore for all time. This is a consequence of the following formal property:
	Let $(\rho,u)$ be a solution to \eqref{eq:CNSK_system_density}--\eqref{eq:CNSK_system_velocity} such that $\int_{\mathbb{R}^d} \rho = 1$. Denoting, for all time $t\geq 0$, $M_x$ the mean position and $M_u$ the mean velocity of the fluid,
	\begin{equation*}
	M_x(t) = \int_{\mathbb{R}^d}\rho(t,x)x \,{\rm d}x \qquad \mbox{and} \qquad M_u(t) = \int_{\mathbb{R}^d}\rho(t,x)u(t,x) \,{\rm d}x,
	\end{equation*}
	we have
	\begin{equation*}
	\dot {M_x}(t) = M_u(t) \quad \mbox{and} \quad \dot M_u(t) = -\lambda M_x(t).
	\end{equation*}
	This directly follows from \eqref{eq:CNSK_system}. In particular, the mean position $M_x$ has a periodic trajectory if either $M_x(0)$ or $M_u(0)$ is non zero.
By setting $\overline{\rho}(t,x) = \rho(t,x+M_x(t))$ and $\overline{u}(t,x) = u(t,x+M_x(t)) - M_u(t)$, it is then easy to see that the new unknowns $(\overline{\rho},\overline{u})$ satisfy the same equations \eqref{eq:CNSK_system_density}--\eqref{eq:CNSK_system_velocity} as $(\rho,u)$, with vanishing mean position and mean velocity.

\paragraph{Change of reference measure.}

Our next concern is to rewrite the system \eqref{eq:CNSK_system} with a different reference measure. This turns out to be more natural in our context. More precisely, instead of considering the Lebesgue measure on $\mathbb{R}^d$, we aim at writing the system with respect to the measure $\mu_m$ such that 
\begin{equation*}
{\rm d}\mu_m = \rho_m \,{\rm d}x, \quad \rho \, {\rm d}x = q \, {\rm d}\mu_m,  \quad q = \frac{\rho}{\rho_m}.
\end{equation*}
For $q:\mathbb{R}^d\to\mathbb{R}$ and $v:\mathbb{R}^d\to\mathbb{R}^d$, we introduce the notations
\begin{align*}
\dvm(v) = \frac{1}{\rho_m}\dv(\rho_m v), \quad \dvm(D(v)) = \frac{1}{\rho_m}\dv(\rho_m D(v)), \quad \Delta_m q = \dvm(\nabla q).
\end{align*}
These twisted derivatives encode the change of measure of reference, and simplify some integration by parts formulae, see \cite{Gianazza_Savare_Toscani_ARMA09}.  For instance, for $v$, $w : \mathbb{R}^d \to \mathbb{R}^d$, $q : \mathbb{R}^d \to \mathbb{R}$ such that all quantities below are well-defined, we have
\begin{equation*}
\int_{\mathbb{R}^d} \nabla q \cdot v \,{\rm d}\mu_m = - \int_{\mathbb{R}^d} q \dvm(v) \,{\rm d}\mu_m, \qquad \int_{\mathbb{R}^d} \dvm(D(v)) \cdot w \,{\rm d}\mu_m = - \int_{\mathbb{R}^d} \tr\left(D(v)D(w)\right)\, {\rm d}\mu_m.
\end{equation*}
Straightforward computations (see Appendix \ref{appendixA}) show that the QNS system \eqref{eq:CNSK_system} rewrites in the new reference measure as
\begin{subequations}
	\label{eq:CNSK_system_m}
	\begin{align}
	\partial_t q + \dvm(qu) &= 0,      \label{eq:CNSK_system_m-density} \\
	\partial_t (qu) + \dvm(u \otimes qu) &= \dvm\left(2\nu q D(u)\right) +  \dvm\left(\kappa^2qD^2(\ln(q))\right) - \lambda \sigma^2 q \nabla \left(\ln(q) \right). \label{eq:CNSK_system_m-velocity}
	\end{align}
\end{subequations}

We emphasize that, firstly, in this new formulation the potential $V$ is absorbed into the different twisted differential operators, see \eqref{eq:potential_force_m}, and secondly that we will mainly consider this system in the following.

\subsection{The QNS system with drag forces}
\label{subsect:Introduction_QNS_drag}

Following \cite{Bresch-Desjardins_CRAS04, Carles_Carrapatoso_Hillairet_AIF2022, Lacroix-Violet_Vasseur_JMPA18, Vasseur_Yu_SIMA2016}, we now introduce smoothing terms corresponding to `drag' forces into the system. As in these references, the smoothing terms will be crucial ingredients in our proof of the existence of global solutions to \eqref{eq:CNSK_system_m}. More precisely, we will first obtain the existence of solutions for a family of regularized systems with drag terms depending on some positive parameters, and then show that these solutions converge, in the limit where the parameters, and therefore the drag terms, vanish. 

\paragraph{The drag forces.}

Our regularization consists in adding the three terms $-r_0u$,  $-r_1q |u|^2u$ and  $- \frac{r_4}{\sigma^4} q |x|^2x$ to the right-hand side of \eqref{eq:CNSK_system_m-velocity}, with $r_0 >0$, $r_1 >0$ and $r_4 >0$. The system then becomes 
\begin{subequations}
	\label{eq:CNSK_system_m_regularized_intro}
	\begin{eqnarray}
	\partial_t q + \dvm(qu) &=& 0,      \label{eq:CNSK_system_m_regularized-density_intro} \\
	\partial_t (qu) + \dvm(qu\otimes u) &=&  \displaystyle \dvm\left(2\nu q D(u)\right) + \dvm\left(\kappa^2qD^2(\ln(q))\right) -\lambda \sigma^2 \nabla q \label{eq:CNSK_system_m_regularized-velocity_intro} \\
	& & \qquad  - r_0 u -r_1q |u|^2u- \frac{r_4}{\sigma^4} q |x|^2x.  \notag
	\end{eqnarray}
\end{subequations}
The terms $-r_0u$ and $-r_1q|u|^2u$ are similar to those previously used in \cite{Lacroix-Violet_Vasseur_JMPA18, Vasseur_Yu_SIMA2016}. We emphasize, however, that the term $- \frac{r_4}{\sigma^4} q |x|^2x$ dit not appear previously in the literature, to our knowledge. As explained below, this term is especially fitted to the model we consider.

Note that other `drag' terms exist in the literature. First, in \cite{Bresch_Desjardins_JMPA06}, the authors introduced the term $- r_1 \rho|u|u$ (with $r_1 > 0$) in lieu of $-r_1 \rho |u|^2u$. Second, the drag term $-r_2 \rho u$ (with $r_2 > 0$) is used in \cite{Bresch_Gisclon_Lacroix-Violet_Vasseur_JMFM22} to prove the exponential convergence to equilibrium of the global weak solution to the compressible Navier-Stokes-Korteweg equations in $\mathbb{T}^3$.

\paragraph{Energy identity.}

As in \eqref{eq:CNSK_energy}--\eqref{eq:CNSK_dissipation}, the QNS system with drag forces \eqref{eq:CNSK_system_m_regularized_intro} enjoys (formally at this stage) the following energy identity
\begin{equation}\label{eq:energy_identity_reg_intro}
\frac{{\rm d}}{{\rm d}t} \mathcal{E}(t) + \mathcal{D}(t) = 0,
\end{equation}
where the `regularized' energy $\mathcal{E}$ is now given by
\begin{equation}\label{def:energy_reg_intro}
\mathcal{E} = \mathcal{E}(q,u)= \int_{\mathbb{R}^d} q\left[\frac12\left(|u|^2 + \kappa^2|\nabla\ln(q)|^2\right)+ a\ln(q)\right]\;{\rm d}\mu_m + \frac{r_4}{4}I_4(q)
\end{equation}
where $I_4(q)$, the fourth order moment of $q$, is given by
\begin{equation}\label{def:I4}
I_4(q) = \int_{\mathbb{R}^d} q \left|\frac{x}{\sigma}\right|^4 {\rm d}\mu_m
\end{equation}
and the `regularized' dissipation $\mathcal{D}$ is given by
\begin{equation}\label{def:dissipation_reg_intro}
		\mathcal{D}=\mathcal{D}(q,u)  =  \displaystyle 2\nu \int_{\mathbb{R}^d} q|D(u)|^2 \;{\rm d}\mu_m +r_0 \int_{\mathbb{R}^d} |u|^2 \;{\rm d}\mu_m+ r_1\int_{\mathbb{R}^d} q|u|^4 \;{\rm d}\mu_m .
\end{equation}

\paragraph{BD entropy identity.}

The Bresch-Desjardins entropy (BD entropy) is a fundamental tool to prove the existence and study the solutions of a large class of compressible fluid systems, see \cite{Bresch-Desjardins_CRAS04, Bresch-Desjardins-Lin_CPDE03} for instance. In our context, the BD entropy for the QNS system with drag forces \eqref{eq:CNSK_system_m_regularized_intro} is given by
	\begin{eqnarray}\label{def:BD_reg-entropie_intro}
\mathcal{E}_{\rm BD}=\mathcal{E}_{\rm BD}(q,u) & = & \int_{\mathbb{R}^d} q\left[\frac12\left(|u+2\nu \nabla\ln(q)|^2 + \kappa^2|\nabla\ln(q)|^2\right)+ a\ln(q)\right]\;{\rm d}\mu_m \notag \\
& & + 2\nu r_0 \int_{\mathbb{R}^d} (q-\ln(q)) \;{\rm d}\mu_m + \frac{r_4}{4}I_4(q).
\end{eqnarray}
Note that, by concavity of the function $\ln$, $q-\ln(q)$ is nonnegative and hence $\mathcal{E}_{\rm BD}$ is nonnegative.

The corresponding dissipation term is
	\begin{eqnarray}\label{def:BD_reg-dissipation_intro} 
	\mathcal{D}_{\rm BD}=\mathcal{D}_{\rm BD}(q,u) & = & 2\nu \int_{\mathbb{R}^d} q|A(u)|^2 \;{\rm d}\mu_m + 2\nu\lambda \sigma^2 \int_{\mathbb{R}^d} q|\nabla \ln(q)|^2  \;{\rm d}\mu_m  + 2 \kappa^2 \nu \int_{\mathbb{R}^d} q |D^2(\ln(q))|^2  \;{\rm d}\mu_m  \notag \\
	& & +r_0 \int_{\mathbb{R}^d} |u|^2 \;{\rm d}\mu_m + r_1\int_{\mathbb{R}^d} q|u|^4 {\rm d}\mu_m + \frac{2r_4\nu}{\sigma^2}I_4(q), 
\end{eqnarray}
and we have the BD entropy equality
	\begin{equation}\label{eq:BD-entropy_equality_intro}
	\frac{{\rm d}}{{\rm d}t} \mathcal{E}_{\rm BD}(t) +\mathcal{D}_{\rm BD} (t) = \mathcal{R}_{\rm BD}(t),
	\end{equation}
	where, denoting $I_2(q)$, the second order moment of $q$, by
	\begin{equation}\label{def:I2}
	I_2(q) = \int_{\mathbb{R}^d} q\left|\frac{x}{\sigma}\right|^2 \, {\rm d}\mu_m,
	\end{equation}
	the right-hand side $\mathcal{R}_{\rm BD}$ is given by
\begin{eqnarray}
	\mathcal{R}_{\rm BD}=\mathcal{R}_{\rm BD}(q,u) & = & \frac{2r_4\nu(d+2)}{\sigma^2}I_2(q) - 2\nu r_1\int_{\mathbb{R}^d} q|u|^2 u \cdot \nabla \ln(q)  \;{\rm d}\mu_m \notag \\
	& & + \frac{2\nu}{\sigma^2}\int_{\mathbb{R}^d} q\left(u+2\nu\nabla\ln(q)\right) \cdot u  \;{\rm d}\mu_m.
	\label{def:BD_reg-RHS_intro}
	\end{eqnarray}

\begin{Remark}
	Note that, we may also define a `modified' BD Entropy, and the corresponding dissipation and reaction terms satisfying \eqref{eq:BD-entropy_equality_intro} too, as follows:
	\begin{eqnarray*}
	\mathcal{E}_{\rm BD} & = & \int_{\mathbb{R}^d} q\left[\frac12\left(|u+2\nu \nabla\ln(\rho)|^2 + \kappa^2|\nabla\ln(q)|^2\right)+ a\ln(q)\right]\;{\rm d}\mu_m + \frac{r_4}{4}I_4(q) \\
	\mathcal{D}_{\rm BD} & = & 2\nu \int_{\mathbb{R}^d} q|A(u)|^2 \;{\rm d}\mu_m + 2\nu\left(\lambda \sigma^2 + \frac{\kappa^2}{\sigma^2}\right) \int_{\mathbb{R}^d} q|\nabla \ln(q)|^2  \;{\rm d}\mu_m  + 2\nu\kappa^2 \int_{\mathbb{R}^d} q |D^2(\ln(q))|^2  \;{\rm d}\mu_m \notag \\
	&  & + \frac{2\nu r_4}{\sigma^2}\left[I_4(q)- (d+2) I_2(q)\right] + 2\nu \lambda I_2(q), \\ 
	\mathcal{R}_{\rm BD} & = & 2\nu \lambda d. \notag 
	\end{eqnarray*}
	In this case, the dissipation term $\mathcal{D}_{\rm BD}$ 
	above is non negative (for all $q$) whenever $r_4 \geq 0$ satisfies $\lambda - \frac{2\nu r_4(d+2)}{\sigma^2} \geq 0$, which is the case in particular for $r_4 = 0$.
\end{Remark}

\paragraph{(Renormalized) Second moment of the density.} We conclude this section with the following observation that may be useful for some purpose, even if we do not use it in the present paper. Let us define (recall that $m_F = 1$) 
\begin{equation}\label{def:tildeI2}
\widetilde{I}_2(q) = I_2(q) - I_2(1) = \int_{\mathbb{R}^d} q \left|\frac{x}{\sigma}\right|^2  \;{\rm d}\mu_m - d.
\end{equation}
Then
\begin{equation*}
\frac{{\rm d} \widetilde{I}_2(q)}{{\rm d}t} = 2 \int_{\mathbb{R}^d} qu\cdot \frac{x}{\sigma^2} \;{\rm d}\mu_m
\end{equation*}
and $\widetilde{I}_2(q)$ is solution to the following ODE (obtained by taking the scalar product of \eqref{eq:CNSK_system_m_regularized-velocity_intro} by $\frac{x}{\sigma^2}$ and integrating on $\mathbb{R}^d$):
\begin{eqnarray*}
	& &\frac{{\rm d}^2}{{\rm d}t^2}{\widetilde{I_2}(q)} + \frac{2\nu}{\sigma^2} \frac{{\rm d}}{{\rm d}t}{\widetilde{I_2}(q)} + 2\left[\lambda + \frac{\kappa^2}{\sigma^4}\right]\widetilde{I_2}(q)\\
	&  &= \frac{2}{\sigma^2}\left[\int_{\mathbb{R}^d} q|u|^2 \;{\rm d}\mu_m + \kappa^2 \int_{\mathbb{R}^d}q|\nabla \ln(q)|^2  \;{\rm d}\mu_m\right] + \frac{4\nu}{\sigma^2} \int_{\mathbb{R}^d} \nabla q \cdot u \;{\rm d}\mu_m \\ 
	&  & \quad - \frac{2r_0}{\sigma^2} \int_{\mathbb{R}^d} u \cdot x    \;{\rm d}\mu_m - \frac{2r_1}{\sigma^2} \int_{\mathbb{R}^d} q|u|^2u \cdot x - \frac{2r_4}{\sigma^2} I_4(q).
\end{eqnarray*}
This ODE is particularly interesting when the drag term coefficients $r_0$, $r_1$ and $r_4$ vanish. In that case, the right-hand side is bounded (thanks to the energy or BD-entropy estimates). Nevertheless, it seems that we lack information to prove convergence (at exponential rate) of $\widetilde{I}_2(q(t))$ to $0$ when $t\to \infty$. 

\subsection{Main result}

We denote by $L^p(\mathbb{R}^d)$ (or simply $L^p$) the classical Lebesgue space (of exponent $p$) with respect to the Lebesgue measure on $\mathbb{R}^d$. In the same way, we denote by $L^p_{\mu_m}(\mathbb{R}^d)$ (or simply $L^p_{\mu_m})$ for $p \in [1,+\infty]$ the Lebesgue space of exponent $p$ with respect to the measure ${\rm d}\mu_m$.
For instance, $r \in L^2_{\mu_m}(\mathbb{R}^d)$ if $r : \mathbb{R}^d \mapsto \mathbb{R}$ is measurable and $\int_{\mathbb{R}^d} r^2 \;{\rm d}\mu_m < \infty$. We may also use the norm $\|\cdot\|_{L^2_{\mu_m}(\mathbb{R}^d)}$ (or simply $\|\cdot\|_{L^2_{\mu_m}}$). Moreover, we extend these notations for Sobolev spaces (of integer exponant) with obvious definitions.

We denote by $\mathcal{D}(\mathbb{R}^d)$ the set of smooth and compactly supported functions from $\mathbb{R}^d$ to $\mathbb{R}$, and likewise for $\mathcal{D}((0,\infty)\times \mathbb{R}^d)$. The dual spaces are denoted by $\mathcal{D}'(\mathbb{R}^d)$ and $\mathcal{D}'((0,\infty)\times \mathbb{R}^d)$, respectively.

\vspace{1em}
Let us introduce the definition of weak solutions to the system \eqref{eq:CNSK_system_m_regularized_intro}. It corresponds to Definition 1.1 in \cite[page 2246]{Carles_Carrapatoso_Hillairet_AIF2022}, adapted to our context. The functions $\Lambda^0$ and $\Lambda$ appearing in the next definition are related to the initial data $(q^0,u^0)$ and a solution $(q,u)$ to \eqref{eq:CNSK_system_m_regularized_intro}, respectively, by the expressions $\Lambda =\sqrt{q}u$ and $\Lambda^0 = \sqrt{q^0}u^0$; to simplify the presentation, and following the convention used in previous works, we do not introduce a notation for $\sqrt{q}$ and $\sqrt{q^0}$, but in principle we should replace $\sqrt{q^0}$ and $\sqrt{q}$ by functions $r$ and $r^0$, respectively, everywhere in Definition \ref{def:WeakSolution_drag_term_intro}, and then define $q^0:=(r^0)^2$ and $q:=r^2$.

\begin{Definition}\label{def:WeakSolution_drag_term_intro}
	Let $(\sqrt{q^0},\Lambda^0) \in L^2_{\mu_m} \times [L^2_{\mu_m}]^d$. A pair $(q,u)$ is called a weak solution to \eqref{eq:CNSK_system_m_regularized_intro} associated to the initial data $(\sqrt{q^0},\Lambda^0)$ if there exists a quadruplet $(\sqrt{q},\Lambda,\mathbf{S}_{\rm K},\mathbf{T}_{\rm NS})$ such that 
	\begin{enumerate}
		\item The following regularities are satisfied
		\begin{equation*}
		\begin{array}{rclcrcl}
		\sqrt{q}   &\in&  L_{\rm loc}^\infty(0,\infty,L^2_{\mu_m}), &  & \nabla\sqrt{q}, \Lambda &\in&  L_{\rm loc}^\infty(0,\infty,[L^2_{\mu_m}]^d), \\
		\mathbf{T}_{\rm NS}, \;\mathbf{S}_{\rm K} &\in & L_{\rm loc}^2(0,\infty,[L^2_{\mu_m}]^{d\times d}), & & \sqrt{r_0}u & \in & L_{\rm loc}^2(0,\infty,[L^2_{\mu_m}]^d),\\
		r_1^{\frac14}q^{\frac14}u & \in& L^4_{\rm loc}(0,\infty;[L^4_{\mu_m}]^d), &  & 	r_0 (q-\ln(q)) &\in& L^\infty_{\rm loc}(0,\infty;L^1_{\mu_m}), \\
		r_4^{\frac14}q^{\frac14}x & \in& L^\infty_{\rm loc}(0,\infty;[L^4_{\mu_m}]^d), & & q\ln(q)&\in& L^\infty_{\rm loc}(0,\infty;L^1_{\mu_m}),\\
		\sqrt{r_4}D^2(\sqrt{q}) &\in& L_{\rm loc}^2(0,\infty,[L^2_{\mu_m}]^{d\times d}), & & r_4^{\frac14}\nabla(q^\frac14) &\in& L^4_{\rm loc}(0,\infty;[L^4_{\mu_m}]^d),
		\end{array}
		\end{equation*}
		and, for $\rho = q \rho_m$,
		\begin{equation*}
		\begin{array}{rlcrl}
		{D^2(\sqrt{\rho})} 
		&\in  L_{\rm loc}^2(0,\infty,[L^2]^{d\times d}), 
		& \quad &  \nabla \left(\rho^{\frac14}\right) &\in  L_{\rm loc}^4(0,\infty,[L^4]^d), 
		\end{array}
		\end{equation*}
		with the compatibility conditions
		\begin{equation*}
		\sqrt{q} \geq 0 \quad\mbox{a.e. on } (0,\infty)\times \mathbb{R}^d, \qquad \Lambda = 0 \quad\mbox{a.e. on } \left\{ \sqrt{q} = 0\right\}.
		\end{equation*}
		The velocity $u$ is then defined by 
		\begin{equation*}
			u = \frac{\Lambda}{\sqrt{q}}{\bf 1}_{\{\sqrt{q} >0\}}.
		\end{equation*}
		
		\item The following equations are satisfied in $\mathcal{D}'((0,\infty)\times \mathbb{R}^d)$, (see \eqref{eq:CNSK_system_m_regularized_intro}),
		\begin{subequations}\label{eq:CNSK_system_m_drag-term_WeakSolution}
			\begin{eqnarray}
				\partial_t \sqrt{q} + {\dv}_m (\sqrt{q}u) & = &\frac12 \tr(\mathbf{T}_{\rm NS}) -\frac1{2\sigma^2} \sqrt{q}u \cdot x, \label{eq:CNSK_system_m_drag-term_WeakSolution-density} \\
				\partial_t(\sqrt{q}\sqrt{q}u) + \dvm(\sqrt{q}u \otimes \sqrt{q}u)  & = & \dvm(2\nu \sqrt{q}\mathbf{S}_{\rm NS} + 2\kappa^2\sqrt{q}\mathbf{S}_{\rm K}) - \lambda\sigma^2\nabla\left(|\sqrt{q}|^2\right) \label{eq:CNSK_system_m_drag-term_WeakSolution-velocity} \\
				& & -r_0u - r_1q|u|^2u - \frac{r_4}{\sigma^4}q|x|^2x \notag
			\end{eqnarray}
		\end{subequations}
		with
		\begin{equation*}
		\mathbf{S}_{\rm NS} = \frac{\mathbf{T}_{\rm NS} + \mathbf{T}_{\rm NS}^\top}{2}
		\end{equation*}
		and the compatibility conditions
		\begin{eqnarray*}
			\sqrt{q}\mathbf{T}_{\rm NS} & = & \nabla(\sqrt{q}\sqrt{q}u) - 2\sqrt{q}u \otimes \nabla(\sqrt{q}), \\
			\sqrt{q}\mathbf{S}_{\rm K} & = & \sqrt{q}D^2(\sqrt{q})-  \nabla(\sqrt{q}) \otimes \nabla(\sqrt{q}) = \frac{1}{\rho_m}\left[\sqrt{\rho}D^2(\sqrt{\rho}) - \nabla(\sqrt{\rho}) \otimes \nabla (\sqrt{\rho}) + \frac\rho{2\sigma^2}{\rm I}_d\right].
		\end{eqnarray*}

		\item For any $\psi \in \mathcal{D}(\mathbb{R}^d)$,
		\begin{eqnarray*}
			\underset{t\to 0^+}{\lim} \int_{\mathbb{R}^d} \sqrt{q}(t,\cdot) \psi\,{\rm d}\mu_m & = & \int_{\mathbb{R}^d} \sqrt{q^0}\psi \,{\rm d}\mu_m, \\
			\underset{t\to 0^+}{\lim} \int_{\mathbb{R}^d} \sqrt{q}(t,\cdot) (\sqrt{q}u)(t,\cdot)\psi(x)\,{\rm d}\mu_m & = & \int_{\mathbb{R}^d} \sqrt{q^0}(\sqrt{q^0}u^0)\psi \,{\rm d}\mu_m.
		\end{eqnarray*}
	\end{enumerate}
\end{Definition}

\begin{Remark}{\label{Rmq:differente_ecriture_equation_masse}}
Note that, following the lines of the proof of Lemma 2.2 in \cite{Carles_Carrapatoso_Hillairet_AHL2018}, if $(\sqrt{q},\sqrt{q}u)$ is a weak solution in the sense of Definition \ref{def:WeakSolution_drag_term_intro}, then $q$ satisfies also \eqref{eq:CNSK_system_m_regularized-density_intro}.

Indeed, from the regularities of $\sqrt{q}u$ (which belongs to $L^\infty(0,\infty;[L^2_{\mu_m}]^d)$) and $\sqrt{q}$ (which belongs to $L^\infty(0,\infty;H^1_{\mu_m})$), multiplying equation \eqref{eq:CNSK_system_m_drag-term_WeakSolution-density} by $\sqrt{q}$, we get
\begin{equation*}
\partial_t q +2\sqrt{q}\dvm(\sqrt{q}u) = \sqrt{q}\tr(\mathbf{T_{\rm NS}}) - \frac{1}{\sigma^2}\sqrt{q}\sqrt{q}u \cdot x.
\end{equation*}
We then obtain from the definition of $\mathbf{T_{\rm NS}}$ on one hand and from the identity $\dvm(qu) = \sqrt{q}\dvm(\sqrt{q}u) + \sqrt{q}u\cdot \nabla \sqrt{q}$ on the other hand, that 
\begin{equation*}
\sqrt{q}\tr(\mathbf{T_{\rm NS}}) = - \dvm(qu) + 2\sqrt{q} \dvm(\sqrt{q}u)+\sqrt{q}\sqrt{q}u \cdot \frac{x}{\sigma^2}
\end{equation*}
and thus, after simplification, it is not difficult to see that $q$ satisfies \eqref{eq:CNSK_system_m_regularized-density_intro}.
\end{Remark}

\begin{Remark}
	In the previous definition, the regularities in 1. are dependent on the coefficients of the drag forces in the sense that, if one coefficient vanishes then the corresponding regularity is dropped.
\end{Remark}

\begin{Remark}
	Thanks to the Logarithmic Sobolev inequality \eqref{eq:LogarithmicSobolevInequality}, the fact that $q\ln(q) \in L^\infty_{\rm loc}(0,\infty;L^1_{\mu_m})$ is actually a consequence of  $\nabla{\sqrt{q}} \in L^\infty_{\rm loc}(0,\infty;[L^2_{\mu_m}]^d)$. 
\end{Remark}

We are now ready to state the main result of this paper.
\begin{Theorem}\label{thm:main}
	Let $\nu >0$, $\kappa >0$, and let us consider non-negative drag force coefficients $r_0 \geq 0$, $r_1 \geq 0$ and $r_4 \geq 0$.	
	Let $q^0\ge0$, $u^0$ be such that $(\sqrt{q^0},\Lambda^0) \in L^2_{\mu_m} \times [L^2_{\mu_m}]^d$ (with $\Lambda^0 = \sqrt{q^0}u^0$) 
	satisfying furthermore $\mathcal{E}(q^0,u^0) < \infty$ and $\mathcal{E}_{\rm BD}(q^0,u^0) < \infty$ and the compatibility condition 
	\begin{equation*}
		\sqrt{q^0} \geq 0 \quad \mbox{a.e. on } \mathbb{R}^d, \qquad \Lambda^0 = 0 \quad \mbox{a.e. on } \{\sqrt{q^0} = 0\}
	\end{equation*}
	and 
	\begin{equation*}
	\int_{\mathbb{R}^d} q^0 \,{\rm d}\mu_m = 1.
	\end{equation*}	
	
	Then, there exists a global weak solution $(q,u)$ to \eqref{eq:CNSK_system_m_regularized_intro} in the sense of Definition \ref{def:WeakSolution_drag_term_intro} associated to the initial data $(\sqrt{q^0},\sqrt{q^0}u^0)$ 
	which satisfies 
	\begin{enumerate}
		\item[-] the conservation of mass, that is, for all $t \geq 0$,
		\begin{equation*}
		\int_{\mathbb{R}^d}    q(t) \;{\rm d}\mu_m = 1.
		\end{equation*}
		\item[-] the energy estimate, for a.e. $t \geq 0$,
		\begin{equation*}
		\mathcal{E}(t) + \int_0^t \mathcal{D}(t') \,{\rm d}t' \leq C_0\left(\mathcal{E}(0)\right),
		\end{equation*}
		where the constant $C_0$ depends on the initial regularized energy \eqref{def:energy_reg_intro} and the time $t$.
		\item[-] the BD entropy estimate, for a.e. $t \geq 0$,
		\begin{equation*}
		\mathcal{E}_{\rm BD}(t) + \int_0^t \mathcal{D}_{\rm BD}(t') \,{\rm d}t' \leq C_{\rm BD, 0}\left(\mathcal{E}(0), \mathcal{E}_{\rm BD}(0)\right),
		\end{equation*}
		where the constant $C_{\rm BD, 0}$ depends on the initial regularized energy \eqref{def:energy_reg_intro}, the initial regularized BD entropy \eqref{def:BD_reg-entropie_intro} and the time $t$.
	\end{enumerate}
\end{Theorem}

	\begin{Remark}
		We can also consider system \eqref{eq:CNSK_system_m_regularized_intro} with $\kappa = 0$, that is the compressible (isothermal) Navier-Stokes equations by taking the limit $\kappa \to 0$ in the previous definition and theorem, see Remark \ref{rem:limit_kappa_0} for more details.
	\end{Remark}

The first (and main) step in the proof of Theorem \ref{thm:main} is the existence of solutions to the QNS system \eqref{eq:CNSK_system_m_regularized_intro} in the case where the coefficients $r_0$, $r_1$, $r_4$ in front of the drag forces   do not vanish and the initial data are regular enough, in a sense to be made precise below. As in previous works \cite{Carles_Carrapatoso_Hillairet_AIF2022, Feireisl_Book04,Vasseur_Yu_SIMA2016}, our strategy to establish this first step consists in introducing artificial smoothing terms into the system, depending on some parameter $\delta_1>0$, prove the existence of solutions to the new, regularized system, using in particular the Faedo-Galerkin method, and then show that the solutions survive in the limit $\delta_1\to0$. Here we emphasize that the regularization we choose is substantially simpler than that used in the previously cited references, namely we only add a smoothing diffusion term into the continuity equation (and its counterpart, for energy estimates, in the momentum equation). In \cite{Carles_Carrapatoso_Hillairet_AIF2022, Vasseur_Yu_SIMA2016}, several further terms are added, that are subsequently dealt with using, in particular, various Sobolev embeddings. In our context, where the reference measure is the Gaussian measure $\mu_m$, Sobolev embeddings are not available, so that we must rely on a simpler regularization procedure. We will show that the latter is indeed sufficient by refining some of the arguments of \cite{Carles_Carrapatoso_Hillairet_AIF2022, Feireisl_Book04,Jungel_SIMA10, Vasseur_Yu_SIMA2016}. The first step of the proof will be the content of Section \ref{sec:2}.

In a second step, in Section \ref{sec:3}, we will conclude the proof of Theorem \ref{thm:main} by showing that the solutions constructed in Section \ref{sec:2} survive for less regular initial data, and in the limit where some of the drag coefficients $r_0$, $r_1$, $r_4$ vanish. As in \cite[Section 4]{Carles_Carrapatoso_Hillairet_AIF2022}, the main ingredient here will be the notion of renormalized weak solutions introduced in \cite{Lacroix-Violet_Vasseur_JMPA18}.

As a direct consequence of Theorem \ref{thm:main} (and the strategy of its proof), we are able to recover Theorem 1.3 in~\cite{Carles_Carrapatoso_Hillairet_AIF2022} where, as mentioned above, the authors consider initially the system \eqref{eq:CNSK_system} without potential ($\lambda =0$) and then use a rescaling to consider a new system in new unknowns. In order to keep this introduction relatively short, we postpone to Section \ref{sect:CCH} a more detailed discussion, as well as a sketch of the proof allowing us to recover the result of \cite{Carles_Carrapatoso_Hillairet_AIF2022}, as we need to introduce many notations (the rescaling $\tau$, the new unknowns $(R,U)$ and the new system) to be able to write a precise statement.

We conclude with two appendices. Appendix \ref{appendixA} contains the details of some
calculations and estimates in the setting of the Gaussian reference measure ${\rm d}\mu_m$, that we use in several places in our proof. Finally, we gather in Appendix \ref{AppendixB} several functional inequalities in the setting of Gaussian measure spaces, including the logarithmic Sobolev inequality and the Poincar\'e inequalities. These inequalities play a crucial role in our proof, in that they allow us to overcome the fact that we cannot rely on the Sobolev embeddings usually employed in the euclidian setting.

\vspace{1em}
In the remainder of this section, we introduce some notations that will be used in the sequel.

\subsection{Notations}

We use the notation $\mathbb{N}$ for the set of all nonnegative integers and denote by $\mathbb{N}^* = \mathbb{N} \setminus\{0\}$ the set of integers greater than one. 

We define $Q_T = (0,T) \times \mathbb{R}^d$, for any $0 < T \leq \infty$ ($Q_\infty = (0,\infty) \times \mathbb{R}^d$).

Recall that $\mu_m$ is the measure on $\mathbb{R}^d$ with density $\rho_m$ given in \eqref{def:rho_m_u_m}.

Recall also that the space of infinitely differentiable functions with compact support in $\mathbb{R}^d$ is denoted $\mathcal{D}(\mathbb{R}^d)$, and likewise for $\mathcal{D}((0,\infty)\times \mathbb{R}^d)$. The dual spaces are denoted by $\mathcal{D}'(\mathbb{R}^d)$ and $\mathcal{D}'((0,\infty)\times \mathbb{R}^d)$, respectively.

For any normed vector spaces $X$, $Y$ (denoting $\|\cdot\|_X$, respectively $\|\cdot\|_Y$ their norm), we define $\mathcal{L}(X,Y)$ as the normed vector space of continuous linear mapping from $X$ to $Y$, endowed with the operator $\|\cdot\|_{\mathcal{L}(X,Y)}$ defined, for any $T \in \mathcal{L}(X,Y)$ by (and by the other equivalent definitions)
\begin{equation*}
\|T\|_{\mathcal{L}(X,Y)} = \sup_{x \in X ; \|x\|_X \leq 1} \|T(x)\|_Y.
\end{equation*}
In particular, we denote by $X' = \mathcal{L}(X,\mathbb{R})$ the dual space of $X$.

For $\Omega$ a measurable subset of $\mathbb{R}^m$ ($m \geq 1$), we define $\mathcal{M}(\Omega;\mathbb{R})$ the space of finite signed measures on $\Omega$. We consider the total variation as the norm on the space $\mathcal{M}(\Omega;\mathbb{R})$: for $\mathfrak{m} \in \mathcal{M}(\Omega;\mathbb{R})$, we define $\|\mathfrak{m}\|_{\mathcal{M}(\Omega;\mathbb{R})} = \displaystyle \sup\left\{ \int_{\Omega} f \,{\rm d}\mathfrak{m} \; ; \; f \in L^\infty(\Omega;\mathbb{R}) \mbox{ and } \|f\|_{L^\infty(\Omega;\mathbb{R})} \leq 1 \right\}$. In the same way, we define $\mathcal{M}(\Omega;\mathbb{R}^d)$ the space of vector-valued (in $\mathbb{R}^d$) measures on $\Omega$ endowed with $l^\infty$-product norm.

\paragraph{Acknowledgements.} We are grateful to K. Carrapatoso and A. Vasseur for useful discussions.

\section{Global weak solutions to the QNS system with drag forces and regular initial data}\label{sec:2}

In this section, we prove Theorem \ref{thm:main} in the case where the coefficients $r_0$, $r_1$, $r_4$ in front of the drag forces in the QNS system \eqref{eq:CNSK_system_m_regularized_intro} do not vanish and the initial data are regular in a sense to be made precise below.

First, in Subsection \ref{subsec:2.1}, we introduce a regularized version of the QNS system \eqref{eq:CNSK_system_m_regularized_intro} depending on a parameter $\delta_1>0$ by adding a diffusion term $\delta_1\Delta_mq$ in the right-hand side of \eqref{eq:CNSK_system_m_regularized-density_intro} and a corresponding term $\delta_1\nabla u \nabla q$ in the right-hand side of \eqref{eq:CNSK_system_m_regularized-velocity_intro}. The precise formulation of the regularized QNS system with drag forces is given in \eqref{eq:CNSK_system_m_regularized}.

Technically, the main result of this section is stated in Theorem \ref{prop:GWS_CNSK_system_m_regularized}, which provides the existence of global solutions for the regularized QNS system \eqref{eq:CNSK_system_m_regularized} for $\delta_1\ge0$ and regular enough initial data (see \eqref{eq:_regularized_initial_data}--\eqref{eq:initial_(q0,u0)} for the precise regularity conditions that we require). Subsections \ref{subsec:Faedo-Galerkin} and \ref{subsection:A priori Estimate} are devoted to the proof of Theorem \ref{prop:GWS_CNSK_system_m_regularized} in the case where $\delta_1>0$. We will closely follow the strategy presented in \cite{Feireisl_Book04}, emphasizing the novelties in the argument due to the setting considered in this paper.

In Subsection \ref{subsec:Faedo-Galerkin}, we introduce a suitable Faedo-Galerkin approximation  and construct global solutions $(q_N,u_N)$ of the regularized QNS system with drag forces \eqref{eq:CNSK_system_m_regularized} within this approximation. Here it should be noticed that, to simplify the notations, we only underline the dependence of the solutions $(q_N,u_N)$ on the Faedo-Galerkin parameter $N$, even if they of course also depend on the drag parameters $r_0$, $r_1$, $r_4$ and the regularizing parameter $\delta_1$. Next, in Subsection \ref{subsection:A priori Estimate}, we show that the Faedo-Galerkin solutions $(q_N,u_N)$ converge in the limit $N\to\infty$ to weak solutions $(q_{\delta_1},u_{\delta_1})$ to \eqref{eq:CNSK_system_m_regularized} with $0<\delta_1\le1$. Here again, to simplify the notations, we only underline the dependence in $\delta_1$ on the solutions $(q_{\delta_1},u_{\delta_1})$ obtained in the limit $N\to\infty$, even if they still depend on $r_0$, $r_1$, $r_4$. Finally, in Subsection \ref{subsect:GWS_drag_term}, we derive a priori estimates for $(q_{\delta_1},u_{\delta_1})$ from the energy and BD entropy estimates, and use them to show that the existence of global solutions persists in the limit $\delta_1\to0$ which completes the proof of Theorem \ref{prop:GWS_CNSK_system_m_regularized}. This leads to a proof of our main result, Theorem \ref{thm:main}, in the case where the drag forces coefficients $r_0$, $r_1$ and $r_4$ do not vanish and the initial data satisfy the regularity conditions \eqref{eq:_regularized_initial_data}--\eqref{eq:initial_(q0,u0)}.

\subsection{The regularized QNS system with drag forces}\label{subsec:2.1}

According to the first step of the strategy outlined in the introduction to this section, we add a diffusion term $\delta_1\Delta_mq$ in the right-hand side of the continuity equation, and a corresponding term (for energy estimates) $\delta_1\nabla u \nabla q$ in the left-hand side of the momentum equation. This leads to considering the following regularized QNS system (with $r_0 >0$, $r_1 >0$, $r_4 >0$ and $\delta_1 \in [0,1]$):
\begin{subequations}
	\label{eq:CNSK_system_m_regularized}
	\begin{eqnarray}
	\partial_t q + \dvm(qu) &=& \delta_1\Delta_mq,      \label{eq:CNSK_system_m_regularized-density} \\
	\partial_t (qu) + \dvm(qu\otimes u) +\delta_1 \nabla u \nabla q &=&  \displaystyle \dvm\left(2\nu q D(u)\right) + \dvm\left(\kappa^2qD^2(\ln(q))\right) -\lambda \sigma^2 \nabla q \label{eq:CNSK_system_m_regularized-velocity} \\
	& & \qquad  - r_0 u -r_1q |u|^2u- \frac{r_4}{\sigma^4} q |x|^2x.  \notag
	\end{eqnarray}
\end{subequations}
As mentioned above, this regularized QNS system should be compared to the corresponding ones in \cite{Carles_Carrapatoso_Hillairet_AIF2022, Feireisl_Book04,Vasseur_Yu_SIMA2016}, where various further terms are added in the right-hand side of the momentum equation \eqref{eq:CNSK_system_velocity}, such as $\nabla(\rho^{-s})$, $\Delta^2 u$, $\rho \nabla \Delta^{2m+1}\rho$, where $s \in (0,\infty)$, $m \in \mathbb{N}$ are large enough parameters.

The formal energy identity \eqref{eq:energy_identity_reg_intro}--\eqref{def:dissipation_reg_intro} now becomes 
\begin{equation}\label{eq:energy_identity_reg}
\frac{{\rm d}}{{\rm d}t} \mathcal{E}_{\rm reg}(t) + \mathcal{D}_{\rm reg}(t) = \mathcal{R}_{\rm reg}(t),
\end{equation}
where the regularized energy $\mathcal{E}_{\rm reg}$ is defined by
\begin{equation}\label{def:energy_reg}
\mathcal{E}_{\rm reg} = \mathcal{E}, 
\end{equation}
(see \eqref{def:energy_reg_intro}), the regularized dissipation $\mathcal{D}_{\rm reg}$ is
\begin{equation}\label{def:dissipation_reg}
\begin{array}{rcl} 
\mathcal{D}_{\rm reg} & = & \displaystyle 2\nu \int_{\mathbb{R}^d} q|D(u)|^2 \;{\rm d}\mu_m + \delta_1\lambda \sigma^2 \int_{\mathbb{R}^d} q |\nabla \ln(q)|^2  \;{\rm d}\mu_m  + \kappa^2\delta_1 \int_{\mathbb{R}^d} q |D^2(\ln(q))|^2  \;{\rm d}\mu_m \\
& & \displaystyle +r_0 \int_{\mathbb{R}^d} |u|^2 \;{\rm d}\mu_m+ r_1\int_{\mathbb{R}^d} q|u|^4 \;{\rm d}\mu_m + \frac{r_4\delta_1}{4\sigma^2}I_4(q),
\end{array} 
\end{equation}
and the right-hand side $\mathcal{R}_{\rm reg}$ in \eqref{eq:energy_identity_reg} is given by
\begin{equation*}
\mathcal{R}_{\rm reg} = \frac{r_4\delta_1(d+2)}{\sigma^2}I_2(q).
\end{equation*}
We recall that the second and fourth order momentum $I_2(q)$, $I_4(q)$ have been defined in \eqref{def:I2} and \eqref{def:I4}, respectively.
Here it should be noticed that the regularizing terms $\delta_1\Delta_mq$, $\delta_1\nabla u \nabla q$ induce the non-negative term $\mathcal{R}_{\rm reg}$. The latter can however be estimated, thanks to the Young inequality, as
\begin{equation}\label{eq:absorb_Rreg}
\mathcal{R}_{\rm reg} \leq \frac{r_4\delta_1}{8\sigma^2}I_4(q) +\frac{2r_4\delta_1(d+2)^2}{\sigma^2},
\end{equation}
(recall that $m_F=\int_{\mathbb{R}^d}q\mathrm{d}\mu_m=1$ throughout the paper). The first term on the right-hand side can be `absorbed' by the dissipation $\mathcal{D}_{\rm reg}$ and we therefore see that the regularized energy should grow at most linearly in time.

We similarly introduce the regularized BD entropy and dissipation by setting 
\begin{equation}\label{def:BD_reg-energy} 
\mathcal{E}_{\rm BD, reg}=\mathcal{E}_{\rm BD}
\end{equation}
(see \eqref{def:BD_reg-entropie_intro}) and (compare with \eqref{def:BD_reg-dissipation_intro}),
\begin{align}
\mathcal{D}_{\rm BD, reg} & =  2\nu \int_{\mathbb{R}^d} q|A(u)|^2 \;{\rm d}\mu_m + (\delta_1+2\nu)\lambda \sigma^2 \int_{\mathbb{R}^d} q|\nabla \ln(q)|^2  \;{\rm d}\mu_m  \notag \\
&\quad + [\kappa^2(\delta_1+2\nu)+(2\nu)^2\delta_1] \int_{\mathbb{R}^d} q |D^2(\ln(q))|^2  \;{\rm d}\mu_m  +r_0 \int_{\mathbb{R}^d} |u|^2 \;{\rm d}\mu_m  \notag \\
& \quad  + 2\nu r_0 \delta_1 \int_{\mathbb{R}^d} |\nabla \ln(q)|^2 \;{\rm d}\mu_m + r_1\int_{\mathbb{R}^d} q|u|^4 {\rm d}\mu_m + \frac{r_4(\delta_1+2\nu)}{\sigma^2}I_4(q). \label{def:BD_reg-dissipation} 
\end{align}
We will see that the following formal identity is satisfied in a suitable sense:
\begin{equation}\label{eq:BD-entropy_equality-regularized}
\frac{{\rm d}}{{\rm d}t} \mathcal{E}_{\rm BD, reg}(t) +\mathcal{D}_{\rm BD, reg} (t) = \mathcal{R}_{\rm BD, reg}(t),
\end{equation}
where the remainder term, $\mathcal{R}_{\rm BD, reg}$, that will be computed below (see \eqref{def:BD_reg-RHS}), satisfies an inequality similar to \eqref{eq:absorb_Rreg} for $\mathcal{R}_{\rm reg}$.	

Recall the notation $Q_\infty = (0,\infty) \times \mathbb{R}^d$. Our definition of global weak solutions to the regularized system \eqref{eq:CNSK_system_m_regularized} is as follows.

\begin{Definition}[Global weak solutions to the regularized QNS system with drag forces] \label{def:WeakSolutionRegularizedEquations}
	Let $r_0 >0$, $r_1 >0$, $r_4 >0$ and $\delta_1 \ge0$. Let $(q^0,u^0) \in L^1_{\mu_m} \times [L^2]^d$. A couple $(q,u)$ is a global weak solution to \eqref{eq:CNSK_system_m_regularized} associated to the initial data $(q^0,u^0)$ if
	\begin{enumerate}
		\item The couple $(q,u)$ satisfies 
		\begin{align*}
		q, \, q \ln(q), \, r_0(q-\ln(q)) \ & \  \in L_{\rm loc}^\infty(0,\infty;L^1_{\mu_m}), \\
		\sqrt{q}u, \,  \nabla\sqrt{q} \ & \ \in  L_{\rm loc}^\infty(0,\infty;[L^2_{\mu_m}]^d), \\
		\sqrt{r_0} u, \,  \sqrt{r_0\delta_1}\nabla\ln(q) \ & \ \in  L_{\rm loc}^2(0,\infty;[L^2_{\mu_m}]^d), \\
		\sqrt{q} D(u), \, \sqrt{q} A(u) , \,    \sqrt{q} D^2(\ln(q)) , \, \sqrt{r_4} D^2(\sqrt{q}) \ & \ \in L_{\rm loc}^2(0,\infty;[L^2_{\mu_m}]^{d\times d}) , \\
		r_4^{\frac14}  q^{\frac14}x \ & \ \in L_{\rm loc}^\infty(0,\infty;[L^4_{\mu_m}]^d) , \\
		r_1^{\frac14}q^{\frac14}u, \, r_4^{\frac14} \nabla{q^{\frac14}} \ & \ \in L_{\rm loc}^4(0,\infty;[L^4_{\mu_m}]^d).
		\end{align*}
		\item For every $\varphi\in \mathcal{D}([0,\infty)\times \mathbb{R}^d)$,
		\begin{align}
		\int_{\mathbb{R}^d}q^0\varphi(0)\,{\rm d}\mu_m+\int_{Q_\infty}q\partial_t\varphi\,{\rm d}\mu_m{\rm d}t+\int_{Q_\infty}qu\cdot\nabla\varphi\,{\rm d}\mu_m{\rm d}t=\delta_1\int_{Q_\infty}(\Delta_mq)\varphi\,{\rm d}\mu_m{\rm d}t.
		\end{align}
		\item For every $\Phi \in [\mathcal{D}([0,\infty)\times \mathbb{R}^d)]^d$, 
		\begin{align}
		&  \int_{\mathbb{R}^d} q^0u^0 \cdot \Phi(0) {\rm d}\mu_m +  \int_{Q_\infty} q u \cdot \partial_t\Phi {\rm d}\mu_m{\rm d}t  + \int_{Q_\infty} q\nabla \Phi u \cdot u {\rm d}\mu_m{\rm d}t -\delta_1 \int_{Q_\infty} \nabla u \nabla q \cdot \Phi {\rm d}\mu_m{\rm d}t \notag \\
		&= 2\nu \int_{Q_\infty} q D(u) : D(\Phi) {\rm d}\mu_m{\rm d}t + 2\kappa^2 \int_{Q_\infty} \left[\sqrt{q} D^2(\sqrt{q})-\nabla(\sqrt{q}) \otimes \nabla(\sqrt{q})\right] : D(\Phi) {\rm d}\mu_m {\rm d}t \notag\\
		&\quad  +\lambda \sigma^2 \int_{Q_\infty} \nabla q  \cdot \Phi {\rm d}\mu_m{\rm d}t + r_0 \int_{Q_\infty} u \cdot \Phi  {\rm d}\mu_m{\rm d}t+ r_1\int_{Q_\infty} q|u|^2u \cdot \Phi  {\rm d}\mu_m{\rm d}t \notag\\
		& \quad +\frac{r_4}{\sigma^4}\int_{Q_\infty} q|x|^2x \cdot \Phi  {\rm d}\mu_m{\rm d}t. \label{eq:CNSK_system_m_regularized-velocity-WS-Definition}
		\end{align}
	\end{enumerate}
\end{Definition}

\begin{Remark}\label{rem:definitions2.1and1.2} 	
	Note that, when $\delta_1 = 0$, the terms $\sqrt{q} D(u), \, \sqrt{q} A(u)$ correspond to the symetric and skew-adjoint parts respectively of $\mathbf{T}_{\rm NS} \in L^2_{\rm loc}(0,\infty;[L^2_{\mu_m}]^{d\times d})$ while $\sqrt{q} D^2(\ln(q))$ has to be understood as $\mathbf{S}_{\rm K} \in L^2_{\rm loc}(0,\infty;[L^2_{\mu_m}]^{d\times d})$, both satisfying the compatibility conditions 
		\begin{eqnarray*}
			\sqrt{q}\mathbf{T}_{\rm NS} & = & \nabla(\sqrt{q}\sqrt{q}u) - 2\sqrt{q}u \otimes \nabla(\sqrt{q}), \\
			\sqrt{q}\mathbf{S}_{\rm K} & = & \sqrt{q}D^2(\sqrt{q})-  \nabla(\sqrt{q}) \otimes \nabla(\sqrt{q}) = \frac{1}{\rho_m}\left[\sqrt{\rho}D^2(\sqrt{\rho}) - \nabla(\sqrt{\rho}) \otimes \nabla (\sqrt{\rho}) + \frac\rho{2\sigma^2}{\rm I}_d\right].
		\end{eqnarray*}
	In that case, Definition \ref{def:WeakSolutionRegularizedEquations} extends Definition \ref{def:WeakSolution_drag_term_intro}. Indeed, when $\delta_1 = 0$, solutions in the sense of Defintion \ref{def:WeakSolutionRegularizedEquations} satisfy the same regularities as in the first point of Definition \ref{def:WeakSolution_drag_term_intro}. Moreover, points 2 and 3 of Definition \ref{def:WeakSolution_drag_term_intro} follow from points 2 and 3 of Definition \ref{def:WeakSolutionRegularizedEquations} (thanks to Remark \ref{Rmq:differente_ecriture_equation_masse} for equation \eqref{eq:CNSK_system_m_drag-term_WeakSolution-density}).

\end{Remark}

The main result of this section, Theorem \ref{prop:GWS_CNSK_system_m_regularized}, establishes the existence of global weak solutions to the regularized QNS system \eqref{eq:CNSK_system_m_regularized}. We consider here an initial data $(q^0,u^0)$ satisfying, for some fixed constant $c^0 >0$ 
\begin{equation}\label{eq:_regularized_initial_data}
q^0 \in H^1_{\mu_m}, \quad c^0 \le q^0 \le \frac{1}{c^0} \mbox{ on } \mathbb{R}^d, \quad \int_{\mathbb{R}^d} q^0 {\rm d}\mu_m = 1, \quad  I_4(q^0 ) = \int_{\mathbb{R}^d} \left|\frac{x}{\sigma}\right|^4 q^0 {\rm d}\mu_m < \infty,
\end{equation}
and
\begin{equation}\label{eq:initial_(q0,u0)}
u^0 \in [L^2]^d, \qquad \rho_m q^0 |u^0|^2 \in L^1, \qquad \rho_mq^0u^0 \in [L^2]^d .
\end{equation}
Moreover, to simplify some estimates in the proof, we suppose that all the regularizing parameters $r_0$, $r_1$, $r_4$ and $\delta_1$ are less than $1$.

\begin{Theorem}[Global weak solutions to the regularized QNS system with drag forces]\label{prop:GWS_CNSK_system_m_regularized}
	Let $0<r_0, r_1, r_4 \le1$ and $0\le\delta_1\le1$. Let $(q^0,u^0) \in L^1_{\mu_m} \times [L^2]^d$ satisfying \eqref{eq:_regularized_initial_data}--\eqref{eq:initial_(q0,u0)}. There exists a global weak solution $(q,u)$ to \eqref{eq:CNSK_system_m_regularized} (in the sense of Definition \ref{def:WeakSolutionRegularizedEquations}), associated to the initial data $(q^0,u^0)$, which satisfies the conservation of mass, for all $t \geq 0$,
	\begin{equation}\label{eq:conservation_mass_thm_regularized}
	\int_{\mathbb{R}^d}    q(t) \;{\rm d}\mu_m = 1
	\end{equation}
	the energy estimate, for a.e. $t \geq 0$,
	\begin{equation}\label{eq:energy_estimates_GWS_CNSK_system_m_regularized}
	\mathcal{E}_{\rm reg}(t) + \frac12 \int_0^t \mathcal{D}_{\rm reg}(t') \,{\rm d}t' \leq C_0\left(\mathcal{E}_{\rm reg}(0),t\right),
	\end{equation}
	where the `constant' $C_0\left(\mathcal{E}_{\rm reg}(0),t\right)$ depends on the initial regularized energy  and the time $t$, and the BD entropy estimate, for a.e. $t \geq 0$,
	\begin{equation}\label{eq:BD_estimates_GWS_CNSK_system_m_regularized}
	\mathcal{E}_{\rm BD,reg}(t) + \frac12 \int_0^t \mathcal{D}_{\rm BD,reg}(t') \,{\rm d}t' \leq C_{\rm BD,0}\left(\mathcal{E}_{\rm reg}(0),\mathcal{E}_{\rm BD,reg}(0),t\right),
	\end{equation}
	where the `constant' $C_{\rm BD,0}\left(\mathcal{E}_{\rm reg}(0),\mathcal{E}_{\rm BD,reg}(0),t\right)$ depends on the initial regularized energy, the initial regularized BD entropy, and the time $t$.
\end{Theorem}

\begin{Remark}\label{rk:1/2}
	The numerical factor $\frac12$ in front of the dissipation term in \eqref{eq:energy_estimates_GWS_CNSK_system_m_regularized} comes from the estimate of $\mathcal{R}_{\rm reg}$, see \eqref{eq:absorb_Rreg}, and likewise for \eqref{eq:BD_estimates_GWS_CNSK_system_m_regularized}.
\end{Remark}

\begin{Remark}
	Using the regularity properties stated in the previous remark, one can verify that Equation \eqref{eq:CNSK_system_m_regularized-density} is in fact satisfied in $L^1_{\rm loc}(0,\infty;L^1_{\mu_m})$ (in the sense that each term appearing in \eqref{eq:CNSK_system_m_regularized-density} belongs to $L^1_{\rm loc}(0,\infty;L^1_{\mu_m})$). See Proposition \ref{prop:FG-convergence} below for more details. Likewise, Equation \eqref{eq:CNSK_system_m_regularized-velocity} is satisfied in $L^2_{\mathrm{loc}}(0,\infty;[(W^{1,\infty}_{\mu_m})']^d)$. 
\end{Remark}

\subsection{Global solutions in the Faedo-Galerkin approximation}\label{subsec:Faedo-Galerkin}

\subsubsection{The Faedo-Galerkin approximation.}

We first introduce a Faedo-Galerkin approximation for the velocity. As usual, we consider a finite-dimensional subspace $\XN$ of the velocity Hilbert space $[L^2]^d$. In our context, where the Hilbert space for the relative density is $L^2_{\mu_m}=L^2(\mathbb{R}^d,{\rm d\mu_m})$, with ${\rm d}\mu_m=\rho_m{\rm d}x$ a Gaussian measure, it is natural to choose $\XN$ as the cartesian product of the vector space spanned by Hermite functions (associated to the Gaussian function $\rho_m$). More precisely, we consider $X_N$ the subspace of $L^2$ spanned by all the Hermite functions of degree less or equal to $N$, 
\begin{equation*}
X_N = \left\{ f : \mathbb{R}^d \to \mathbb{R} \; ; \; \exists P \in \mathcal{H}_N \mbox{ such that } f = P \rho_m^{\frac12}  \right\}, 
\end{equation*}
equipped with the $L^2$-norm, where $\mathcal{H}$ is the vector space spanned by the Hermite polynomials on $\mathbb{R}^d$ and $\mathcal{H}_N = \left\{ P \in \mathcal{H} \; \mbox{s.t.} \; \deg(P) \leq N \right\}$. We then set
\begin{equation*}
\XN=[X_N]^d.
\end{equation*}

\begin{Remark}\label{rk:elementary}$ $ For later purpose, we record here the following elementary properties:
	\begin{enumerate}[label=(\roman*)]
		\item $X_N \subset L^\infty$ continuously. 
		\item For all $\alpha \in \mathbb{N}^d$, the operator $\partial^\alpha : X_N \mapsto X_{N+|\alpha|}$ is bounded.
		\item For all $\beta \in \mathbb{N}^d$, the operator $x^\beta : X_N \mapsto X_{N+|\beta|}$ is bounded.
	\end{enumerate}
\end{Remark}
Remark \ref{rk:elementary} in turn implies that $\dvm : \XN \mapsto X_{N+1}$ is bounded: There exists a constant $C_m(N) >0$ such that, for all $u_N \in \XN$,
\begin{equation}\label{eq:bound_div_UN}
\|\dvm(u_N)\|_{L^\infty} \leq C_m(N) \|u_N\|_{[L^2]^d}.
\end{equation}

For any $N\geq 1$, we introduce the initial data for the approximate velocity, $u_N^0$, as follows. Let $\Pi_N$ be the orthogonal projection onto the vector space spanned by $\{\rho_mq^0w_N \, \text{ s.t. } w_N \in X_N\}$ in $[L^2]^d$.  We then set $u^0_N:=\Pi_N(u^0)$. Equivalently, $u^0_N$ is the unique element of $\XN$ such that
\begin{equation}\label{def:XN_initial_data}
\int_{\mathbb{R}^d} q^0 u_N^0 \cdot w_N \;{\rm d}\mu_m = \int_{\mathbb{R}^d} q^0 u^0 \cdot w_N    \;{\rm d}\mu_m \qquad \mbox{for all } w_N \in \XN.
\end{equation}

Our goal in this subsection is to show that, for any initial data $(q^0,u^0_N)$ with $q^0$ satisfying \eqref{eq:_regularized_initial_data} and $u^0_N\in\XN$, for any $T>0$, the regularized system \eqref{eq:CNSK_system_m_regularized} has a global solution $(q_N,u_N)$ with $q_N$ in a suitable space (see Definition \ref{def:function_space_qN} below) and $u_N \in \mathcal{C}^1([0,T];\XN)$.

Recall that the regularized energy $\mathcal{E}_{\rm reg}$ and dissipation $\mathcal{D}_{\rm reg}$ have been introduced in \eqref{def:energy_reg} and \eqref{def:dissipation_reg}, respectively. Recall also that the positive parameters $r_0,r_1,r_4,\delta_1$ are supposed to be less than $1$ to simplify. In this subsection we will prove the following proposition.
\begin{Proposition}[Solutions to the regularized QNS system with drag forces in the Faedo-Galerkin approximation]\label{prop:existence_FD}
	Let $0<r_0,r_1,r_4,\delta_1 \le1$. Let $N\in\mathbb{N}^*$, $T>0$ and $(q^0,u^0_N)\in L^2_{\mu_m}\times\XN$ with $q^0$ satisfying \eqref{eq:_regularized_initial_data}. There exists a unique global weak solution 
	\begin{equation}
	(q_N,u_N)\in [L^2(0,T;H^2_{\mu_m}) \cap  \mathcal{C}^0([0,T];H^1_{\mu_m}) \cap H^1(0,T;L^2_{\mu_m})]\times \mathcal{C}^1([0,T];\XN)
	\end{equation}
	to \eqref{eq:CNSK_system_m_regularized} such that $(q_N(0),u_N(0))=(q^0,u^0_N)$, in the sense that
	\begin{equation}\label{eq:Fokker-Planck-thm}
	\partial_t q_N + \dvm(q_Nu_N) = \delta_1 \Delta_m q_N  \quad \mbox{ in } L^2(0,T;L^2_{\mu_m}),
	\end{equation}
	and for all $\Phi_N \in \mathcal{C}^1([0,T];\XN)$ such that $\Phi_N(T) = 0$, (recall that $Q_T=(0,T)\times\mathbb{R}^d$),
	\begin{align}
	& \int_{\mathbb{R}^d} q^0u_N^0 \cdot \Phi_N(0) {\rm d}\mu_m +  \int_{Q_T} q_N u_N \cdot \partial_t\Phi_N {\rm d}\mu_m{\rm d}t  + \int_{Q_T} q_N\nabla \Phi_N u_N \cdot u_N {\rm d}\mu_m{\rm d}t \notag\\
	& -\delta_1 \int_{Q_T} \nabla u_N \nabla q_N \cdot \Phi_N {\rm d}\mu_m{\rm d}t \notag\\
	=\, & 2\nu \int_{Q_T} q_N D(u_N) : D(\Phi_N) {\rm d}\mu_m{\rm d}t + 2\kappa^2 \int_{Q_T} \left[\sqrt{q_N} D^2(\sqrt{q_N})-\nabla(\sqrt{q_N}) \otimes \nabla(\sqrt{q_N})\right] : D(\Phi_N) {\rm d}\mu_m {\rm d}t \notag\\
	&  +\lambda \sigma^2 \int_{Q_T} \nabla q_N  \cdot \Phi_N {\rm d}\mu_m{\rm d}t + r_0 \int_{Q_T} u_N \cdot \Phi_N  {\rm d}\mu_m{\rm d}t+ r_1\int_{Q_T} q_N|u_N|^2u_N \cdot \Phi_N  {\rm d}\mu_m{\rm d}t \notag\\
	&+ \frac{r_4}{\sigma^4}\int_{Q_T} q_N|x|^2x \cdot \Phi_N  {\rm d}\mu_m{\rm d}t. \label{eq:CNSK_system_m_regularized-velocity-WS-Definition_XN-thm}
	\end{align}
	The solution $(q_N,u_N)$ satisfies, for all $0\le t\le T$, the inequalities
	\begin{equation}\label{eq:positivity_qN-prop}
	c^0 \exp\left(-\int_0^t \|\dvm(u_N(\tau))\|_{L^\infty}\,{\rm d}\tau\right)\le q_N(t)\le \frac{1}{c^0} \exp\left(\int_0^t \|\dvm(u_N(\tau))\|_{L^\infty}\,{\rm d}\tau\right),
	\end{equation}
	for some constant $c^0>0$, the conservation of mass 
	\begin{equation}\label{eq:conservation-mass-qN}
	\int_{\mathbb{R}^d} q_N(t) \;{\rm d}\mu_m = 1
	\end{equation}
	and the energy estimate, 
	\begin{equation*}
	\mathcal{E}_{\rm reg}(q_N(t),u_N(t)) + \frac12 \int_0^t \mathcal{D}_{\rm reg}(t') \,{\rm d}t' \leq C_0\left(\mathcal{E}_{\rm reg}(q_N(0),u_N(0)),t\right),
	\end{equation*}
	where the `constant' $C_0$ depends on the initial regularized energy  and on the time $t$.
\end{Proposition}

\begin{Remark}
	As in Remark \ref{rk:1/2}, the numerical factor $\frac12$ in front of the dissipation comes from the estimate of $\mathcal{R}_{\rm reg}$, see \eqref{eq:absorb_Rreg}.
\end{Remark}

In order to establish Proposition \ref{prop:existence_FD}, we follow the Faedo-Galerkin approach. First we assume that $u_N \in \mathcal{C}^1([0,T];\XN)$ is fixed and show that the Fokker-Planck equation \eqref{eq:CNSK_system_m_regularized-density} has a unique solution $q_N=\mathcal{S}_T^{(q)}(u_N)$ (see Remark \ref{rk:def_STq} for the precise definition of $\mathcal{S}_T^{(q)}$). Next we prove that a solution $u_N\in \mathcal{C}^1([0,T];\XN)$ to  \eqref{eq:CNSK_system_m_regularized-velocity} (with $q$ given by $\mathcal{S}_T^{(q)}(u_N)$) necessarily satisfies an ODE which can be solved by a Banach fixed point argument. This allows us to construct a (local in time) solution to the system \eqref{eq:CNSK_system_m_regularized} in the Faedo-Galerkin approximation. The global existence (on any time interval $[0,T]$, with $T>0$) finally follows from a priori, energy estimates.

Throughout the remainder of this subsection, we assume that $0<r_0,r_1 ,r_4 ,\delta_1 \le1$ are fixed.

\subsubsection{The Fokker-Planck equation.} 

Assuming that an integer $N\ge1$, a positive time $T>0$ and a vector field $u_N \in \mathcal{C}^1([0,T];\XN)$ such that $u_N(0) = u_N^0$ (for a given initial velocity $u_N^0 \in \XN$) are fixed, we wish  to construct a density $q_N$ associated to $u_N$ as a solution to the Fokker-Planck equation
\begin{equation}\label{eq:Fokker-Planck}
\partial_t q + \dvm(qu_N) = \delta_1 \Delta_m q  \quad \mbox{ in } (0,T) \times \mathbb{R}^d ,
\end{equation}
associated to an initial data $q^0$ satisfying \eqref{eq:_regularized_initial_data}. We introduce the following function space.
\begin{Definition}\label{def:function_space_qN}
	Let $T>0$.  We define the space $H^{2,1}_{\mu_m}(Q_T)$ as
	\begin{equation*}
	H^{2,1}_{\mu_m}(Q_T) = L^2(0,T;H^2_{\mu_m}) \cap  \mathcal{C}^0([0,T];H^1_{\mu_m}) \cap H^1(0,T;L^2_{\mu_m}),
	\end{equation*}
	equipped with the usual norm
	\begin{equation*}
	\|q\|_{H^{2,1}_{\mu_m}(Q_T)} = \|q\|_{L^2(0,T;H^2_{\mu_m})} + \|q\|_{\mathcal{C}^0([0,T];H^1_{\mu_m})} + \|\partial_tq\|_{L^2(0,T;L^2_{\mu_m})}.
	\end{equation*}
\end{Definition}

In the next proposition, we only focus on the existence of \emph{local} solutions, for simplicity of presentation and since it is sufficient for our purpose. It would however not be difficult to extend the result to the existence of global solutions. 

\begin{Proposition}[Local solutions to the Fokker-Planck equation]
	\label{prop:Fokker-PlanckEquation_existence}
	Let $N\in\mathbb{N}^*$, and $(q^0,u^0_N)\in L^2_{\mu_m}\times\XN$ with $q^0$ satisfying \eqref{eq:_regularized_initial_data}. There exists $T>0$ such that the following holds.
	\begin{enumerate}[label=(\roman*)]
		\item Let $u_N \in \mathcal{C}^1([0,T];\XN)$ be such that $u_N(0) = u_N^0$. The Fokker-Planck Equation \eqref{eq:Fokker-Planck} admits a unique solution $q_N$ in $H^{2,1}_{\mu_m}(Q_{T})$ such that $q(0)=q^0$, and there is a constant $C(\|u_N\|_{L^\infty(0,T;\XN)}) >0$ such that 
		\begin{equation}\label{eq:regularity_qN}
		\|q_N\|_{H^{2,1}_{\mu_m}(Q_T)} \leq C(\|u_N\|_{L^\infty(0,T;\XN)}) \|q^0\|_{H^1_{\mu_m}}.	
		\end{equation}
		Moreover, with $c^0>0$ as in \eqref{eq:_regularized_initial_data}, we have
		\begin{equation}\label{eq:positivity_qN}
		c^0 \exp\left(-\int_0^t \|\dvm(u_N(\tau))\|_{L^\infty}\,{\rm d}\tau\right) \le q_N(t)\le \frac{1}{c^0} \exp\left(\int_0^t \|\dvm(u_N(\tau))\|_{L^\infty}\,{\rm d}\tau\right),
		\end{equation}
		for all $t\in[0,T]$.

		\item Let $u_N^k \in \mathcal{C}^1([0,T];\XN)$, $k=1, 2$, be such that $u_N^k(0) = u_N^0$. 
		Let $q_N^k\in H^{2,1}_{\mu_m}(Q_T)$ be the solution to \eqref{eq:Fokker-Planck} associated to $u_N^k$ and the initial data $q^0$. There exists a positive constant $$C(\underset{k=1,2}{\max}\|u_N^k\|_{L^\infty(0,T;\XN)},\|q^0\|_{H^1_{\mu_m}})$$ such that 
		\begin{equation}\label{eq:Lipschitz_qN}
		\| q_N^1 - q_N^2 \|_{H^{2,1}_{\mu_m}(Q_T)} \leq C(\underset{k=1,2}{\max}\|u_N^k\|_{L^\infty(0,T;\XN)},\|q^0\|_{H^1_{\mu_m}})\|u^1_N-u^2_N\|_{L^\infty(0,T;\XN)}. 
		\end{equation}
	\end{enumerate}
\end{Proposition}

\begin{proof}
	\begin{enumerate}[label=(\roman*)]
		\item First, it follows from \cite[Theorem 4.1]{Lunardi_TAMS97} that the operator $\delta_1\Delta_m$ with domain $H^2_{\mu_m}$ is self-adjoint in $L^2_{\mu_m}$ and that the usual parabolic regularity holds for the associated evolution equation, in the sense that the unique solution to 
		\begin{equation*}
		\partial_t q = \delta_1 \Delta_m q  \quad \mbox{ in } (0,T) \times \mathbb{R}^d, \qquad q(0) = q^0 ,
		\end{equation*}
		is given by $(t,x)\mapsto(e^{t\delta_1\Delta_m}q^0)(x)$ and satisfies
		\begin{equation}\label{eq:estim_Lunardi}
		\|e^{t\delta_1\Delta_m}q^0\|_{H^{2,1}_{\mu_m}(Q_T)} \leq C \|q^0\|_{H^1_{\mu_m}},
		\end{equation}
		for some positive constant $C$ depending only on $\sigma$, $\delta_1$ and the dimension $d$.
		
		Next we consider the integral equation
		\begin{equation}\label{eq:integral_qN}
		q_N(t,x)=e^{t\delta_1\Delta_m}q^0(x)-\int_0^te^{(t-s)\delta_1\Delta_m}\dvm(q_N(s,x)u_N(s,x)){\rm d}s=:\mathcal{T}(q_N)(t,x).
		\end{equation}
		Using $\|e^{t\Delta_m}f\|_{H^1_{\mu_m}}\le Ct^{-\frac12}\|f\|_{L^2_{\mu_m}}$ (which in turn follows from $\|f\|_{H^1_{\mu_m}}\le C\|(1-\Delta_m)^{\frac12}f\|_{L^2_{\mu_m}}$) together with \eqref{eq:bound_div_UN}, one easily deduces, in particular, that 
		\begin{equation}\label{eq:TqN1-2}
		\|\mathcal{T}(q_N^1)-\mathcal{T}(q_N^2)\|_{L^\infty(0,T;H^1_{\mu_m})}\le CT^{\frac12}\|u_N\|_{L^\infty(0,T;\XN)}\|q_N^1-q_N^2\|_{L^\infty(0,T;H^1_{\mu_m})}.
		\end{equation}
		Using in addition \eqref{eq:estim_Lunardi}, a standard fixed point argument shows that \eqref{eq:integral_qN} has a unique solution in $\mathcal{C}^0([0,\bar T];H^1_{\mu_m})$, for $\bar T>0$ small enough. 
		Using then that $\|f\|_{H^2_{\mu_m}}\le C\|(1-\Delta_m)f\|_{L^2_{\mu_m}}$ (since the domain of $\Delta_m$ coincides with $H^2_{\mu_m}$), it is not difficult to deduce that $q_N$ also belongs to $L^2(0,T;H^2_{\mu_m})\cap H^1(0,T;L^2_{\mu_m})$. This establishes \eqref{eq:regularity_qN}.
		
		In order to prove the second inequality in \eqref{eq:positivity_qN}, it suffices to proceed as follows (see also \cite[Chapter 7.3.1]{Feireisl_Book04}). Let
		\begin{equation*}
		f(t,x)=q_N(t,x)-\frac{1}{c^0}\exp\left(\int_0^t \|\dvm(u_N(\tau))\|_{L^\infty}\,{\rm d}\tau\right).
		\end{equation*}
		Then one easily verifies that
		\begin{align*}
		\partial_tf+ \dvm(fu_N) - \delta_1 \Delta_m f \le 0
		\end{align*}
		on $(0,T)\times\mathbb{R}^d$. Taking the scalar product with $f^+=\max(f,0)$, we obtain
		\begin{equation*}
		\frac12\frac{{\rm d}}{{\rm d}t}\|f^+\|_{L^2_{\mu_m}}^2+\delta_1\|\nabla f^+\|_{L^2_{\mu_m}}^2\le-\frac12\int_{\mathbb{R}^d}|f^+|^2\dvm(u_N){\rm d} \mu_m\le\frac12\|\dvm(u_N)\|_{L^\infty}\|f^+\|_{L^2_{\mu_m}}^2,
		\end{equation*}
		from which we can conclude, by Gronwall's lemma, that the second inequality in \eqref{eq:positivity_qN} holds. The proof of the first inequality in \eqref{eq:positivity_qN} is similar.
		
		\item By \eqref{eq:integral_qN}, we have
		\begin{align*}
		q_N^1(t,x)-q_N^2(t,x)&=\int_0^te^{(t-s)\delta_1\Delta_m}\dvm((q_N^1-q_N^2)(s,x)u_N^1(s,x)){\rm d}s\\
		&\quad+\int_0^te^{(t-s)\delta_1\Delta_m}\dvm(q_N^2(s,x)(u_N^1-u_N^2)(s,x)){\rm d}s.
		\end{align*}
		The $H^{2,1}_{\mu_m}(Q_T)$-norm of the first term is bounded by $CT^{\frac12}\|u_N^1\|_{L^\infty(0,T;\XN)}\|q_N^1-q_N^2\|_{L^\infty(0,T;H^1_{\mu_m})}$ according to \eqref{eq:TqN1-2}. Using similar arguments together with \eqref{eq:regularity_qN}, the $H^{2,1}_{\mu_m}(Q_T)$-norm of the second term can be estimated by $C(\|u_N^2\|_{L^\infty(0,T;\XN)}) \|q^0\|_{H^1_{\mu_m}} \|u_N^1-u_N^2\|_{L^\infty(0,T;\XN)}$.	For $T$ small enough, this proves \eqref{eq:Lipschitz_qN}.
	\end{enumerate}
\end{proof}

\begin{Remark}\label{rk:def_STq}
	Proposition \ref{prop:Fokker-PlanckEquation_existence} allows us to introduce, given $T>0$ small enough and an initial data $(q^0,u^0_N)\in L^2_{\mu_m}\times\XN$ with $q^0$ satisfying \eqref{eq:_regularized_initial_data}, a map  
	\begin{equation}
	\mathcal{S}_T^{(q)} : \mathcal{C}^1([0,T];\XN) \to H^{2,1}_{\mu_m}(Q_T), \qquad  u_N \mapsto q_N,
	\end{equation}
	which associates to any $u_N \in \mathcal{C}^1([0,T];\XN)$ the solution $q_N$ to \eqref{eq:Fokker-Planck}. By \eqref{eq:Lipschitz_qN}, the map $\mathcal{S}_T^{(q)}$ is locally Lipschitz continuous. 
\end{Remark}

\subsubsection{Ordinary differential equation for the velocity.} \label{subsection:FG}
\label{subsection:Discrete_equation_of_momentum}

Now we wish to construct a solution $u_N$ to \eqref{eq:CNSK_system_m_regularized-velocity-WS-Definition_XN-thm} with $q_N=\mathcal{S}_T^{(q)}(u_N)$ (see Remark \ref{rk:def_STq} for the notation).

Following \cite[Chapter 7.3.3]{Feireisl_Book04}, \cite[section 3.1]{Jungel_SIMA10} or \cite[page 1493]{Vasseur_Yu_SIMA2016} let us introduce, for $q \in L^2_{\mu_m}$, the `mass operator' $\mathfrak{M}[q] : \XN \to \XN'$, defined, for all $v_N, w_N \in \XN$, by
\begin{equation*}
\left\langle  \mathfrak{M}[q] v_N, w_N  \right\rangle_{\XN',\XN} = \int_{\mathbb{R}^d} q v_N \cdot w_N \,{\rm d}\mu_m.
\end{equation*}
Clearly, the operator $\mathfrak{M}[q]$ is symmetric and satisfies
\begin{equation*}
\left\|\mathfrak{M}[q]\right\|_{\mathcal{L}(\XN,\XN')} \leq C(N)\|q\|_{L^2_{\mu_m}},
\end{equation*}
for some positive constant $C(N)$.  Moreover, if $q \in L^2_{\mu_m}$ satisfies $q \geq c >0$ (with $c$ a constant),  $\mathfrak{M}[q]$ is positive-definite in the sense that
\begin{equation*}
\underset{w_N \in \XN,\, \|w_N\|_{\XN}=1}{\inf} \left\langle \mathfrak{M}[q] w_N, w_N  \right\rangle_{\XN',\XN} \geq c .
\end{equation*}
Hence, as $\XN$ is finite-dimensional, the operator $\mathfrak{M}[q]$ is invertible and 
\begin{equation}\label{eq:mathfrak{M}_invertible}
\|{\mathfrak{M}[q]}^{-1}\|_{\mathcal{L}(\XN',\XN)} \leq c^{-1}.
\end{equation} 

In what follows we assume for simplicity that $T\le1$ (which allows us to estimate e.g. terms of order $T$ by terms of order $T^{\frac12}$). We introduce the notations
\begin{align*} 
\mathcal{R}_c^M &= \left\{q \in L^2_{\mu_m} \; \mbox{ such that } \|q\|_{L^2_{\mu_m}} \leq M \quad \mbox{and} \quad q \geq c >0 \right\}, \\ 
\mathcal{R}_{c,T}^M &= \left\{q \in \mathcal{C}(0,T;L^2_{\mu_m}) \; \mbox{ such that } q(t)\in \mathcal{R}_c^M \quad \mbox{for all} \quad t\in[0,T] \right\}. 
\end{align*}
We have the following easy lemma.
\begin{Lemma}\label{lem:Lipschitz_q}
	Let $N\in\mathbb{N}^*$, $c>0$, $M>0$. The mapping $q \mapsto \mathfrak{M}[q]^{-1}$ is Lipschitz continuous from $\mathcal{R}_c^M$
	to $\mathcal{L}(\XN',\XN)$. Moreover, if $0 < T \leq 1$ and $q^1,q^2 \in H_m^{2,1}(Q_T)\cap\mathcal{R}_{c,T}^M$,
	\begin{equation}\label{eq:LipschitzMq}
	\left\|\mathfrak{M}[q^1]^{-1} -\mathfrak{M}[q^2]^{-1}\right\|_{\mathcal{C}^0([0,T];\mathcal{L}(\XN',\XN))}  \leq  C(N,c^{-1})MT^{\frac12}\|q^1-q^2\|_{H_m^{2,1}(Q_T)},
	\end{equation}
	for some positive constant $C(N,c^{-1})$.
\end{Lemma}
\begin{proof}
	Let $q^1$, $q^2 \in \mathcal{R}_c^M$. We have
	\begin{equation*}
	\mathfrak{M}[q^1]^{-1} -\mathfrak{M}[q^2]^{-1} = \mathfrak{M}[q^1]^{-1} \left(\mathfrak{M}[q^2] - \mathfrak{M}[q^1]\right)\mathfrak{M}[q^2]^{-1}, 
	\end{equation*}
	and therefore, thanks to \eqref{eq:mathfrak{M}_invertible},
	\begin{equation}\label{eq:LipschitzMq0}
	\left\|\mathfrak{M}[q^1]^{-1} -\mathfrak{M}[q^2]^{-1}\right\|_{\mathcal{L}(\XN',\XN)} \leq C(N)c^{-2}\|q^1-q^2\|_{L^1_{\mu_m}} \leq C(N,c^{-1},M)\|q^1-q^2\|_{L^2_{\mu_m}}.
	\end{equation}
	In order to prove \eqref{eq:LipschitzMq}, it suffices to combine \eqref{eq:LipschitzMq0} with the fact that, if $q \in H_m^{2,1}(Q_T)$ satisfies $q(0) = 0$, then
	\begin{equation}\label{eq:estim-q_diff-norms}
	\|q\|_{\mathcal{C}^0([0,T];L^2_{\mu_m})} \leq T^{\frac12}\|\partial_tq\|_{L^2(0,T;L^2_{\mu_m})} \leq T^{\frac12}\|q\|_{H^{2,1}_{\mu_m}(Q_T)}.
	\end{equation}
\end{proof}

We introduce a few further notations. For any given $q \in L^2_{\mu_m}$, the operator $\mathcal{A}[q]:\XN\to\XN'$ is defined by
\begin{equation*}
\left\langle \mathcal{A}[q](u_N), \Phi_N \right\rangle_{\XN',\XN} 
=  -2\nu \int_{\mathbb{R}^d} q D(u_N) : D(\Phi_N) \;{\rm d}\mu_m - r_0 \int_{\mathbb{R}^d} u_N \cdot \Phi_N \;{\rm d}\mu_m,
\end{equation*}
for all $u_N, \Phi_N \in \XN$. Similarly, for $(q,v_N)\in H^1_{\mu_m}\times \XN$, we set
\begin{align*}
\left\langle \mathcal{N}[q,v_N](u_N), \Phi_N \right\rangle_{\XN',\XN} 
& =  \int_{\mathbb{R}^d} q v_N \cdot (\nabla \Phi_N) u_N{\rm d}\mu_m - \delta_1 \int_{\mathbb{R}^d}  \nabla u_N\nabla q \cdot \Phi_N  \;{\rm d}\mu_m \\
&\quad- r_1 \int_{\mathbb{R}^d} q|v_N|^2u_N \cdot \Phi_N  {\rm d}\mu_m ,
\end{align*}
for all $u_N, \Phi_N \in \XN$. Finally, for all $q \in H^2_{\mu_m}\cap\mathcal{R}_c^M$, we define $\mathcal{B}[q] \in \XN'$  by
\begin{align*}
\left\langle \mathcal{B}[q], \Phi_N \right\rangle_{\XN',\XN}   &=   - 2\kappa^2  \int_{\mathbb{R}^d} \left[\sqrt{q} D^2(\sqrt{q})-\nabla(\sqrt{q}) \otimes \nabla(\sqrt{q})\right] : D(\Phi_N) {\rm d}\mu_m  -\lambda \sigma^2 \int_{\mathbb{R}^d} \nabla q  \cdot \Phi_N {\rm d}\mu_m  \\
&\quad  - \frac{r_4}{\sigma^4}\int_{\mathbb{R}^d} q|x|^2x \cdot \Phi_N    \;{\rm d}\mu_m,
\end{align*}
for all $\Phi_N \in \XN$.
\begin{Lemma}\label{lem:Lipschitz_N_A_B}
	Let $N\in\mathbb{N}^*$, $0<T\le1$, $c>0$, $M>0$, $u_N^0\in\XN$ and $q^0\in L^2_{\mu_m}$. 
	\begin{enumerate}[label=(\roman*)]
		\item For all $(q,v_N)\in [H_m^{2,1}(Q_T)\cap\mathcal{R}_{c,T}^M] \times \mathcal{C}^0([0,T];\XN)$, we have
		\begin{equation*}
		\mathcal{N}[q,v_N]\in L^\infty(0,T;\mathcal{L}(\XN,\XN')), \quad \mathcal{A}[q] \in L^\infty(0,T;\mathcal{L}(\XN,\XN')), \quad \mathcal{B}[q] \in L^2(0,T;\XN').
		\end{equation*}
		\item Let $(q^1,v_N^1)$, $(q^2,v_N^2) \in [H_m^{2,1}(Q_T)\cap\mathcal{R}_{c,T}^M] \times \mathcal{C}^0([0,T];\XN)$ be such that $q^k(0) = q^0$, $v_N^k(0) = u_N^0$ and $\|v_N^k\|_{\mathcal{C}^0([0,T];\XN)} \leq M$ for $k=1,2$.
		Then we have the following estimates 
		\begin{equation*}
		\|\mathcal{A}[q^1(t)]-\mathcal{A}[q^2(t)]\|_{L^\infty(0,T;\mathcal{L}(\XN,\XN'))} \leq  C(N)M T^{\frac12} \|q^1 - q^2\|_{H^{2,1}_{\mu_m}(Q_T)}, 
		\end{equation*}
		for some positive constant $C(N)$, 
		\begin{align*}
		& \|\mathcal{N}[q^1(t),v_N^1(t)]-\mathcal{N}[q^2(t),v_N^2(t)]\|_{L^\infty(0,T;\mathcal{L}(\XN,\XN'))} \\ & \leq  C(N)M^2 \left[\|q^1 - q^2\|_{H^{2,1}_{\mu_m}(Q_T)}+\|v_N^1-v_N^2\|_{\mathcal{C}^0([0,T];\XN)}\right], 
		\end{align*}
		and
		\begin{equation*}
		\|\mathcal{B}[q^1]-\mathcal{B}[q^2]\|_{L^2(0,T;\XN')} \leq  C(N)M(1+T^\frac12)\|q^1-q^2\|_{H_m^{2,1}(Q_T)}.
		\end{equation*}
	\end{enumerate}
\end{Lemma}
\begin{proof}
	Consider for instance the term $\mathcal{A}[q]$. For $q\in H^{2,1}_{\mu_m}(Q_T) \subset  \mathcal{C}^0([0,T];H^1_{\mu_m})$, we can estimate, for all $u_N, \Phi_N \in \XN$, and all $t\in[0,T]$,
	\begin{align*}
	\left|\left\langle \mathcal{A}[q(t)](u_N), \Phi_N \right\rangle_{\XN',\XN} \right|
	&\le 2\nu \left| \int_{\mathbb{R}^d} q(t) D(u_N) : D(\Phi_N) \;{\rm d}\mu_m\right| + r_0\left| \int_{\mathbb{R}^d} u_N \cdot \Phi_N \;{\rm d}\mu_m \right| \\
	&\le C\big( \|q\|_{\mathcal{C}^0([0,T];L^2_{\mu_m})} \|D(u_N)\|_{L^\infty}\|D(\Phi_N)\|_{L^2_{\mu_m}} + \|u_N\|_{L^2_{\mu_m}}\|\Phi_N\|_{L^2_{\mu_m}}\big) \\
	&\le C(N)\big( \|q\|_{\mathcal{C}^0([0,T];L^2_{\mu_m})} \|u_N\|_{L^2}\|\Phi_N\|_{L^2} + \|u_N\|_{L^2}\|\Phi_N\|_{L^2}\big),
	\end{align*}
	where we used in particular Remark \ref{rk:elementary} in the last inequality. This proves $(i)$ for $\mathcal{A}[q]$. The proof of $(ii)$ for $\mathcal{A}[q]$ is similar, using in addition \eqref{eq:estim-q_diff-norms}. The other terms can be treated in an analogous way.
\end{proof}
Using the notations introduced above, one sees that the linearized equation corresponding to \eqref{eq:CNSK_system_m_regularized-velocity-WS-Definition_XN-thm} reads 
\begin{equation}\label{eq:NS_XN_linearized_velocity}
\frac{{\rm d}}{{\rm d}t}\left(\mathfrak{M}[q_N]u_N\right)(t) = \mathcal{A}[q_N(t)](u_N(t)) + \mathcal{N}[q_N(t),u_N(t)](u_N(t)) + \mathcal{B}[q_N(t)],  
\end{equation}
with initial condition $\mathfrak{M}[q_N](u_N)(0) = \mathfrak{M}[q^0](u_N^0)$. Here we recall that $\mathfrak{M}[q^0](u_N^0) \in \XN'$ is given, for all $w_N \in \XN$, by
\begin{equation*}
\left\langle \mathfrak{M}[q^0](u_N^0), w_N \right\rangle_{\XN',\XN} = \int_{\mathbb{R}^d} q^0 u_N^0 \cdot w_N \;{\rm d}\mu_m = \int_{\mathbb{R}^d} (\rho_m q^0u_N^0) \cdot w_N \;{\rm d}x = \int_{\mathbb{R}^d} (\rho_m q^0u^0) \cdot w_N \;{\rm d}x,
\end{equation*}
where we used the definition \eqref{def:XN_initial_data} of $u_N^0$ in the last equality.

\subsubsection{Existence of solutions to the regularized QNS system with drag forces in the Faedo-Galerkin approximation.} \label{subsection:FP}

\paragraph{Local existence.}

Now we prove the local existence of solutions to the approximate system \eqref{eq:Fokker-Planck-thm}--\eqref{eq:CNSK_system_m_regularized-velocity-WS-Definition_XN-thm}, using in particular the Lipschitz continuity properties established in Lemmas \ref{lem:Lipschitz_q} and \ref{lem:Lipschitz_N_A_B}.

\begin{Proposition}[Local existence of solutions to the approximate system \eqref{eq:Fokker-Planck-thm}--\eqref{eq:CNSK_system_m_regularized-velocity-WS-Definition_XN-thm}]\label{prop:local-exist-approx}
	Let $N\in\mathbb{N}^*$ and $(q^0,u^0_N)\in L^2_{\mu_m}\times\XN$ with $q^0$ satisfying \eqref{eq:_regularized_initial_data}. There exists $T>0$ and a unique 
	\begin{equation*}
	(q_N,u_N)\in H^{2,1}_{\mu_m}(Q_T)\times \mathcal{C}^1([0,T];\XN)
	\end{equation*}
	such that
	$q_N$ is a solution to \eqref{eq:Fokker-Planck-thm} in $L^2(0,T;L^2_{\mu_m})$, satisfying $q_N(0)=q^0$ and \eqref{eq:regularity_qN}--\eqref{eq:positivity_qN}, and $u_N$ is a solution to \eqref{eq:CNSK_system_m_regularized-velocity-WS-Definition_XN-thm}, for all $\Phi_N \in \mathcal{C}^1([0,T];\XN)$ such that $\Phi_N(T) = 0$, with $u_N(0)=u_N^0$.
\end{Proposition}
\begin{proof}
	The proof follows from standard arguments and is similar to that detailed in previous works, see e.g. \cite[Chapter 7.3.3]{Feireisl_Book04}. We only sketch it. Recalling the notation $\mathcal{S}_T^{(q)}$ from Remark \ref{rk:def_STq}, we first solve the integral equation
	\begin{align}
	u_N(t)  =  \mathfrak{M}[\mathcal{S}_T^{(q)}(u_N)(t)]^{-1}\Big[\mathfrak{M}[q^0](u_N^0) + \int_0^t &\left(\mathcal{A}[\mathcal{S}_T^{(q)}(u_N)(s)]+\mathcal{N}[\mathcal{S}_T^{(q)}(u_N)(s),u_N(s)]\right)(u_N(s)) \notag \\
	&\quad + \mathcal{B}[\mathcal{S}_T^{(q)}(u_N)(s)]\Big]\,{\rm d}s =: \mathcal{T}(u_N)(t). \label{eq:NS_XN_linearized_velocity_int-proof}
	\end{align}
	Let $M>0$ and $\bar B$ be the closed ball centered at (the function independent of time) $u_N^0$ and of radius $M$ in $L^\infty(0,T;\XN)$. From Proposition \ref{prop:Fokker-PlanckEquation_existence}, Remark \ref{rk:def_STq}, Lemmas \ref{lem:Lipschitz_q} and \ref{lem:Lipschitz_N_A_B}, it is not difficult to deduce that, for $T>0$ small enough, the mapping $\mathcal{T}$ is a contraction from $\bar B$ into itself. Therefore the Banach fixed point theorem gives the existence of a unique $u_N\in\bar B$ satisfying \eqref{eq:NS_XN_linearized_velocity_int-proof}. In turn, it is not difficult to deduce from \eqref{eq:NS_XN_linearized_velocity_int-proof} that $u_N \in \mathcal{C}^1([0,T];\XN)$ by using again the results from Proposition \ref{prop:Fokker-PlanckEquation_existence}, Remark \ref{rk:def_STq}, Lemmas \ref{lem:Lipschitz_q} and \ref{lem:Lipschitz_N_A_B}.
	
	To conclude the proof of the proposition, it suffices to apply  Proposition \ref{prop:Fokker-PlanckEquation_existence} and set $q_N=\mathcal{S}_T^{(q)}(u_N)$.
\end{proof}

\paragraph{Global existence.} \label{subsection:GE}

Now we fix a time $T>0$ and show that the local solution $(q_N,u_N)$ constructed in Proposition \ref{prop:local-exist-approx} above is in fact global on $[0,T]$. This is done by proving that $(q_N,u_N)$ is bounded in $H^1_{\mu_m} \times \XN$ on $[0,\overline{T}]$ for all $0<  \overline{T} \leq T$.

\begin{Lemma}\label{lem:FG-energy}
	Let $N\in\mathbb{N}^*$, $T>0$ and let $(q_N,u_N) \in H_m^{2,1}(Q_{\overline{T}})\times \mathcal{C}^1([0,\overline{T}];\XN)$, with  $0< \overline{T} \leq T$, $\overline{T}$ small enough, be a local-in-time solution to \eqref{eq:Fokker-Planck-thm}--\eqref{eq:CNSK_system_m_regularized-velocity-WS-Definition_XN-thm}. Then $(q_N,u_N)$ satisfies the (regularized) energy identity \eqref{eq:energy_identity_reg} and, in particular, 
	\begin{equation}\label{eq:energy_identity_reg_XN}
	\frac{{\rm d} \mathcal{E}_{\rm reg}}{{\rm d}t}(q_N(t),u_N(t)) + \frac12 \mathcal{D}_{\rm reg}(q_N(t),u_N(t)) \leq 2 r_4\delta_1\left[\frac{(d+2)^2}{\sigma^2}\right],
	\end{equation}
	where the regularized energy $\mathcal{E}_{\rm reg}$ and the regularized dissipation $\mathcal{D}_{\rm reg}$ are given in \eqref{def:energy_reg} and \eqref{def:dissipation_reg}, respectively.
	Therefore, for almost all $t \in [0,\overline{T}]$, 
	\begin{equation}\label{eq:energy_inequality_reg_XN}
	{\mathcal{E}}_{\rm reg}(q_N(t),u_N(t)) + \frac12\int_0^t \mathcal{D}_{\rm reg}(q_N(t'),u_N(t')) \,{\rm d}t' \leq  \mathcal{E}_{\rm reg}(q_N^0,u_N^0) + 2r_4\delta_1\left[\frac{(d+2)^2}{\sigma^2}\right]t.
	\end{equation}

\end{Lemma}

\begin{proof}
	The energy identity is obtained by applying \eqref{eq:CNSK_system_m_regularized-velocity-WS-Definition_XN-thm} with $\Phi_N = u_N$ together with a simple integration by parts. 
	One of the integration by parts reads, for instance:
		\begin{eqnarray*}
		- \frac{1}{\sigma^4}\int_{\mathbb{R}^d} q|x|^2x \cdot \nabla \ln(q) \;{\rm d}\mu_m & = & \frac{1}{\sigma^4}\int_{\mathbb{R}^d} q \dvm(|x|^2x) \;{\rm d}\mu_m \\
		& = & \frac1{\sigma^4}\int_{\mathbb{R}^d}q\left((d+2)|x|^2 - \frac{|x|^4}{\sigma^2}\right)   \;{\rm d}\mu_m = \frac{d+2}{\sigma^2}I_2(q) - \frac{1}{\sigma^2}I_4(q).
		\end{eqnarray*}
	The other integrations by parts follow classical calculations, as the twisted operators offset the change of reference measure.

	Note that $q_N$ is positive and regular (see \eqref{eq:regularity_qN}--\eqref{eq:positivity_qN}) so that each term from the different integrals make sense and the integrations by parts can be justified.
	
	Integrating \eqref{eq:energy_identity_reg_XN}, we obtain \eqref{eq:energy_inequality_reg_XN}.
\end{proof}

The existence of global solutions to the approximate system \eqref{eq:Fokker-Planck-thm}--\eqref{eq:CNSK_system_m_regularized-velocity-WS-Definition_XN-thm} stated in Proposition \ref{prop:existence_FD} is now a direct consequence of the local existence and the energy estimates. 
\begin{proof}[Proof of Proposition \ref{prop:existence_FD}]
	For the global existence, it suffices to combine Proposition \ref{prop:local-exist-approx}, Lemma \ref{lem:FG-energy} and a standard argument. The inequalities in \eqref{eq:positivity_qN-prop} follow exactly as in the proof of Proposition \ref{prop:Fokker-PlanckEquation_existence}. The conservation of mass is an easy consequence of the fact that $q_N$ satisfies \eqref{eq:Fokker-Planck-thm}. The energy estimate is proven in Lemma \ref{lem:FG-energy}.
\end{proof}

\subsection{Convergence of the Faedo-Galerkin approximation}\label{subsection:A priori Estimate}

In this subsection, we show that a global solution $(q_N,u_N)$, constructed in Proposition \ref{prop:existence_FD}, to the approximate system \eqref{eq:Fokker-Planck-thm}--\eqref{eq:CNSK_system_m_regularized-velocity-WS-Definition_XN-thm} converges, in the limit $N\to\infty$, to a solution of the regularized system \eqref{eq:CNSK_system_m_regularized}.  As in the previous subsection, the parameters $0<r_0,r_1 ,r_4 ,\delta_1 \le1$ are fixed throughout Subsection \ref{subsection:A priori Estimate}.

\subsubsection{Energy estimates.}

We first deduce from the energy estimate stated in Lemma~\ref{lem:FG-energy} the following a priori estimates for a solution $(q_N,u_N)$ to \eqref{eq:Fokker-Planck-thm}--\eqref{eq:CNSK_system_m_regularized-velocity-WS-Definition_XN-thm}.

\begin{Lemma}\label{lem:apriori}
	Let $N\in\mathbb{N}^*$, $T>0$ and $(q^0,u^0_N)\in L^2_{\mu_m}\times\XN$ with $q^0$ satisfying \eqref{eq:_regularized_initial_data}. Let $(q_N,u_N)\in H^{2,1}_{\mu_m}(Q_T)\times \mathcal{C}^1([0,{T}];\XN)$ be a global solution to \eqref{eq:Fokker-Planck-thm}--\eqref{eq:CNSK_system_m_regularized-velocity-WS-Definition_XN-thm} in the sense of Proposition \ref{prop:existence_FD}. Then there exists a positive constant $C$ independent of $N$ and the parameters $r_0,r_1,r_4,\delta_1$  such that
	\begin{subequations}
		\label{eq:FG_apriori_estimates}
		\begin{align}
		\|\sqrt{q_N}\|_{L^\infty(0,T;L^2_{\mu_m})}+\|\nabla {\sqrt{q_N}}\|_{L^\infty(0,T;[L^2_{\mu_m}]^d)}+ \|q_N \ln(q_N)\|_{L^\infty(0,T;L^1_{\mu_m})}  &\leq C,     \label{eq:FG_apriori_estimates_q_energy} \\
		\sqrt{r_0}\|u_N\|_{L^2(0,T;[L^2_{\mu_m}]^d)} + \|\sqrt{q_N} u_N\|_{L^\infty(0,T;[L^2_{\mu_m}]^d)} + r_4^{\frac14}\|q_N^{\frac14}x\|_{L^\infty(0,T;[L^4_{\mu_m}]^d)} &\leq C, \label{eq:FG_apriori_estimates_q_energy2} \\
		\sqrt{\delta_1} \|\sqrt{q_N}D^2(\ln(q_N))\|_{L^2(0,T;[L^2_{\mu_m}]^{d \times d})} &\leq  C, \label{eq:FG_apriori_estimates_q_dissipation} \\
		r_4^{\frac14}\delta_1^{\frac14}\|q_N^{\frac14}x\|_{L^4(0,T;[L^4_{\mu_m}]^d)} + \|\sqrt{q_N}D(u_N)\|_{L^2(0,T;[L^2_{\mu_m}]^{d\times d})}  + r_1^{\frac14}\|q_N^{\frac14}u_N\|_{L^4(0,T;[L^4_{\mu_m}]^d)} &\leq C, \label{eq:FG_apriori_estimates_dissipation}
		\end{align}
		and
		\begin{equation}
		r_4\delta_1\left[\|D^2(\sqrt{q_N})\|_{L^2(0,T;[L^2_{\mu_m}]^{d\times d})}^2 + \|\nabla{q_N^{\frac14}}\|_{L^4(0,T;[L^4_{\mu_m}]^d)}^4\right] \leq C. \label{eq:FG_apriori_estimates_hessian_dissipation}
		\end{equation}
	\end{subequations}
\end{Lemma}
\begin{proof}
	Estimates \eqref{eq:FG_apriori_estimates_q_energy}--\eqref{eq:FG_apriori_estimates_dissipation} directly follow from Lemma \ref{lem:FG-energy} and the expressions of $\mathcal{E}_{\rm reg}$ and $\mathcal{D}_{\rm reg}$ given in \eqref{def:energy_reg} and \eqref{def:dissipation_reg}, respectively.
	
	Estimate \eqref{eq:FG_apriori_estimates_hessian_dissipation} is obtained by combining the last term in \eqref{eq:FG_apriori_estimates_q_dissipation}, the first term in \eqref{eq:FG_apriori_estimates_dissipation} and Lemma \ref{lem:Hessian_estimate}.
\end{proof}

\begin{Remark}\label{rk:eq_sqrt(q_N)}
	Given a local solution $(q_N,u_N)$ as in Proposition \ref{prop:local-exist-approx}, the regularity and positivity of $q_N$ (see \eqref{eq:regularity_qN}--\eqref{eq:positivity_qN}), together with the fact that $u_N\in \mathcal{C}^1([0,T];\XN)$ imply that $(\sqrt{q_N},u_N)$ satisfies the following equation in $L^2(0,T;L^2_{\mu_m})$,
	\begin{equation}\label{eq:eq-for-sqrt(qN)}
	\partial_t(\sqrt{q_N}) + \frac{1}{2}\dvm(\sqrt{q_N}u_N) + 2(q_N^{\frac14}u_N) \cdot \nabla(q_N^{\frac14}) = \delta_1 \Delta_m(\sqrt{q_N}) + 4\delta_1 |\nabla(q_N^{\frac14})|^2,
	\end{equation}
	with
	\begin{equation*}
	\dvm(\sqrt{q_N}u_N) = \dv(\sqrt{q_N}u_N) - \left(\frac{x}{\sigma^2}q_N^{\frac14}\right) \cdot (q_N^\frac14 u_N) .
	\end{equation*}
	See the proof of Proposition \ref{prop:FG-convergence} for the justification that all terms in \eqref{eq:eq-for-sqrt(qN)} belong to $L^2(0,T;L^2_{\mu_m})$. 
\end{Remark}

\subsubsection{Convergence of the approximate solutions in the limit $N\to\infty$.}

The main tool to prove the convergence of a solution $(q_N,u_N)$ to \eqref{eq:Fokker-Planck-thm}--\eqref{eq:CNSK_system_m_regularized-velocity-WS-Definition_XN-thm} in the limit $N\to\infty$ is the well-known Aubin-Lions-Simon Lemma (see e.g. \cite[Theorem II.5.16]{Boyer_Fabrie_book}) which we now recall.

\begin{Lemma}[Aubin-Lions-Simon Lemma] \label{lemma:Aubin-Lions-Simon}
	
	Let $B \subset B_0 \subset B_{*}$ be three Banach spaces. We assume that the embedding of $B_0$ in $B_{*}$ is continuous and the embedding of $B$ in $B_0$ is compact. 
	
	Let $p$, $p_{*}$ such that $1 \leq p, p_{*} \leq \infty$. For $T>0$, we define
	\begin{equation*}
	W = \left\{ v \in L^{p}(0,T;B) \; ; \; \frac{{\rm d}v}{{\rm d}t} \in L^{p_{*}}(0,T; B_{*})  \right\}.
	\end{equation*}
	Then
	\begin{enumerate}[label=(\roman*)]
		\item If $p < \infty$, the embedding of $W$ in $L^p(0,T; B_0)$ is compact.
		\item If $p=\infty$ and $p_*>1$, the embedding of $W$ in $\mathcal{C}^0([0,T];B_0)$ is compact.
	\end{enumerate}
\end{Lemma} 

The following proposition establishes, in particular, the existence of a global weak solution to the regularized QNS system \eqref{eq:CNSK_system_m_regularized}. 

\begin{Proposition}\label{prop:FG-convergence}
	Let $T>0$ and $(q^0,u^0)\in L^2_{\mu_m}\times[L^2]^d$ satisfying \eqref{eq:_regularized_initial_data}--\eqref{eq:initial_(q0,u0)}. For all $N\in\mathbb{N}^*$, let $(q_N,u_N)\in H^{2,1}_{\mu_m}(Q_T)\times \mathcal{C}^1([0,{T}];\XN)$ be a global solution to \eqref{eq:Fokker-Planck-thm}--\eqref{eq:CNSK_system_m_regularized-velocity-WS-Definition_XN-thm}, associated to the initial data $(q^0,u^0_N)$ (with $u^0_N$ defined by \eqref{def:XN_initial_data}), in the sense of Proposition \ref{prop:existence_FD}. Then
	
	\begin{enumerate}[label=(\roman*)]
		\item	The sequence $(\sqrt{q_N})_{N}$ converges strongly (along some subsequence which will be denoted by the same symbol) in $L^2(0,T;H^1_{\mu_m}) \cap \mathcal{C}^0([0,T];L^2_{\mu_m})$. We set ${\s}=\lim\sqrt{q_N}$.
		
		\item The sequence $(q_Nu_N)_{N}$ converges strongly (along some subsequence which will be denoted by the same symbol) in $L^2(0,T;[L^1_{\mu_m}]^d)$.  We set $J=\lim q_Nu_N$. 
		
		\item	Let $q = {\s}^2$ and let $u : (0,T) \times \mathbb{R}^d \mapsto \mathbb{R}^d$ be defined by 
		\begin{equation*}
		u(t,x) = \frac{J(t,x)}{{\s}^2(t,x)} \quad \mbox{for a.e. } (t,x) \in (0,T) \times \mathbb{R}^d \mbox{ such that } {\s}(t,x) \neq 0, \quad u(t,x) = 0 \quad \mbox{elsewhere.}
		\end{equation*}
		Then $(q,u)$ is a solution to \eqref{eq:CNSK_system_m_regularized} in $\mathcal{D}'([0,T)\times\mathbb{R}^d)\times[\mathcal{D}'([0,T)\times \mathbb{R}^d)]^d$. 
	\end{enumerate}
\end{Proposition}

\begin{proof}
	
	In this proof $C$ denotes a positive constant dependent on the time $T$ and the parameters $\delta_1$, $r_0$, $r_1$, $r_4$, but independent of $N$.
	
	To prove $(i)$, we first show that $(\sqrt{q_N})_N$ converges strongly in $L^2(0,T;H^1_{\mu_m})$. 
	Estimate \eqref{eq:FG_apriori_estimates_hessian_dissipation} implies that
	\begin{equation*}
	\|\Delta\sqrt{q_N}\|_{L^2(0,T;L^2_{\mu_m})}\le C,
	\end{equation*}
	uniformly in $N$. Together with the conservation of mass recalled in \eqref{eq:FG_apriori_estimates_q_energy}, this shows that $(\sqrt{q_N})_{N}$ is uniformly bounded in $L^2(0,T;H^2_{\mu_m})$.  The compact embedding $H^2_{\mu_m} \hookrightarrow H^1_{\mu_m}$ therefore shows that there exists a subsequence of $(\sqrt{q_N})_{N}$ (denoted with the same index) and an element $\s$ such that $(\sqrt{q_N})_{N}$ strongly converges to $\s$ in $L^2(0,T;H^1_{\mu_m})$.
	
	Next we show that $(\sqrt{q_N})_{N}$ strongly converges to $\s$ in $\mathcal{C}^0([0,T];L^2_{\mu_m})$. We recall from Remark \ref{rk:eq_sqrt(q_N)} that 
	\begin{equation}\label{eq:sqrt(qN)}
	\partial_t(\sqrt{q_N}) =- \frac{1}{2} \big(\dv(\sqrt{q_N}u_N) - \big(\frac{x}{\sigma^2}q_N^{\frac14}\big) \cdot (q_N^\frac14 u_N)\big)  - 2(q_N^{\frac14}u) \cdot \nabla(q_N^{\frac14}) + \delta_1 \Delta_m(\sqrt{q_N}) + 4\delta_1 |\nabla(q_N^{\frac14})|^2.
	\end{equation}
	Each term of the right-hand side of the previous equation can be bounded in $L^2(0,T;L^2_{\mu_m})$ uniformly in $N$, using the a priori estimates of Lemma \ref{lem:apriori}. Indeed, 
	using \eqref{eq:FG_apriori_estimates_dissipation}, the second term can be bounded as
	\begin{equation*}
	\big\|\big(\frac{x}{\sigma^2}q_N^{\frac14}\big) \cdot (q_N^\frac14 u_N)\big\|_{L^2(0,T;L^2_{\mu_m})}\le\big\|\frac{x}{\sigma^2}q_N^{\frac14}\big\|_{L^4(0,T;[L^4_{\mu_m}]^d)} \big\|q_N^\frac14 u_N\big\|_{L^4(0,T;[L^4_{\mu_m}]^d)}\le C.
	\end{equation*}
	The same holds for the term $(q_N^{\frac14}u) \cdot \nabla(q_N^{\frac14})$, using in addition \eqref{eq:FG_apriori_estimates_hessian_dissipation}. To bound $\Delta_m(\sqrt{q_N})$, we can use Proposition \ref{prop:StrongPoincare-consequences}  together with \eqref{eq:FG_apriori_estimates_q_energy} and \eqref{eq:FG_apriori_estimates_hessian_dissipation}, writing,
	\begin{align*}
	\big\|\Delta_m(\sqrt{q_N})\big\|_{L^2(0,T;L^2_{\mu_m})}&\le\big\|\Delta(\sqrt{q_N})\big\|_{L^2(0,T;L^2_{\mu_m})}+\sigma^{-2}\big\|x\cdot\nabla\sqrt{q_N}\big\|_{L^2(0,T;L^2_{\mu_m})}\\
	&\le C\big(\big\|\Delta(\sqrt{q_N})\big\|_{L^2(0,T;L^2_{\mu_m})}+\big\|D^2(\sqrt{q_N})\big\|_{L^2(0,T;[L^2_{\mu_m}]^{d\times d})}+\big\|\nabla\sqrt{q_N}\big\|_{L^2(0,T;[L^2_{\mu_m}]^{d})}\big) \\
	&\le C.
	\end{align*}
	Using again \eqref{eq:FG_apriori_estimates_hessian_dissipation}, we have $\||\nabla(q_N^{\frac14})|^2\|_{L^2(0,T;[L^2_{\mu_m}]^d)}\le C$. It remains to estimate the term $\dv(\sqrt{q_N}u_N)$. By a direct computation, we can write, for a.e. $t\ge0$,
	\begin{equation}\label{eq:korn1}
	\|\dv(\sqrt{q_N}u_N)\|_{L^2_{\mu_m}}^2\le\|\nabla(\sqrt{q_N}u_N)\|^2_{[L^2_{\mu_m}]^{d\times d}} \le 2\|D(\sqrt{q_N}u_N)\|^2_{[L^2_{\mu_m}]^{d \times d}} + \sigma^{-2} \|\sqrt{q_N}u_N\|^2_{[L^2_{\mu_m}]^d},
	\end{equation}
	(see \cite[Eq.(25)]{Carrapatoso_Dolbeault_Herau_Mischler_Mouhot_ARMA22}); see also \cite[Theorem 1]{Carrapatoso_Dolbeault_Herau_Mischler_Mouhot_ARMA22} for general `weighted' Korn-type inequalities). We have $\|\sqrt{q_N}u_N\|_{L^\infty(0,T;[L^2_{\mu_m}]^d)}\le C$ by \eqref{eq:FG_apriori_estimates_q_energy2}. Moreover, writing (see Remark \ref{rem:AppendixA-rho-rho_m} below for the notation $\overset{{\rm sym}}{\otimes}$)
	\begin{equation}\label{eq:korn2}
	D(\sqrt{q_N}u_N)=\sqrt{q_N}D(u_N)+ 2q_N^{\frac14}u_N \overset{{\rm sym}}{\otimes} \nabla \left(q_N^{\frac14}\right),
	\end{equation}
	and using \eqref{eq:FG_apriori_estimates_dissipation} and \eqref{eq:FG_apriori_estimates_hessian_dissipation}, we deduce that 
	\begin{equation}\label{eq:korn3}
	\|D(\sqrt{q_N}u_N)\|_{L^2(0,T;[L^2_{\mu_m}]^{d\times d})}\le C.
	\end{equation}
	It then follows from \eqref{eq:korn1} that $\|\dv(\sqrt{q_N}u_N)\|_{L^2(0,T;L^2_{\mu_m})}\le C$. Putting all together, we have shown that $\partial_t(\sqrt{q_N})$ is uniformly bounded in $L^2(0,T;L^2_{\mu_m})$. Since in addition $(\sqrt{q_N})_N$ is also uniformly bounded in $L^\infty(0,T;H^1_{\mu_m})$ by \eqref{eq:FG_apriori_estimates_q_energy}, the Aubin-Lions-Simon Lemma implies that $(\sqrt{q_N})_N$ converges strongly to $s$ in $\mathcal{C}^0([0,T];L^2_{\mu_m})$. This proves $(i)$.

	Now we prove $(ii)$. We apply again the Aubin-Lions-Simon Lemma, showing that $(\partial_t(q_N u_N))_N$ is uniformly bounded in $L^1(0,T;[(W^{1,\infty}_{\mu_m})']^d)$ (note that $W^{1,\infty}_{\mu_m}=W^{1,\infty}_{x}$) and $(q_N u_N)_N$ is uniformly bounded in $L^1(0,T;[W^{1,1}_{\mu_m}]^d)$. First, \eqref{eq:FG_apriori_estimates_q_energy}, \eqref{eq:FG_apriori_estimates_q_energy2} and the Cauchy-Schwarz inequality imply that
	\begin{equation}\label{eq:boundq_Nu_N}
	\|q_Nu_N\|_{L^\infty(0,T;[L^1_{\mu_m}]^d)}\le C,
	\end{equation}
	uniformly in $N$. Next we write
	\begin{equation}\label{eq:nablaq_Nu_N}
	\nabla (q_Nu_N) = (\nabla\sqrt{q_N})\otimes\sqrt{q_N}u_N+\sqrt{q_N}\,\nabla(\sqrt{q_N}u_N).
	\end{equation}
	The first term is estimated as 
	\begin{equation*}
	\big\|(\nabla\sqrt{q_N})\otimes\sqrt{q_N}u_N\big\|_{L^\infty(0,T;[L^1_{\mu_m}]^{d\times d})}\le\big\|\nabla\sqrt{q_N}\big\|_{L^\infty(0,T;[L^2_{\mu_m}]^d)}\big\|\sqrt{q_N}u_N\big\|_{L^\infty(0,T;[L^2_{\mu_m}]^d)}\le C,
	\end{equation*}
	using \eqref{eq:FG_apriori_estimates_q_energy} and \eqref{eq:FG_apriori_estimates_q_energy2}. For the second term, using in addition \eqref{eq:korn1} and \eqref{eq:korn3}, we obtain  likewise that 
	\begin{equation*}
	\|\sqrt{q_N}\,\nabla(\sqrt{q_N}u_N)\|_{L^2(0,T;[L^1_{\mu_m}]^{d\times d})} \le C .
	\end{equation*}
	Eq. \eqref{eq:nablaq_Nu_N} together with the two previous estimates imply that 
	\begin{equation}\label{eq:bound_nablaq_Nu_N}
	\|\nabla (q_Nu_N)\|_{L^2(0,T;[L^1_{\mu_m}]^{d\times d})}\le C,
	\end{equation}
	uniformly in $N$, and it then follows from \eqref{eq:boundq_Nu_N} and \eqref{eq:bound_nablaq_Nu_N} that
	\begin{equation}\label{eq:unif_bound_qNuN}
	\|q_N u_N\|_{L^2(0,T;[W^{1,1}_{\mu_m}]^d)}\le C.
	\end{equation}
	In order to bound the time derivative $(\partial_t(q_N u_N))_N$ in $L^2(0,T;[(W^{1,\infty}_{\mu_m})']^d)$, we recall that  $(q_N,u_N)$ satisfies \eqref{eq:CNSK_system_m_regularized-velocity-WS-Definition_XN-thm},
	\begin{align}
	&   \int_{Q_T} \partial_t(q_N u_N) \cdot \Phi_N {\rm d}\mu_m{\rm d}t \notag\\
	&= \int_{Q_T} q_N\nabla \Phi_N u_N \cdot u_N {\rm d}\mu_m{\rm d}t  -\delta_1 \int_{Q_T} \nabla u_N \nabla q_N \cdot \Phi_N {\rm d}\mu_m{\rm d}t \notag\\
	&\quad+ 2\nu \int_{Q_T} q_N D(u_N) : D(\Phi_N) {\rm d}\mu_m{\rm d}t + 2\kappa^2 \int_{Q_T} \left[\sqrt{q_N} D^2(\sqrt{q_N})-\nabla(\sqrt{q_N}) \otimes \nabla(\sqrt{q_N})\right] : D(\Phi_N) {\rm d}\mu_m {\rm d}t \notag\\
	&\quad  +\lambda \sigma^2 \int_{Q_T} \nabla q_N  \cdot \Phi_N {\rm d}\mu_m{\rm d}t + r_0 \int_{Q_T} u_N \cdot \Phi_N  {\rm d}\mu_m{\rm d}t+ r_1\int_{Q_T} q_N|u_N|^2u_N \cdot \Phi_N  {\rm d}\mu_m{\rm d}t \notag\\
	&\quad+ \frac{r_4}{\sigma^4}\int_{Q_T} q_N|x|^2x \cdot \Phi_N  {\rm d}\mu_m{\rm d}t. \label{eq:CNSK_system_m_regularized-velocity-WS-Definition_XN-proof}
	\end{align}
	For the first term in the right-hand side, we use \eqref{eq:FG_apriori_estimates_q_energy} which implies that $\|q_Nu_N\cdot u_N\|_{L^\infty(0,T;L^1_{\mu_m})}\le C$. For the second term, we write
	\begin{equation*}
	\nabla u_N \nabla q_N=2\nabla(\sqrt{q_N})\sqrt{q_N}\nabla u_N=2\nabla(\sqrt{q_N})\left(\nabla(\sqrt{q_N}u_N)-2q_N^{\frac14}u_N \otimes \nabla \left(q_N^{\frac14}\right)\right).
	\end{equation*}
	Combining \eqref{eq:FG_apriori_estimates_q_energy}, \eqref{eq:korn1}--\eqref{eq:korn2}, \eqref{eq:FG_apriori_estimates_dissipation} and \eqref{eq:FG_apriori_estimates_hessian_dissipation}, we obtain that $\|\nabla u_N \nabla q_N\|_{L^2(0,T;[L^1_{\mu_m}]^d)}\le C$. The third term is estimated thanks to \eqref{eq:FG_apriori_estimates_q_energy} and \eqref{eq:FG_apriori_estimates_dissipation}, which yield $\|q_ND(u_N)\|_{L^2(0,T;[L^1_{\mu_m}]^{d\times d})}\le C$. For the fourth term, it follows from \eqref{eq:FG_apriori_estimates_q_energy} and \eqref{eq:FG_apriori_estimates_q_dissipation} that $\|\sqrt{q_N} D^2(\sqrt{q_N})\|_{L^2(0,T;[L^1_{\mu_m}]^{d \times d})}\le C$ and from \eqref{eq:FG_apriori_estimates_q_energy} that $\|\nabla(\sqrt{q_N}) \otimes \nabla(\sqrt{q_N})\|_{L^\infty(0,T;[L^1_{\mu_m}]^{d \times d})}\le C$. For the fifth term, \eqref{eq:FG_apriori_estimates_q_energy} and \eqref{eq:FG_apriori_estimates_dissipation} give $\|\nabla q_N\|_{L^\infty(0,T;[L^1_{\mu_m}]^d)}\le C$. The sixth term is directly estimated thanks to \eqref{eq:FG_apriori_estimates_q_energy2} which gives $\|u_N\|_{L^2(0,T;[L^2_{\mu_m}]^d)}\le C$. The last two terms are estimated in the same way, combining \eqref{eq:FG_apriori_estimates_q_energy2} and \eqref{eq:FG_apriori_estimates_dissipation}, which gives $\|q_N|u_N|^2u_N\|_{L^2(0,T;[L^1_{\mu_m}]^d)}\le C$ and $\|q_N|x|^2x\|_{L^2(0,T;[L^1_{\mu_m}]^d)}\le C$. These bounds together with \eqref{eq:CNSK_system_m_regularized-velocity-WS-Definition_XN-proof} show that
	\begin{equation}\label{eq:unif_bound_nablaqNuN}
	\|\partial_t(q_N u_N)\|_{L^2(0,T;[(W^{1,\infty}_{\mu_m})']^d)}\le C.
	\end{equation}
	From \eqref{eq:unif_bound_qNuN}, \eqref{eq:unif_bound_nablaqNuN}, the compact embedding $W^{1,1}_{\mu_m} \hookrightarrow L^1_{\mu_m}$ (see \cite[Theorem 3.1]{Hooton_JMAA81}) and the Aubin-Lions-Simon Lemma, we deduce that there exists a subsequence of $(q_Nu_N)_N$ (denoted by the same symbol) and an element $J$ in $L^2(0,T;[L^1_{\mu_m}]^d)$ such that $(q_Nu_N)_{N}$ converges to $J$ in $L^2(0,T;[L^1_{\mu_m}]^d)$. This establishes $(ii)$.
	
	To verify $(iii)$, it suffices to observe that, by the uniform bounds proven in $(i)$, each term of the equation \eqref{eq:sqrt(qN)} satisfied by $\sqrt{q_N}$ is uniformly bounded in $L^2(0,T;L^2_{\mu_m})$ and therefore converges weakly to the corresponding term for $(q,u)$ instead of $(q_N,u_N)$. The same holds for the equation \eqref{eq:CNSK_system_m_regularized-velocity} in $L^2(0,T;[(W^{1,\infty}_{\mu_m})']^d)$ by the estimates proven in $(ii)$.
	
\end{proof}

\subsection{Convergence of the solutions to the regularized QNS system in the limit $\delta_1\to0$}\label{subsect:GWS_drag_term}

In order to prove the convergence of the solutions to the regularized QNS system \eqref{eq:CNSK_system_m_regularized} constructed in Proposition \ref{prop:FG-convergence}, the main ingredients are the energy and BD entropy estimates \eqref{eq:energy_estimates_GWS_CNSK_system_m_regularized} and \eqref{eq:BD_estimates_GWS_CNSK_system_m_regularized} (with $\delta_1>0$) in the statement of Theorem \ref{prop:GWS_CNSK_system_m_regularized}. Establishing \eqref{eq:energy_estimates_GWS_CNSK_system_m_regularized} and \eqref{eq:BD_estimates_GWS_CNSK_system_m_regularized} is the purpose of the following two paragraphs. We begin with the energy estimate \eqref{eq:energy_estimates_GWS_CNSK_system_m_regularized} which is easier to derive.

\subsubsection{Energy estimates for solutions to the regularized QNS system}

\begin{Proposition}\label{prop:En}
	Under the conditions of Proposition \ref{prop:FG-convergence}, let $(q,u)$ be a solution to \eqref{eq:CNSK_system_m_regularized} constructed as in Proposition \ref{prop:FG-convergence}. There exists a constant $C_0(\mathcal{E}_{\rm reg}(q^0,u^0),T)>0$ depending on $\mathcal{E}_{\rm reg}(q^0,u^0)$ and $T$ (but independent of $\delta_1$) such that, for almost every $t \in (0,T)$, 
	\begin{equation}\label{eq:En_ineq}
	\mathcal{E}_{\rm reg}(t) + \frac12\int_0^t \mathcal{D}_{\rm reg}(t') \;{\rm d}t' \leq  \mathcal{C}_{0}(\mathcal{E}_{\rm reg}(q^0,u^0), T).
	\end{equation}
\end{Proposition}

\begin{proof}
	It follows from \eqref{eq:energy_inequality_reg_XN}, lower semicontinuity properties and the uniform bounds proven in the proof of Proposition \ref{prop:FG-convergence} that
	\begin{equation*}
	{\mathcal{E}}_{\rm reg}(q(t),u(t)) + \frac12\int_0^t \mathcal{D}_{\rm reg}(q(t'),u(t')) \,{\rm d}t' \leq  \mathcal{E}_{\rm reg}(q_N^0,u_N^0) + r_4\delta_1\left[\frac{(d+2)^2}{2\sigma^2}\right]t.
	\end{equation*}
	Let us justify that $\mathcal{E}_{\rm reg}(q_N^0,u_N^0)\le \mathcal{E}_{\rm reg}(q^0,u^0)$. Recalling that $q_N^0=q^0$, we see in the expression \eqref{def:energy_reg} of $\mathcal{E}_{\rm reg}$ that only the kinetic fluid energy $\frac12 \int_{\mathbb{R}^d} q_N^0 |u_N^0|^2    \;{\rm d}\mu_m$ depends on $u_N^0$. Recall that $u_N^0 \in \XN$ is defined by \eqref{def:XN_initial_data}
	and thus, for all $N\geq 1$,
	\begin{equation*}
	\int_{\mathbb{R}^d}q_N^0 |u_N^0|^2  \;{\rm d}\mu_m 
	=  \int_{\mathbb{R}^d} q^0u^0 \cdot u_N^0 \;{\rm d}\mu_m  \leq  \frac12 \int_{\mathbb{R}^d}q^0 |u_N^0|^2  \;{\rm d}\mu_m + \frac12 \int_{\mathbb{R}^d}q^0 |u^0|^2  \;{\rm d}\mu_m.
	\end{equation*}
	Therefore
	\begin{equation*}
	\frac12 \int_{\mathbb{R}^d}q_N^0 |u_N^0|^2  \;{\rm d}\mu_m \leq \frac12 \int_{\mathbb{R}^d}q^0 |u^0|^2  \;{\rm d}\mu_m,
	\end{equation*}
	which concludes the proof of \eqref{eq:energy_estimates_GWS_CNSK_system_m_regularized}.
\end{proof}

\subsubsection{BD entropy estimates and conclusion of the proof of Theorem \ref{prop:GWS_CNSK_system_m_regularized} for $\delta_1>0$.}

Recall that the BD entropy and dissipation $\mathcal{E}_{\rm BD, reg}$, $\mathcal{D}_{\rm BD, reg}$ for $(q,u)$ have been defined in \eqref{def:BD_reg-energy}--\eqref{def:BD_reg-dissipation}. We now aim at establishing the following estimate.

\begin{Proposition}\label{prop:BD}
	Under the conditions of Proposition \ref{prop:FG-convergence}, let $(q,u)$ be a solution to \eqref{eq:CNSK_system_m_regularized} constructed as in Proposition \ref{prop:FG-convergence}. There exists a constant $C_{\rm BD,0}(\mathcal{E}_{\rm reg}(q^0,u^0), \mathcal{E}_{\rm BD, reg}(q^0,u^0),T)>0$ depending on $\mathcal{E}_{\rm reg}(q^0,u^0)$, $\mathcal{E}_{\rm BD, reg}(q^0,u^0)$ and $T$ (but independent of $0<\delta_1\le 1$) such that, for almost every $t \in (0,T)$, 
	\begin{equation}\label{eq:EBD_ineq}
	\mathcal{E}_{\rm BD, reg}(t) + \frac12\int_0^t \mathcal{D}_{\rm BD, reg}(t') \;{\rm d}t' \leq  \mathcal{C}_{\rm BD, 0}(\mathcal{E}_{\rm reg}(q^0,u^0), \mathcal{E}_{\rm BD, reg}(q^0,u^0),T).
	\end{equation}
\end{Proposition}

Before turning to the proof of Proposition \ref{prop:BD}, let us compute the time-derivative of $q_N\nabla \ln(q_N)$ (given an approximate solution $(q_N,u_N)$ to \eqref{eq:CNSK_system_m_regularized}), analogously to what we did in Remark \ref{rk:eq_sqrt(q_N)} for $\sqrt{q_N}$.
\begin{Remark}\label{rk:eq_ln(q_N)}
	Given a sequence $((q_N,u_N))_N$ as in Proposition \ref{prop:FG-convergence}, one can verify that the following equation is satisfied in $L^1(0,T;[(W^{1,\infty}_{\mu_m})']^d)$:
	\begin{align}
	&\partial_t(q_N\nabla \ln(q_N)) + \dvm\left(\nabla \ln(q_N) \otimes q_N(u_N - \delta_1\nabla \ln(q_N))\right) \notag\\
	&= \delta_1 \dvm(q_ND^2\ln(q_N)) - \dvm(q_N\nabla u_N ^\top) + \frac{q_N}{\sigma^2}(u_N - \delta_1\nabla \ln(q_N)). \label{eq:q_nabla_log_q_equation}
	\end{align}
	In turn, using $q_N\partial_t(\nabla \ln(q_N))=\partial_t(q_N\nabla \ln(q_N))-\partial_t(q_n)\nabla \ln(q_N)$, this also yields, in $L^1(0,T;[(W^{1,\infty}_{\mu_m})']^d)$,
	\begin{align}
	&q_N\left[\partial_t(\nabla \ln(q_N)) + D^2(\ln(q_N))(u_N - \delta_1\nabla \ln(q_N))\right] \notag\\
	&= \delta_1 \dvm(q_ND^2\ln(q_N)) - \dvm(q_N\nabla u_N^\top) + \frac{q_N}{\sigma^2}(u_N - \delta_1\nabla \ln(q_N)). \label{eq:nabla_log_q_equation}
	\end{align}
	See the proof of the next lemma for justifications that all terms of the previous equations belong to $L^1(0,T;[(W^{1,\infty}_{\mu_m})']^d)$.
\end{Remark}

The proof of Proposition \ref{prop:BD} relies on the following two lemmas.
\begin{Lemma}\label{lm:BD1}
	Under the conditions of Proposition \ref{prop:FG-convergence}, let $(q,u)$ be a solution to \eqref{eq:CNSK_system_m_regularized} associated to a sequence $((q_N,u_N))_N$ as in Proposition \ref{prop:FG-convergence}. Then, for almost every $t\in(0,T)$, 
	\begin{align}
	& \frac{{\rm d}}{{\rm d}t} \int_{\mathbb{R}^d} q_N u_N \cdot \nabla(\ln(q_N)) \,{\rm d}\mu_m +\delta_1 \int_{\mathbb{R}^d} \nabla u_N \nabla q_N \cdot \nabla(\ln(q_N)) \;{\rm d}\mu_m \notag \\
	& + (2\nu+\delta_1) \int_{\mathbb{R}^d} q_ND(u_N) : D^2(\ln(q_N)) \;{\rm d}\mu_m +\kappa^2 \int_{\mathbb{R}^d} q_N |D^2(\ln(q_N))|^2 \;{\rm d}\mu_m \notag \\
	&  + 4\lambda \sigma^2 \int_{\mathbb{R}^d} |\nabla(\sqrt{q_N})|^2  \;{\rm d}\mu_m +r_0 \int_{\mathbb{R}^d} u_N \cdot \nabla \ln(q_N) \;{\rm d}\mu_m + r_1 \int_{\mathbb{R}^d} |u_N|^2 u_N \cdot \nabla q_N  \;{\rm d}\mu_m  \notag\\
	& =   \int_{\mathbb{R}^d} q_N \nabla u_N : \nabla u_N^\top \;{\rm d}\mu_m  + \frac{1}{\sigma^2}\int_{\mathbb{R}^d} q_Nu_N \cdot (u_N -\delta_1\nabla \ln(q_N)) \;{\rm d}\mu_m. \label{eq:produit_scalaire_u_nabla_q-proof}
	\end{align}	
\end{Lemma}

\begin{proof}
	Let $(q_N,u_N)$ be a sequence associated with $(q,u)$ as in Proposition \ref{prop:FG-convergence}. From that proposition and its proof, it follows that, for all $\Phi\in L^2(0,T;W^{1,\infty}_{\mu_m})$,
	\begin{align}
	&   \int_{Q_T} \partial_t(q_N u_N) \cdot \Phi {\rm d}\mu_m{\rm d}t \notag\\
	&= \int_{Q_T} q_N\nabla \Phi u_N \cdot u_N {\rm d}\mu_m{\rm d}t  -\delta_1 \int_{Q_T} \nabla u_N \nabla q_N \cdot \Phi {\rm d}\mu_m{\rm d}t \notag\\
	&\quad+ 2\nu \int_{Q_T} q_N D(u_N) : D(\Phi) {\rm d}\mu_m{\rm d}t + 2\kappa^2 \int_{Q_T} \left[\sqrt{q_N} D^2(\sqrt{q_N})-\nabla(\sqrt{q_N}) \otimes \nabla(\sqrt{q_N})\right] : D(\Phi) {\rm d}\mu_m {\rm d}t \notag\\
	&\quad  +\lambda \sigma^2 \int_{Q_T} \nabla q_N  \cdot \Phi {\rm d}\mu_m{\rm d}t + r_0 \int_{Q_T} u_N \cdot \Phi  {\rm d}\mu_m{\rm d}t+ r_1\int_{Q_T} q_N|u_N|^2u_N \cdot \Phi  {\rm d}\mu_m{\rm d}t \notag\\
	&\quad+ \frac{r_4}{\sigma^4}\int_{Q_T} q_N|x|^2x \cdot \Phi  {\rm d}\mu_m{\rm d}t. \label{eq:CNSK_system_m_regularized-velocity-WS-Definition_XN-proof2}
	\end{align}
	Let $\chi\in\mathcal{D}(0,T)$ be a test function only depending on time. We would like to apply \eqref{eq:CNSK_system_m_regularized-velocity-WS-Definition_XN-proof2} with 
	\begin{equation*}
	\Phi=\chi\nabla(\ln(q_N)),
	\end{equation*}
	but the regularity properties satisfied by $q_N$ (see Proposition \ref{prop:existence_FD} and Lemma \ref{lem:apriori}) are not enough to deduce that $\nabla(\ln(q_N))$ belongs to $L^2(0,T;W^{1,\infty}_{\mu_m})$. However we have that $\nabla(\ln(q_N))=\frac{\nabla q_N}{q_N}\in L^\infty(0,T;L^2_{\mu_m})$ by Proposition \ref{prop:existence_FD}, and we also observe that
	\begin{equation*}
	D^2(\ln(q_N))=\frac{D^2(q_N)}{q_N}-\frac{\nabla q_N\otimes \nabla q_N}{q_N^2}=\frac{D^2(q_N)}{q_N}-16\frac{\nabla q_N^{\frac14}\otimes \nabla q_N^{\frac14}}{\sqrt{q_N}} \in L^2(0,T;[L^2_{\mu_m}]^{d\times d}),
	\end{equation*}
	by Proposition \ref{prop:existence_FD} and Lemma \ref{lem:apriori}. For $\varepsilon>0$, let $\eta_\varepsilon$ be an approximation to the identity ($\eta_\varepsilon(x)=\varepsilon^{-d}\eta(\varepsilon^{-1}x)$ with $\eta\in\mathcal{D}(\mathbb{R}^d)$ and $\int_{\mathbb{R}^d}\eta(x){\rm d}x=1$), and let $\Phi_\varepsilon=\chi (\eta_\varepsilon*\nabla\ln(q_N))$. We then have that $\Phi_\varepsilon\in L^2(0,T;W^{1,\infty}_{\mu_m})$ and
	\begin{equation}\label{eq:limit_Phi_epsilon}
	\Phi_\varepsilon\to\chi\nabla\ln(q_N), \quad \nabla\Phi_\varepsilon=\chi(\eta_\varepsilon*D^2(\ln(q_N)))\to \chi D^2(\ln(q_N)), \quad \varepsilon\to0,
	\end{equation}
	in $L^2(0,T;[L^2_{\mu_m}]^{d})$ and $L^2(0,T;[L^2_{\mu_m}]^{d\times d})$, respectively. Now we apply \eqref{eq:CNSK_system_m_regularized-velocity-WS-Definition_XN-proof2} with $\Phi=\Phi_\varepsilon$, and we observe that all terms in  \eqref{eq:CNSK_system_m_regularized-velocity-WS-Definition_XN-proof2} (with $\Phi=\Phi_\varepsilon$) integrated against $\Phi_\varepsilon$, $D(\Phi_\varepsilon)$ or $\nabla(\Phi_\varepsilon)$ belong either to $L^2(0,T;[L^2_{\mu_m}]^{d})$ or $L^2(0,T;[L^2_{\mu_m}]^{d\times d})$ (for instance, the term $\partial_t(q_N u_N)=\partial_t(q_N)u_N+q_N\partial_t(u_N)$ belongs to $L^2(0,T;[L^2_{\mu_m}]^{d})$ because $q_N\in H^1(0,T;L^2_{\mu_m})$  and $u_N\in\mathcal{C}^1([0,T];L^\infty_{\mu_m})$ by Proposition \ref{prop:existence_FD}; the term $\sqrt{q_N} D^2(\sqrt{q_N})$ belongs to $L^2(0,T;[L^2_{\mu_m}]^{d\times d})$ because $q_N\in L^\infty(0,T;L^\infty_{\mu_m})$ by Proposition \ref{prop:existence_FD} and $D^2(\sqrt{q_N})\in L^2(0,T;[L^2_{\mu_m}]^{d\times d})$ by Lemma \ref{lem:apriori}; the term $q_N|x|^2x$ belongs to $L^2(0,T;[L^2_{\mu_m}]^{d})$ because $q_N\in L^\infty(0,T;L^\infty_{\mu_m})$ and $|x|^2x\in L^2(0,T;[L^2_{\mu_m}]^{d})$, \emph{etc.}). Therefore, since \eqref{eq:limit_Phi_epsilon} holds, we can let $\varepsilon\to0$, obtaining
	\begin{align}
	&  \int_{Q_T} \partial_t(q_N u_N) \cdot \chi\nabla\ln(q_N) {\rm d}\mu_m{\rm d}t \notag\\
	&= \int_{Q_T} \chi q_ND^2(\ln(q_N)) u_N \cdot u_N {\rm d}\mu_m{\rm d}t  -\delta_1 \int_{Q_T} \chi\nabla u_N \nabla q_N \cdot \nabla\ln(q_N) {\rm d}\mu_m{\rm d}t \notag\\
	&\quad + 2\kappa^2 \int_{Q_T} \chi\left[\sqrt{q_N} D^2(\sqrt{q_N})-\nabla(\sqrt{q_N}) \otimes \nabla(\sqrt{q_N})\right] : D^2(\ln(q_N)) {\rm d}\mu_m {\rm d}t \notag\\
	&\quad+ 2\nu \int_{Q_T} \chi q_N D(u_N) : D^2(\ln(q_N)) {\rm d}\mu_m{\rm d}t  +\lambda \sigma^2 \int_{Q_T} \chi\nabla q_N  \cdot \nabla\ln(q_N) {\rm d}\mu_m{\rm d}t  \notag \\
	&\quad + r_0 \int_{Q_T} \chi u_N \cdot \nabla\ln(q_N)  {\rm d}\mu_m{\rm d}t+ r_1\int_{Q_T} \chi q_N|u_N|^2u_N \cdot \nabla\ln(q_N)  {\rm d}\mu_m{\rm d}t \notag\\
	&\quad+ \frac{r_4}{\sigma^4}\int_{Q_T} \chi q_N|x|^2x \cdot \nabla\ln(q_N)  {\rm d}\mu_m{\rm d}t. \label{eq:CNSK_system_m_regularized-velocity-WS-Definition_XN-proof3}
	\end{align}
	Similar computations as in \cite[Appendix A]{Carles_Carrapatoso_Hillairet_AIF2022} then lead to
	\begin{align}
	&   \int_{Q_T} q_N u_N \cdot \left(\partial_t(\chi\nabla\ln(q_N)) +\chi D^2(\ln(q_N)) u\right) {\rm d}\mu_m{\rm d}t    -16 \delta_1 \int_{Q_T} \chi \sqrt{q_N}\nabla u_N \nabla \left(q_N^{\frac14}\right) \cdot \nabla \left(q_N^{\frac14}\right) {\rm d}\mu_m{\rm d}t \notag \\
	&= 2\nu \int_{Q_T} \chi  (\sqrt{q_N} D(u_N)) : (\sqrt{q_N}D^2(\ln(q_N))) {\rm d}\mu_m{\rm d}t +  \kappa^2 \int_{Q_T} \chi q_N|D^2(\ln(q_N))|^2 {\rm d}\mu_m {\rm d}t \notag \\
	&\quad +\lambda \sigma^2  \int_{Q_T} \chi q_N|\nabla q_N|^2{\rm d}\mu_m{\rm d}t + 4r_0 \int_{Q_T} \chi u_N \cdot \nabla\ln(q_N) {\rm d}\mu_m{\rm d}t \notag \\
	&\quad+ 4r_1\int_{Q_T}\chi \left(q_N^{\frac34}|u_N|^2u_N\right) \cdot \nabla \left(q_N^{\frac14}\right)  {\rm d}\mu_m{\rm d}t + \frac{4r_4}{\sigma^4}\int_{Q_T}\chi \left(q_N^{\frac34}|x|^2x\right)  \cdot \nabla \left(q_N^{\frac14}\right)   {\rm d}\mu_m{\rm d}t. \label{eq:BD2}
	\end{align}
	In order to justify the time integration by parts to deduce the last equation from \eqref{eq:CNSK_system_m_regularized-velocity-WS-Definition_XN-proof3}, we write $\partial_t(q_N u_N)=\partial_t(q_N)u_N+q_N\partial_t( u_N)$ and integrate by parts the second term thanks to \eqref{eq:q_nabla_log_q_equation} (observing that $q_N\in L^\infty(0,T;L^\infty_{\mu_m})$), yielding
	\begin{align*}
	\int_{Q_T} \partial_t(q_N u_N) \cdot \chi\nabla\ln(q_N) {\rm d}\mu_m{\rm d}t&=\int_{Q_T} u_N \cdot \chi \partial_t(q_N)\nabla\ln(q_N) {\rm d}\mu_m{\rm d}t-\int_{Q_T}  u_N \cdot  \partial_t(\chi q_n\nabla\ln(q_N)) {\rm d}\mu_m{\rm d}t \\
	&=-\int_{Q_T}  u_N \cdot q_N\partial_t(\chi\nabla\ln(q_N)) {\rm d}\mu_m{\rm d}t.
	\end{align*}
	Next, using the differentiation in $\mathcal{D}'(0,T)$ (and denoting $\left\langle \cdot,\cdot \right\rangle$ the duality bracket on $\mathcal{D}'(0,T)$), we can rewrite
	\begin{align}
	&\int_{Q_T} q_N u_N \cdot \left(\partial_t(\chi\nabla\ln(q_N)) +D^2(\ln(q_N)) u_N\right) {\rm d}\mu_m{\rm d}t \notag \\
	& =  - \left\langle \frac{\rm d}{{\rm d} t} \left(\int_{\mathbb{R}^d} q_N u_N \cdot \nabla \ln(q_N)    \;{\rm d}\mu_m\right), \chi \right\rangle 
	+ \int_{Q_T} \chi  q_N u_N \cdot \left(\partial_t\nabla \ln(q_N) +D^2(\ln(q_N)) u_N\right)  {\rm d}\mu_m  {\rm d}t. \label{eq:BD3}
	\end{align}
	The last term in the previous identity can be computed using \eqref{eq:nabla_log_q_equation} to obtain, almost everywhere in $(0,T)$,
	\begin{align}
	& \int_{\mathbb{R}^d}  q_N u_N \cdot \left(\partial_t\nabla \ln(q_N) +D^2(\ln(q_N)) u_N\right)   \;{\rm d}\mu_m \notag \\
	& =  \int_{\mathbb{R}^d}  u_N \cdot \big(D^2(\ln(q_N)) \delta_1\nabla \ln(q_N)+  \delta_1 \dvm(q_ND^2\ln(q_N)) - \dvm(q_N\nabla u_N^\top) + \frac{q_N}{\sigma^2}(u_N - \delta_1\nabla \ln(q_N)) \big)  {\rm d}\mu_m \notag \\
	& = \int_{\mathbb{R}^d} q_N \nabla u_N : \nabla u_N^\top \;{\rm d}\mu_m - \delta_1 \int_{\mathbb{R}^d} q_N |D^2(\ln(q_N))|^2\;{\rm d}\mu_m +\frac1{\sigma^2}\int_{\mathbb{R}^d} (\sqrt{q_N}u_N-2\delta_1\nabla(\sqrt{q_N})) \cdot (\sqrt{q_N}u_N) {\rm d}\mu_m \notag \\
	& \quad + \delta_1 \int_{\mathbb{R}^d} (q_N^{\frac14}u_N) \cdot (\sqrt{q_N}D^2(\ln(q_N)))(4\nabla q_N^{\frac14})  \;{\rm d}\mu_m . \label{eq:BD4}
	\end{align}	
	Inserting \eqref{eq:BD3}--\eqref{eq:BD4} into \eqref{eq:BD2} and using, in particular, that
	\begin{align*}
	r_0 \int_{\mathbb{R}^d} u_N \cdot \nabla \ln(q_N) \;{\rm d}\mu_m & = - r_0 \frac{\rm d}{{\rm d}t} \int_{\mathbb{R}^d} \ln(q_N) \;{\rm d}\mu_m +r_0\delta_1 \int_{\mathbb{R}^d} |\nabla \ln(q_N)|^2 \;{\rm d}\mu_m \\
	& =  r_0 \frac{\rm d}{{\rm d}t} \int_{\mathbb{R}^d} (q_N-\ln(q_N)) \;{\rm d}\mu_m +r_0\delta_1 \int_{\mathbb{R}^d} |\nabla \ln(q_N)|^2 \;{\rm d}\mu_m,
	\end{align*}
	(where we used the conservation of mass $\int_{\mathbb{R}^d} q_N \;{\rm d}\mu_m = 1$),  we obtain \eqref{eq:produit_scalaire_u_nabla_q-proof}. 
\end{proof}

Arguing in a similar way, we can establish the following:

\begin{Lemma}\label{lm:BD2}
	Under the conditions of Proposition \ref{prop:FG-convergence}, let $(q,u)$ be a solution to \eqref{eq:CNSK_system_m_regularized} associated to a sequence $((q_N,u_N))_N$ as in Proposition \ref{prop:FG-convergence}. Then, for almost every $t\in(0,T)$, 
	\begin{align}
	&  \frac12 \frac{{\rm d}}{{\rm d}t} \int_{\mathbb{R}^d} q_N |\nabla(\ln(q_N))|^2 \,{\rm d}\mu_m +\delta_1 \int_{\mathbb{R}^d} q_N |D^2(\ln(q_N))|^2 \;{\rm d}\mu_m  \notag \\
	& =  \int_{\mathbb{R}^d} q_N\nabla u_N ^\top : D^2(\ln(q_N)) \;{\rm d}\mu_m  + \frac{1}{\sigma^2}\int_{\mathbb{R}^d} \left(q_Nu_N - \delta_1\nabla\ln(q_N)\right) \cdot \nabla \ln(q_N)  \;{\rm d}\mu_m. \label{eq:produit_scalaire_nabla_q_nabla_q}
	\end{align}
\end{Lemma}

\begin{proof}
	It suffices to proceed in the same way as in the proof of Lemma \ref{lm:BD1}. 
\end{proof}

Now we can prove Proposition \ref{prop:BD}.

\begin{proof}[Proof of Proposition \ref{prop:BD}]
	Adding $\eqref{eq:energy_identity_reg}$ for $(q_N,u_N)$ (see Lemma \ref{lem:FG-energy}), $(2\nu) \times \eqref{eq:produit_scalaire_u_nabla_q-proof}$ and $(2\nu)^2 \times \eqref{eq:produit_scalaire_nabla_q_nabla_q}$, straightforward computations lead to the following equality satisfied by the BD entropy:
	\begin{equation}\label{eq:BD-entropy_equality}
	\frac{{\rm d}}{{\rm d}t} \mathcal{E}_{\rm BD, reg}(q_N,u_N)(t) +\mathcal{D}_{\rm BD, reg}(q_N,u_N) (t) = \mathcal{R}_{\rm BD, reg}(q_N,u_N)(t),
	\end{equation}
	where $\mathcal{E}_{\rm BD, reg}(q_N,u_N)$, $\mathcal{D}_{\rm BD, reg}(q_N,u_N)$ (with $(q_N,u_N)$ instead of $(q,u)$) are defined in \eqref{def:BD_reg-energy}--\eqref{def:BD_reg-dissipation} and  (compare to \eqref{def:BD_reg-RHS_intro})
	\begin{align} 
	\mathcal{R}_{\rm BD, reg}(q,u) & = \frac{r_4(\delta_1+2\nu)(d+2)}{\sigma^2}\int_{\mathbb{R}^d} q \left|\frac{x}{\sigma}\right|^2    \;{\rm d}\mu_m -2\nu \delta_1 \int_{\mathbb{R}^d} q D(u) : D^2(\ln(q))  \;{\rm d}\mu_m \notag\\
	&\quad - 2\nu \delta_1 \int_{\mathbb{R}^d} \nabla u \nabla q \cdot \nabla\ln(q) \;{\rm d}\mu_m - 2\nu r_1\int_{\mathbb{R}^d} q|u|^2 u \cdot \nabla \ln(q)  \;{\rm d}\mu_m \notag \\
	&\quad + \frac{2\nu}{\sigma^2}\int_{\mathbb{R}^d} q\left(u+2\nu\nabla\ln(q)\right) \cdot (u -\delta_1\nabla \ln(q)) \;{\rm d}\mu_m.
	\label{def:BD_reg-RHS}
	\end{align}
	From the expression of $\mathcal{R}_{\rm BD, reg}$ and the a priori bounds on $(q_N,u_N)$ stated in Lemma \ref{lem:apriori}, it is not difficult to verify, arguing as in the proof of Proposition \ref{prop:FG-convergence}, that there exists a constant $C_{\rm BD} >0$ (independent of $\delta_1$) and a positive constant $C$ independent of the parameters such that
	\begin{equation}\label{eq:estim_RBD}
	\mathcal{R}_{\rm BD,reg}(q_N,u_N) \leq C_{\rm BD} + Cr_4 \|\nabla(q_N^{\frac14})\|_{L^4_{\mu_m}}.
	\end{equation}
	Now from the expression \eqref{def:BD_reg-dissipation} of $\mathcal{D}_{\rm BD, reg}$ we see that
	\begin{equation*}
	2\nu\kappa^2 \int_{\mathbb{R}^d} q_N |D^2(\ln(q_N))|^2 + \frac{\nu r_4}{2\sigma^4}I_4(q_N) \le \mathcal{D}_{\rm BD, reg}(q_N,u_N),
	\end{equation*}
	and therefore, by Lemma \ref{lem:Hessian_estimate}, for a constant $C>0$ independent of $r_4$
	\begin{equation*}
	C\min(r_4,4\kappa^2)\left[ \|D^2(\sqrt{q_N})\|_{L^2_{\mu_m}}^2 + \|\nabla (q_N^{\frac14})\|_{L^4_{\mu_m}}^4\right] \leq \mathcal{D}_{\rm BD, reg}(q_N,u_N).
	\end{equation*}		
	Together with \eqref{eq:BD-entropy_equality} and \eqref{eq:estim_RBD}, this yields
	\begin{equation*}
	\frac{{\rm d}}{{\rm d}t} \mathcal{E}_{\rm BD, reg}(q_N,u_N)(t) +\frac12\mathcal{D}_{\rm BD, reg}(q_N,u_N) (t) \le C_{\rm BD},
	\end{equation*}
	for some positive constant $C_{\mathrm{BD}}$ (which differs from that of \eqref{eq:estim_RBD}) independent of $\delta_1$.
	Integrating from $0$ to $t$, we obtain \eqref{eq:EBD_ineq} with $(q_N,u_N)$ instead of $(q,u)$. We conclude the proof of the proposition by letting $N\to\infty$ exactly as in the proof of Proposition \ref{prop:En}. 
\end{proof}

To conclude this subsection, we observe that, putting together the previous propositions, the proof of Theorem \ref{prop:GWS_CNSK_system_m_regularized} in the case where $\delta_1>0$ is now complete.

\begin{proof}[Proof of Theorem \ref{prop:GWS_CNSK_system_m_regularized} for $\delta_1>0$]
	Let $(q,u)$ be a solution to \eqref{eq:CNSK_system_m_regularized}, associated to a sequence $((q_N,u_N))_N$, as in Proposition \ref{prop:FG-convergence}. Since $(\sqrt{q_N})_{N}$ strongly converges to $\sqrt{q}$ in $\mathcal{C}^0([0,T];L^2_{\mu_m})$, the conservation of mass \eqref{eq:conservation_mass_thm_regularized} for $q$ is a direct consequence of the conservation of mass for $q_N$ (see \eqref{eq:conservation-mass-qN}).

	The regularity properties for $(q,u)$ stated in Definition \ref{def:WeakSolutionRegularizedEquations}  directly follow from the conservation of mass \eqref{eq:conservation_mass_thm_regularized}, the energy estimate \eqref{eq:energy_estimates_GWS_CNSK_system_m_regularized} proven in Proposition \ref{prop:En} and the BD entropy estimate \eqref{eq:BD_estimates_GWS_CNSK_system_m_regularized} proven in Proposition \ref{prop:BD}.
\end{proof}

\subsubsection{Passing to the limit $\delta_1 \to 0$}

Starting with a solution $(q_{\delta_1},u_{\delta_1})$ (with $\delta_1>0$) to the regularized QNS system \eqref{eq:CNSK_system_m_regularized} and letting $\delta_1\to0$, we now obtain the statement of Theorem \ref{prop:GWS_CNSK_system_m_regularized} in the case where $\delta_1=0$.

\begin{proof}[Proof of Theorem \ref{prop:GWS_CNSK_system_m_regularized} for $\delta_1=0$]
	Let $(q_{\delta_1},u_{\delta_1})$ be a solution to \eqref{eq:CNSK_system_m_regularized} with $\delta_1>0$ constructed as in Proposition \ref{prop:FG-convergence}. Similarly to the a priori estimates obtained for the Faedo-Galerkin solutions $(q_N,u_N)$ in Lemma \ref{lem:apriori}, the energy estimate \eqref{eq:energy_estimates_GWS_CNSK_system_m_regularized}, BD entropy estimate \eqref{eq:BD_estimates_GWS_CNSK_system_m_regularized} and Lemma \ref{lem:Hessian_estimate} imply the following bounds for $(q_{\delta_1},u_{\delta_1})$: 
	\begin{align*}
	\|q_{\delta_1}\|_{L^\infty(0,T;L^1_{\mu_m})}+\|q_{\delta_1} \ln(q_{\delta_1})\|_{L^\infty(0,T;L^1_{\mu_m})}+r_0\|q_{\delta_1}-\ln(q_{\delta_1})\|_{L^\infty(0,T;L^1_{\mu_m})}&\le C, \\
	\|\sqrt{q_{\delta_1}}u_{\delta_1}\|_{L^\infty(0,T;[L^2_{\mu_m}]^d)} +  \|\nabla\sqrt{q_{\delta_1}}\|_{L^\infty(0,T;[L^2_{\mu_m}]^d)} + \sqrt{r_0} \|u_{\delta_1}\|_{L^2(0,T;[L^2_{\mu_m}]^d)}&\le C, \\
	\|\sqrt{q_{\delta_1}} \nabla u_{\delta_1}\|_{L^2(0,T;[L^2_{\mu_m}]^{d\times d})}+  \| \sqrt{q_{\delta_1}} D^2(\ln(q_{\delta_1}))\|_{L^2(0,T;[L^2_{\mu_m}]^{d\times d})}+\sqrt{r_4} \|D^2(\sqrt{q_{\delta_1}})\|_{L^2(0,T;[L^2_{\mu_m}]^{d\times d})}&\le C , \\
	r_4^{\frac14} \| q_{\delta_1}^{\frac14}x \|_{L^\infty(0,T;[L^4_{\mu_m}]^d)} + r_1^{\frac14}\|q_{\delta_1}^{\frac14}u\|_{L^4(0,T;[L^4_{\mu_m}]^d)}+ r_4^{\frac14} \|\nabla{q_{\delta_1}^{\frac14}}\|_{L^4(0,T;[L^4_{\mu_m}]^d)}&\le C ,
	\end{align*}
for some constant $C$ independent of $0<r_0,r_1,r_4,\delta_1\le1$.  In particular, we see that $(q_{\delta_1},u_{\delta_1})$ satisfy, uniformly in $0<\delta_1\le1$, the same a priori estimates that were satisfied by $(q_N,u_N)$ uniformly in $N$ (see Lemma \ref{lem:apriori}). Here it should be noted that the BD entropy estimate \eqref{eq:BD_estimates_GWS_CNSK_system_m_regularized} allows us to the refine the bounds on $\sqrt{q_{\delta_1}} D^2(\ln(q_{\delta_1}))$ and $q_{\delta_1}^{\frac14}x$ and therefore, by Lemma \ref{lem:Hessian_estimate}, those on $D^2(\sqrt{q_{\delta_1}})$ and $\nabla(q_{\delta_1}^{\frac14})$, in the sense that the corresponding bounds for $(q_N,u_N)$ in Lemma \ref{lem:apriori}, which depended on $\delta_1$, are replaced by uniform bounds in $\delta_1$ for $(q_{\delta_1},u_{\delta_1})$.
	
	Given these uniform a priori estimates for $(q_{\delta_1},u_{\delta_1})$, we can now reproduce the argument used in Proposition \ref{prop:FG-convergence} to prove the convergence of the Faedo-Galerkin solutions $(q_N,u_N)$ in the limit $N\to\infty$. The gives the existence of global weak solutions to \eqref{eq:CNSK_system_m_regularized} for $\delta_1=0$. The conservation of mass, energy estimate and BD entropy estimate are then obtained by following the argument used in the proof of Proposition \ref{prop:En}.
\end{proof}

\section{Global weak solutions to the QNS system without drag terms} \label{sec:3}

The goal of this section is to prove Theorem \ref{thm:main}, using in particular the result established at the end of the previous section, namely Theorem \ref{prop:GWS_CNSK_system_m_regularized} with $\delta_1=0$. To fix the ideas, we only consider here the most difficult case in the statement of Theorem \ref{thm:main}, namely when all the coefficients in front of the drag forces vanish, $r_0=r_1=r_4=0$. One can easily adapt the proof to the general case where some of these coefficients are non-negative. We recall that $\nu >0$ and $\kappa >0$ are fixed.

Our aim is to construct global weak solutions to the system \eqref{eq:CNSK_system_m} which is, as mentioned in the introduction, a reformulation of the system \eqref{eq:CNSK_system}, the system of interest in this paper. To obtain such a result, we will use the notion of renormalized weak solutions introduced in \cite{Lacroix-Violet_Vasseur_JMPA18} and used in particular in \cite[Section 4]{Carles_Carrapatoso_Hillairet_AIF2022}. More precisely, we will prove Theorem \ref{th:solutionwithoutdragterms} below which in turn implies Theorem \ref{thm:main}. The proof is divided in several steps. Firstly, we show that there exist renormalized weak solutions to the QNS system with drag forces \eqref{eq:CNSK_system_m_regularized_intro}. We use for that the result obtained in the previous section by proving that a weak solution is also a renormalized solution. Secondly, we show that any renormalized solution of \eqref{eq:CNSK_system_m} is also a weak solution. Finally, we prove a stability result which gives that a sequence of renormalized weak solutions of \eqref{eq:CNSK_system_m_regularized_intro}, converges, up to a subsequence, to a renomalized weak solution of \eqref{eq:CNSK_system_m}, when the drag coefficients tend to zero.

Let us now define what is called a renormalized weak solution for the system \eqref{eq:CNSK_system_m_regularized_intro}.

\begin{Definition}(Renormalized weak solution \cite[Section 4]{Carles_Carrapatoso_Hillairet_AIF2022}) \label{definition:RenormalizedWeakSolution}
	Let $r_0 \geq 0$, $r_1 \geq 0$, $r_4\geq 0$, $\kappa > 0$ and $\nu >0$. Let $(\sqrt{q^0},\sqrt{q^0}u^0) \in H^1_{\mu_m} \times \left[L^2_{\mu_m}\right]^d$ satisfying :
	\begin{equation}\label{eq:initial_data_compatibility}
	\int_{\mathbb{R}^d} (\sqrt{q^0})^2 \;{\rm d}\mu_m = 1 \quad \mbox{and (compatibility conditions) } \sqrt{q^0}  \geq 0 \mbox{ a.e. in }\mathbb{R}^d, \mbox{  } \sqrt{q^0} u^0 = 0 \mbox{ a.e. on } \left\{\sqrt{q^0} = 0 \right\}.
	\end{equation}
	A couple $(\sqrt{q},\sqrt{q}u)$ is called a (global) renormalized weak solution to \eqref{eq:CNSK_system_m_regularized_intro} in $\mathbb{R}^d$ associated with the initial data $(\sqrt{q^0},\sqrt{q^0}u^0)$ if there exists a quadruplet $(\sqrt{q},\sqrt{q}u, \mathbf{T}_{\rm NS},\mathbf{S}_{\rm K})$ such that
	
	\begin{enumerate}
		\item The same regularity properties as the ones stated in Definition \ref{def:WeakSolution_drag_term_intro} are satisfied.
		\item There exists a constant $C >0$ (depending on the initial data $(\sqrt{q^0},\sqrt{q^0}u^0)$) such that, for any function $\Theta \in W^{2,\infty}(\mathbb{R}^d)$, there exist two measures $\mathfrak{m}_{\Theta}^{1} \in \mathcal{M}((0,\infty)\times \mathbb{R}^d;\mathbb{R})$ and $\mathfrak{m}_{\Theta}^{2} \in \mathcal{M}((0,\infty)\times \mathbb{R}^d;\mathbb{R}^{d\times d \times d})$ such that  
		\begin{equation*}
		\|\mathfrak{m}_\Theta^1\|_{\mathcal{M}((0,\infty)\times \mathbb{R}^d;\mathbb{R})} + \|\mathfrak{m}_\Theta^2\|_{\mathcal{M}((0,\infty)\times \mathbb{R}^d;\mathbb{R}^{d\times d \times d})} \leq C \|D^2(\Theta)\|_{L^\infty}
		\end{equation*}
		and the following equations are satisfied in $\mathcal{D}'((0,\infty)\times \mathbb{R}^d)$ and $[\mathcal{D}'((0,\infty)\times \mathbb{R}^d)]^d$, respectively,
		\begin{subequations}
			\label{eq:CNSK_system_m_RenormalizedWeakSolution}
			\begin{eqnarray}
			\partial_t {\sqrt{q}} + \dvm({\sqrt{q}}u) & = & \frac12 \tr(\mathbf{T}_{\rm NS}) - \frac1{2\sigma^2} \sqrt{q}u \cdot x, \label{eq:CNSK_system_m_RenormalizedWeakSolution-density} \\
			\partial_t(\sqrt{q}\sqrt{q}\Theta(u)) + \dvm\left({\sqrt{q}}\Theta(u) {\sqrt{q}}u \right) & = & \dvm\left(\left[2\nu {\sqrt{q}}\mathbf{S}_{\rm NS} + 2\kappa^2{\sqrt{q}} \mathbf{S}_{\rm K}\right](\nabla\Theta)(u) \right)  \label{eq:CNSK_system_m_RenormalizedWeakSolution-velocity} \\
			& & 
			 + \mathbf{f} \cdot (\nabla\Theta)(u) + \mathfrak{m}_\Theta^1 \notag
			\end{eqnarray}
		\end{subequations}
			where 
			\begin{equation}\label{def:f}
				\mathbf{f} = -  r_0 u - r_1 q|u|^2u-\dfrac{r_4}{\sigma^4}q|x|^2 x- 2\sqrt{q}\lambda\sigma^2\nabla \sqrt{q},
			\end{equation}
		and with $\mathbf{S}_{\rm NS} = \frac12\left(\mathbf{T}_{\rm NS} + \mathbf{T}_{\rm NS}^\top \right)$, the symmetric part of $\mathbf{T}_{\rm NS}$, satisfying the compatibility conditions:
		\begin{eqnarray*}
			{\sqrt{q}} \frac{\partial \Theta}{\partial x_i}(u) [\mathbf{T}_{\rm NS}]_{j,k} & = & \partial_k\left(q \frac{\partial \Theta}{\partial x_i}(u) u_j\right) - 2 {\sqrt{q}} u_j \frac{\partial \Theta}{\partial x_i}(u) \partial_k {\sqrt{q}} + (\mathfrak{m}_\Theta^2)_{i,j,k} \qquad \mbox{for all } i,j,k = 1, 2, 3, \\
			{\sqrt{q}}\mathbf{S}_{\rm K} & = &  {\sqrt{q}}D^2({\sqrt{q}}) - \nabla {\sqrt{q}} \otimes \nabla {\sqrt{q}} = \frac{1}{\rho_m}\left[\sqrt{\rho}D^2(\sqrt{\rho}) - \nabla(\sqrt{\rho}) \otimes \nabla (\sqrt{\rho}) + \frac\rho{2\sigma^2}{\rm I}_d\right].
		\end{eqnarray*}
		\item For any $\psi \in \mathcal{D}(\mathbb{R}^d)$,
		\begin{eqnarray*}
			\underset{t\to 0^+}{\lim} \int_{\mathbb{R}^d} \sqrt{q}(t,\cdot) \psi\,{\rm d}\mu_m & = & \int_{\mathbb{R}^d} \sqrt{q^0}\psi \,{\rm d}\mu_m, \\
			\underset{t\to 0^+}{\lim} \int_{\mathbb{R}^d} \sqrt{q}(t,\cdot) (\sqrt{q}u)(t,\cdot)\psi\,{\rm d}\mu_m & = & \int_{\mathbb{R}^d} \sqrt{q^0}(\sqrt{q^0}u^0)\psi \,{\rm d}\mu_m.
		\end{eqnarray*}
	\end{enumerate}	
\end{Definition}

Let us now state the main theorem of this section.  

\begin{Theorem}\label{th:solutionwithoutdragterms}
	Let $(\sqrt{q^0},\sqrt{q^0}u^0) \in H^1_{\mu_m} \times \left[L^2_{\mu_m}\right]^d$ satisfying \eqref{eq:initial_data_compatibility}. 
	\begin{enumerate}
		\item For any $r_0 > 0$, $r_1 > 0$ and $r_4> 0$, any weak solution to \eqref{eq:CNSK_system_m_regularized_intro} in the sense of Definition \ref{def:WeakSolution_drag_term_intro} is a renormalized solution to \eqref{eq:CNSK_system_m_regularized_intro} in the sense of Definition \ref{definition:RenormalizedWeakSolution}, with initial data $(\sqrt{q^0},\sqrt{q^0}u^0)$. 
		\item Any renormalized solution to \eqref{eq:CNSK_system_m_regularized_intro} in the sense of Definition \ref{definition:RenormalizedWeakSolution} is a weak solution to \eqref{eq:CNSK_system_m_regularized_intro} in the sense of Definition \ref{def:WeakSolution_drag_term_intro}.
		\item

		Consider the sequences 
		$(\sqrt{q^0_n},\sqrt{q^0_n}u_n^0)$ uniformly bounded in $H^1_{\mu_m} \times \left[L^2_{\mu_m}\right]^d$, satisfying \eqref{eq:initial_data_compatibility} and converging to $(\sqrt{q^0},\sqrt{q^0}u^0)$, and the
		 sequences $r_{0,n}$, $r_{1,n}$ and $r_{4,n}$ converging to $0$, all  defined in Section \ref{subsec:smoothinitialdata} below.

		Denote $(\sqrt{q_n},\sqrt{q_n}u_n)$ the sequence of associated weak renormalized solution to \eqref{eq:CNSK_system_m_regularized_intro} (associated with $({\mathbf{T}}_{{\rm NS},n},\mathbf{S}_{{\rm K},n})$), with initial data $(\sqrt{q^0_n},\sqrt{q^0_n}u_n^0)$ and drag terms $r_{0,n}$, $r_{1,n}$ and $r_{4,n}$. Then there exist a subsequence $(\sqrt{q_n},\sqrt{q_n}u_n)$ (still denoted with $n$), tensors ${\mathbf{T}}_{\rm NS}$, $\mathbf{S}_{\rm K}$ in $L^2_{\rm loc}(0,\infty;L^2_{\mu_m})$ and $(\sqrt{q},\sqrt{q}u)$ a renormalized solution to  \eqref{eq:CNSK_system_m} (associated with $({\mathbf{T}}_{\rm NS},\mathbf{S}_{\rm K})$), with initial data $(\sqrt{q^0},\sqrt{q^0}u^0)$, such that
		\begin{enumerate}
			\item[-] $\sqrt{q_n}$ converges to $\sqrt{q}$ in $L^\infty_{\rm loc}(0,\infty;H^1_{\mu_m})$;
			\item[-] $\sqrt{q_n}u_n$ converges to $\sqrt{q}u$ in $L^\infty_{\rm loc}(0,\infty;L^2_{\mu_m})$;
			\item[-] the sequences $\mathbf{T}_{{\rm NS},n}$ and $\mathbf{S}_{{\rm K},n}$ converge weakly in $L^2_{\rm loc}(0,\infty;L^2_{\mu_m})$ to $\mathbf{T}_{\rm NS}$ and $\mathbf{S}_{\rm K}$, respectively.
		\end{enumerate}
		Moreover for every compactly supported function $\Theta \in W^{2,\infty}(\mathbb{R}^d)$, the sequence $\sqrt{q_n}\Theta(u_n)$ converges strongly in $L^\infty_{\rm loc}(0,\infty;L^2_{\mu_m})$ to $\sqrt{q}\Theta(u)$.  
	\end{enumerate}
\end{Theorem}

\begin{Remark}\label{rem:limit_kappa_0}
As in \cite{Carles_Carrapatoso_Hillairet_AIF2022}, it would also be possible to let $\kappa\to0$ in our strategy of proof. As $\sigma$ (and thus $\rho_m$) depends on $\kappa$ by \eqref{def:sigma}, it is convenient to do a change of unknows and variables, for $\theta = \frac{\sigma^2}{\eps} >0$ (a universal time where $\eps >0$ is fixed),
\begin{equation*}
\rho(t,x) = \frac{1}{\sigma^d}R\left(\frac{t}{\theta}, \frac{x}{\sigma}\right), \qquad u(t,x) = \frac{\sigma}{\theta}U\left(\frac{t}{\theta}, \frac{x}{\sigma}\right)
\end{equation*}
to obtain the following equivalent system, where $\tau = \frac{t}{\theta}$, $y = \frac{x}{\sigma}$,
\begin{eqnarray*}
\partial_\tau R + {\dv}_y(RU) & = & 0, \\
R\left(\partial_\tau U + \nabla_yUU\right) & = & \frac{2\nu \theta}{\sigma^2}{\dv}_y(RD_yU) + \frac{\kappa^2\theta^2}{\sigma^4} {\dv}_y\left(R D_y^2\ln(R)\right) - \frac{a\theta^2}{\sigma^2}\nabla_y\ln(R) -\lambda\theta^2 R \nabla_y\ln(R) \\
& = & \frac{2\nu}{\eps}{\dv}_y(RD_yU) + \frac{\kappa^2}{\eps^2} {\dv}_y\left(R D_y^2\ln(R)\right) - \frac{a\sigma^2}{\eps^2}\nabla_y\ln(R) - \frac{\lambda \sigma^4}{\eps^2}R \nabla_y\ln(R).
\end{eqnarray*}
Note that $\rho_m$ then becomes $R_m : x \in \mathbb{R}^d \mapsto (2\pi)^{-\frac d2}\exp\left(-\frac{|x|^2}{2}\right) \in \mathbb{R}$, and that we can then easily track the dependence in $(\kappa_n)$, a sequence of positive real numbers converging to $0$ (observe that, by \eqref{def:sigma}, $(\sigma_n)$ tends to $\sqrt{\frac{a}{\lambda}}$ when $(\kappa_n)$ tends to $0$).
\end{Remark}

As mentioned before Theorem \ref{th:solutionwithoutdragterms} implies Theorem \ref{thm:main} in the case where the drag coefficients $r_0$, $r_1$, $r_4$ vanish. Indeed, the first point allows us to construct, for a fixed $n$, a renormalized solution to the system \eqref{eq:CNSK_system_m_regularized_intro} with drag coefficients $r_{0,n}, r_{1,n}$ and $r_{4,n}$. This gives the existence of a sequence of renormalized solutions. Thanks to the third point we obtain the existence of a renormalized weak solution for the system without drag forces \eqref{eq:CNSK_system_m}, which is also a global weak solution thanks to the second point of the theorem. 

Note that, the proof of the second point of Theorem \ref{th:solutionwithoutdragterms} is not difficult and follows the lines of \cite[Section 4]{Lacroix-Violet_Vasseur_JMPA18}. The details will be skipped here.

The remainder of this section is devoted to the proof of Theorem \ref{th:solutionwithoutdragterms} and divided into three subsections. In Subsection \ref{subsec:weaktorenormalized} we prove the first point of Theorem \ref{th:solutionwithoutdragterms}. In Subsection \ref{subsec:smoothinitialdata}, we construct a smooth sequence of initial data. Finally, Subsection \ref{subsec:stability} is devoted to the proof of the stability part, {\it i.e.} the third point of Theorem \ref{th:solutionwithoutdragterms}.

\subsection{Proof of the first point of Theorem \ref{th:solutionwithoutdragterms}: Weak solutions are renormalized weak solutions.}\label{subsec:weaktorenormalized}

The goal of this section is to prove that, for $r_0 >0$, $r_1 >0$ and $r_4 >0$, any weak solution of \eqref{eq:CNSK_system_m_drag-term_WeakSolution} in the sense of Definition~\ref{def:WeakSolution_drag_term_intro} is also a renormalized weak solutions of \eqref{eq:CNSK_system_m_drag-term_WeakSolution} in the sense of Definition~\ref{definition:RenormalizedWeakSolution}. 

Note that the regularities and the equation \eqref{eq:CNSK_system_m_RenormalizedWeakSolution-density} of Definition \ref{definition:RenormalizedWeakSolution} are the same as those of Definition \ref{def:WeakSolution_drag_term_intro}. Likewise, point $3$ concerning the initial data is the same in both definitions. 
It thus remains to prove that Eq. \eqref{eq:CNSK_system_m_RenormalizedWeakSolution-velocity} is also satisfied for all $\Theta \in W^{2,\infty}(\mathbb{R}^d)$.

\paragraph{Equation for $\phi(q)$.}

Equation \eqref{eq:CNSK_system_m_drag-term_WeakSolution-density} rewrites as
\begin{equation}\label{eq:EDP_sqrt{q}}
\partial_t(\sqrt{q}) + 2(q^\frac{1}{4}u) \cdot \nabla(q^{\frac14}) + \frac{1}{2}\tr({\mathbf{T}}_{\rm NS}) - \frac12 \left(q^{\frac14}\frac{x}{\sigma^2}\right) \cdot \left(q^{\frac14}u\right) = 0.
\end{equation}
Then $\partial_t(\sqrt{q}) \in L^2(0,T;L^2_{\mu_m})$ thanks to the regularities of $\sqrt{q}$, $q^{\frac14}u$, $q^{\frac14}x$, $\nabla(q^{\frac14})$ and $\mathbf{T}_{\rm NS}$ in Definition \ref{def:WeakSolution_drag_term_intro}. We also know from the conservation of mass that $\sqrt{q} \in L^\infty(0,T;L^2_{\mu_m})$.

For any $\phi \in C_c^1(\mathbb{R};]0,\infty[)$ (a cut-off of small and large values), 
from the equation \eqref{eq:EDP_sqrt{q}}, $\phi(q) = \phi((\sqrt{q})^2)$ satisfies
\begin{equation*}
\partial_t(\phi(q)) + u \cdot \nabla(\phi(q)) = -\sqrt{q} \phi'(q) \tr(\mathbf{T}_{\rm NS}) + \phi'(q) \left(\sqrt{q}\frac{x}{\sigma^2}\right) \cdot \left(\sqrt{q}u\right).
\end{equation*}
Moreover, we can multiply the momentum equation \eqref{eq:CNSK_system_m_drag-term_WeakSolution-velocity} with $\phi(q)$ to obtain
\begin{equation}\label{eq:equation_phi(q)_u}
\partial_t(q\phi(q) u)+\dvm\left(\phi(q)u \otimes qu\right) = \dvm\left(\phi(q)\sqrt{q}\mathbf{S}\right) + \phi(q) \mathbf{f} - \sqrt{q} \mathbf{S}\nabla(\phi(q)) +\left[\partial_t(\phi(q)) + u \cdot \nabla(\phi(q))\right] qu
\end{equation}
where $\mathbf{f}$ is defined in \eqref{def:f} and $\mathbf{S} = 2\nu \mathbf{S}_{\rm NS} + 2\kappa^2\mathbf{S}_{\rm K}$.

These calculations can be justified by rewriting the system in the original variables $(\rho,u)$ where $\rho = q \rho_m$ and then by applying directly the strategy of \cite{Lacroix-Violet_Vasseur_JMPA18} as $\rho$ satisfies the same regularities with the addition of $\int_{\mathbb{R}^d} \rho(\cdot,x) |x|^2 \,{\rm d}x$ belongs to $L^\infty_{\rm loc}(0,\infty;\mathbb{R})$ with bound depending only on the initial energy and BD entropy (thanks to the strong Poincaré inequality, see Proposition \ref{prop:SP_inequality}).

For, $l \geq 1$, let us consider $\phi_l$ a cut-off (Equation (3.1) in \cite{Lacroix-Violet_Vasseur_JMPA18}) defined on $[0,\infty)$ by, for all $y \in [0,\infty)$,
\begin{equation*}
\phi_l(y) = (2ly-1){\bf 1}_{\left[\frac1{2l},\frac1l\right]}(y) + {\bf 1}_{\left[\frac1l,l\right]}(y) + \left(2-\frac{y}{l}\right){\bf 1}_{\left[l,2l\right]}(y).
\end{equation*}
The function $\phi_l(q)q^{-\frac14}$ belongs to $L^\infty_{\rm loc}(0,\infty;L^\infty)$ and for $r_1 >0$, we have $q^{\frac14}u \in L_{\rm loc}^4(0,\infty;L^4_{\mu_m})$, thus $v = \phi_l(q)u = \left(\phi_l(q)q^{-\frac14}\right)q^{\frac14}u \in L_{\rm loc}^4(0,\infty;L^4_{\mu_m})$. Then, for any $\Theta \in W^{2,\infty}(\mathbb{R}^d)$, using the ideas developped in \cite{Lacroix-Violet_Vasseur_JMPA18}, we obtain, taking the scalar product of \eqref{eq:equation_phi(q)_u} with $(\nabla\Theta(v))$, 
\begin{eqnarray*}
\partial_t(q\Theta(v))+\dvm(\Theta(v) qu) & = & \dvm\left(\phi_l(q)\sqrt{q}\mathbf{S}\nabla(\Theta(v))\right) +\phi_l(q) \mathbf{f} \cdot \nabla(\Theta(v)) - \sqrt{q} \mathbf{S}\nabla(\phi_l(q)) \cdot \nabla(\Theta(v)) \\
 & & -\sqrt{q} \phi_l'(q) \tr(\mathbf{T}_{\rm NS})qu \cdot \nabla(\Theta(v))  + \phi_l'(q) \left(\sqrt{q}\frac{x}{\sigma^2}\right) \cdot \left(\sqrt{q}u\right) qu \cdot \nabla(\Theta(v)) \\
	&  &   - \tr\left(\phi_l(q)\sqrt{q}\mathbf{S}\nabla(\nabla\Theta(v))\right).
\end{eqnarray*}
Because the sequence $(\phi_l)_{l\geq 1}$ is an approximation of the constant function $\bf 1$ in $W^{2,\infty}(\mathbb{R}^d)$, we can pass to the limit $l \to \infty$ in the previous identity to obtain \eqref{eq:CNSK_system_m_RenormalizedWeakSolution-velocity}.  In particular, we obtain (as $\phi_l'(q) \to 0$, $\nabla (\phi_l(q)) \to 0$ and $\sqrt{q}\nabla v \to {\bf T}_{\rm NS}$),
\begin{equation*}
\mathfrak{m}^1_\Theta =  - \tr \left(\mathbf{S} (D^2\Theta)(u){\bf T}_{\rm NS}\right) \in L^1_{\rm loc}(0,\infty;L^1_{\mu_m})
\end{equation*}
with the bound (see \eqref{def:energy_reg_intro} and \eqref{def:BD_reg-entropie_intro})
\begin{align*}
\|\mathfrak{m}^1_\Theta\|_{L^1_{\rm loc}(0,\infty;L^1_{\mu_m})}& \leq \|D^2\Theta\|_{L^\infty} \|\mathbf{S}\|_{L^2_{\mathrm{loc}}(0,\infty:[L^2_{\mu_m}]^{d\times d})} \|{\bf T}_{\rm NS}\|_{L^2_{\mathrm{loc}}(0,\infty:[L^2_{\mu_m}]^{d\times d})}\\
&\leq \|D^2\Theta\|_{L^\infty}\left(\mathcal{E}(\sqrt{q^0},\sqrt{q^0}u^0)+ \mathcal{E}_{\rm BD}(\sqrt{q^0},\sqrt{q^0}u^0)\right).
\end{align*}
Still following the lines of \cite{Lacroix-Violet_Vasseur_JMPA18} (or equivalently \cite{Carles_Carrapatoso_Hillairet_AIF2022}), we can obtain the compatibility condition on the tensor $\mathbf{T}_{\rm NS}$ by computing (for regularized $\sqrt{q}$ and $u$) the following expression, for all $i,j,k =1, 2, 3$,
\begin{equation*}
\sqrt{q} \frac{\partial \Theta}{\partial x_i}(u) \sqrt{q} \frac{\partial u_j}{\partial x_k} = \partial_k\left((\sqrt{q})^2 \frac{\partial \Theta}{\partial x_i}(u) u_j\right) - 2 \sqrt{q} u_j \frac{\partial \Theta}{\partial x_i}(u) \partial_k (\sqrt{q})  - (\sqrt{q} u_j) \frac{\partial^2 \Theta}{\partial x_i\partial x_p}(u) \sqrt{q} \frac{\partial u_p}{\partial x_k}.
\end{equation*}
Thus we obtain the expression of $\mathfrak{m}^2_\Theta$ as follows, for all $i,j,k=1,\ldots,d$,
\begin{equation*}
(\mathfrak{m}^2_\Theta)_{i,j,k} =- (\sqrt{q} u_j) [D^2(\Theta)(u)\mathbf{T}_{\rm NS}]_{ik}
\end{equation*}
which belongs to $L^1_{\rm loc}(0,\infty;L^1_{\mu_m})$ and satisfies, for all $i,j,k=1,\ldots,d$,
\begin{equation*}
\|(\mathfrak{m}^2_\Theta)_{i,j,k}\|_{L^1_{\rm loc}(0,\infty;L^1_{\mu_m})} \leq \|D^2\Theta\|_{L^\infty}\left(\mathcal{E}(\sqrt{q^0},\sqrt{q^0}u^0)+ \mathcal{E}_{\rm BD}(\sqrt{q^0},\sqrt{q^0}u^0)\right).
\end{equation*}

This proves the first point of Theorem \ref{th:solutionwithoutdragterms}.

\subsection{Construction of sequences of smooth initial data and drag coefficients.}{\label{subsec:smoothinitialdata}}

Let us construct, for a given $(\sqrt{q^0},\sqrt{q^0}u^0)$ as in Definition \ref{definition:RenormalizedWeakSolution}, a sequence $(\sqrt{q^0_n},\sqrt{q^0_n}u^0_n)_{n}$ satisfying \eqref{eq:_regularized_initial_data}, for all $n$.

\paragraph{The sequence $(\sqrt{q^0_n})_n$.}

Let $\chi, \zeta \in \mathcal{D}(\mathbb{R}^d)$ such that 
\begin{equation*}
{\bf 1}_{\left\{|x| \leq \frac{1}{2}\right\}} \leq \chi \leq {\bf 1}_{\left\{|x| < 1\right\}}, \quad {\rm supp}(\zeta) \subset B(0_{\mathbb{R}^d},1) \quad \mbox{and} \quad \int_{\mathbb{R}^d} \zeta = 1
\end{equation*}
and define, for $\sqrt{q^0}$ as in Definition \ref{definition:RenormalizedWeakSolution} and for all $n\geq 1$,
\begin{equation*}
{\sqrt{q^0_n}} = \frac{\left(({\sqrt{q^0}}) \chi_n+\frac1n\right) \ast \zeta_n}{\left\|\left(({\sqrt{q^0}}) \chi_n+\frac1n\right) \ast \zeta_n\right\|_{L^2_{\mu_m}}}
\end{equation*}
with, for all $x \in \mathbb{R}^d$ and all $n\geq 1$, 
\begin{equation*}
\chi_n(x) = \chi\left(\frac{x}{n}\right) \quad \mbox{and} \quad \zeta_n(x) = n^d \zeta\left(nx\right).
\end{equation*}

Let us give a result follwing Proposition 4.2. in \cite{Carles_Carrapatoso_Hillairet_AIF2022}.
\begin{Proposition} \label{prop:initial_data_sequence}
	The sequence $(\sqrt{q_n^0})_{n\geq 1}$ satisfies
	\begin{equation*}
	\underset{n\to \infty}{\limsup} \int_{\mathbb{R}^d} |\nabla {\sqrt{q^0_n}}|^2 \;{\rm d}\mu_m \leq \int_{\mathbb{R}^d} |\nabla {\sqrt{q^0}}|^2 \;{\rm d}\mu_m. 
	\end{equation*}
\end{Proposition}

Note that this implies also, thanks to the Log-Sobolev inequality (see Appendix \ref{AppendixB}), that 
\begin{equation*}
2\int_{\mathbb{R}^d} q^0_n \ln(\sqrt{q^0_n}) \;{\rm d}\mu_m = \int_{\mathbb{R}^d} q_n^0 \ln(q_n^0)\;{\rm d}\mu_m
\end{equation*}
is bounded\footnote{Recall that $\displaystyle \int_{\mathbb{R}^d} (q_n^0\ln(q_n^0) - q_n^0 +1) \;{\rm d}\mu_m \geq 0$ and thus $\int_{\mathbb{R}^d} q_n^0\ln(q_n^0) \;{\rm d}\mu_m$ is bounded from below independently of $n$.} independently of $n$.

\paragraph{The sequence $(\sqrt{q_n^0}u_n^0)_n$.} For $n\geq1$, we define
\begin{equation*}
u_n^0 = \frac{\sqrt{q^0}u^0}{\sqrt{q_n^0}}\chi_n.
\end{equation*}
Note that $u_n^0$ satisfies \eqref{eq:initial_(q0,u0)} for all $n\geq1$.

\paragraph{The sequences of drag terms.} Still following \cite[Section 4]{Carles_Carrapatoso_Hillairet_AIF2022}, we consider such a sequence of initial data and, in the same time, sequences  $(r_{0,n})_{n\geq 1}$,  $(r_{1,n})_{n\geq 1}$  and $(r_{4,n})_{n\geq 1}$ defined, for all $n\geq1$, by
\begin{equation}\label{def:r1n}
\displaystyle r_{0,n} = \frac{1}{\displaystyle  n+\left(\int_{\mathbb{R}^d} (q_n^0-\ln(q_n^0)) \;{\rm d}\mu_m\right)^2}, \quad r_{1,n} = \frac1n \quad \mbox{and} \qquad r_{4,n} = \frac{1}{n+I_4(q_n^0)^2}
\end{equation}
allowing then both $r_{4,n} \underset{n\to \infty}{\longrightarrow} 0$ and $r_{4,n}I_4(q_n^0) \underset{n\to \infty}{\longrightarrow} 0$ (while the sequence $(I_4(q_n^0))_{n\geq 1}$ may not be bounded) and in the same way both $r_{0,n}  \underset{n\to \infty}{\longrightarrow} 0$ and $\displaystyle r_{0,n}\int_{\mathbb{R}^d} (q_n^0-\ln(q_n^0)) \;{\rm d}\mu_m \underset{n\to \infty}{\longrightarrow} 0$.

\begin{Remark}
In the simpler cases where $r_0>0$, $r_1>0$ and/or $r_4>0$, we can consider any sequences $(r_{0,n})_n$, $(r_{1,n})_n$ and $(r_{4,n})_n$ converging to $r_0$, $r_1$ and $r_4$, respectively. In this case, in particular, $(r_{0,n})_n$ and $(r_{4,n})_n$ do not need to depend on the sequence of initial data $(\sqrt{q^0_{n}})_n$.
\end{Remark}

\subsection{Proof of the third point of Theorem \ref{th:solutionwithoutdragterms}: limits with respect to the drag terms}\label{subsec:stability}

	Let $(\sqrt{q^0},\sqrt{q^0}u^0) \in H^1_{\mu_m} \times \left[L^2_{\mu_m}\right]^d$ satisfying \eqref{eq:initial_data_compatibility}.
Let us then define a sequence of initial data $(\sqrt{q_n^0},\sqrt{q_n^0}u_n^0)$ and drag terms $(r_{0,n},r_{1,n},r_{4,n})$ as in Section~\ref{subsec:smoothinitialdata}.
	
For any $n \geq 1$, thanks to Theorem \ref{prop:GWS_CNSK_system_m_regularized}, as the initial data $(\sqrt{q_n^0},\sqrt{q_n^0}u_n^0)$ satisfies \eqref{eq:_regularized_initial_data} and \eqref{eq:initial_(q0,u0)}, there exists a global weak solution $(\sqrt{q_n},\sqrt{q_n}u_n)$ to \eqref{eq:CNSK_system_m_regularized} with initial data $(\sqrt{q_n^0},\sqrt{q_n^0}u_n^0)$ and positive drag terms $(r_{0,n},r_{1,n},r_{4,n})$ in the sense of Definition \ref{def:WeakSolution_drag_term_intro} (see Remark \ref{rem:definitions2.1and1.2}).

Thanks to the first point of Theorem \ref{th:solutionwithoutdragterms} (see Section \ref{subsec:weaktorenormalized}), for any $n\geq1$, $(\sqrt{q_n},\sqrt{q_n}u_n)$ is a renormalized weak solution to \eqref{eq:CNSK_system_m_drag-term_WeakSolution} in the sense of Definition \ref{definition:RenormalizedWeakSolution}.

This sequence of solutions $((\sqrt{q_n},\sqrt{q_n}u_n))_{n\geq 1}$ satisfies the system and energy and entropy estimates of Subsection \ref{subsect:GWS_drag_term} and, thanks to Proposition \ref{prop:initial_data_sequence}, the right-hand side of the energy and entropy estimates are themselves bounded independently of $n$.

More precisely, the sequences $(\sqrt{q_n},\sqrt{q_n}u_n,\mathbf{T}_{{\rm NS},n}, 
\mathbf{S}_{{\rm K},n})_{n\geq 1}$ satisfy
\begin{equation*}
\sqrt{q_n} \in L^\infty_{\rm loc}(0,\infty,H^1_{\mu_m}), \qquad \sqrt{q_n}u_n \in  L^\infty_{\rm loc}(0,\infty,L^2_{\mu_m}), \qquad \mathbf{T}_{{\rm NS},n}, \mathbf{S}_{{\rm K},n} \in L^2_{\rm loc}(0,\infty;[L^2_{\mu_m}]^{d\times d}),
\end{equation*}	
with bounds independent of $n$. Therefore, by compactness, there exists $\sqrt{q}$, $\sqrt{q}u$, $\mathbf{T}_{\rm NS}$ and $\mathbf{S}_{\rm K}$ such that 
\begin{equation*}
\sqrt{q} \in L^\infty_{\rm loc}(0,\infty,H^1_{\mu_m}), \qquad \sqrt{q}u \in  L^\infty_{\rm loc}(0,\infty,L^2_{\mu_m}), \qquad \mathbf{T}_{{\rm NS}}, \mathbf{S}_{{\rm K}} \in L^2_{\rm loc}(0,\infty;[L^2_{\mu_m}]^{d\times d}),
\end{equation*}
and subsequences (without relabelling subsequences) such that
\begin{equation*}
\begin{array}{rcll} 
\sqrt{q_n} & \underset{n\to \infty}{\longrightarrow} & \sqrt{q} &\mbox{in } L^\infty_{\rm loc}(0,\infty,H^1_{\mu_m}) - w*, \\
\sqrt{q_n}u_n & \underset{n\to \infty}{\longrightarrow} & \sqrt{q}u &\mbox{in } L^\infty_{\rm loc}(0,\infty,[L^2_{\mu_m}]^d) - w*, \\
\mathbf{T}_{{\rm NS},n} & \underset{n\to \infty}{\longrightarrow} & \mathbf{T}_{\rm NS} &\mbox{in }L^2_{\rm loc}(0,\infty;[L^2_{\mu_m}]^{d\times d}) -w, \\
\mathbf{S}_{{\rm K},n} & \underset{n\to \infty}{\longrightarrow} & \mathbf{S}_{\rm K} &\mbox{in }L^2_{\rm loc}(0,\infty;[L^2_{\mu_m}]^{d\times d}) -w.
\end{array} 
\end{equation*}

By the same type of argument as in the proof of Proposition \ref{prop:FG-convergence} (namely a priori bounds obtained from the energy and BD entropy estimates and the Aubin-Lions-Simon Lemma \ref{lemma:Aubin-Lions-Simon}), we can pass to the limit in the energy and BD entropy estimates and also in Equation \eqref{eq:CNSK_system_m_drag-term_WeakSolution-density}.

Note that, going back to the original density $\rho = q \rho_m$, and denoting, for each $n \geq 1$, $\rho_n = q_n \rho_m$, we infer from the bound on $\mathbf{S}_{{\rm K},n}$, the compatibility condition in the definition of the renormalized weak solutions and Lemmata \ref{lemma:Appendix_rho_q1} and \ref{lemma:Appendix_rho_q2} that $(\rho_n)_{n\geq 1}$ is bounded in $L^2$ such that $(\sqrt{\rho_n}D^2(\ln(\rho_n)))_{n\geq 1}$ is also bounded in $[L^2]^{d\times d}$. 

Using also Lemma \ref{lem:Hessian_estimate}, we conclude that $(D^2(\sqrt{\rho_n}))_{n\geq 1}$ and $\left(\nabla\left(\rho_n^{\frac14}\right)\right)_{n\geq 1}$ are bounded respectively in $L^2_{\rm loc}(0,\infty;[L^2]^{d\times d})$ and $L^4_{\rm loc}(0,\infty;[L^4]^d)$ and therefore weakly converge to $D^2(\sqrt{\rho})$ and $\nabla \left(\rho^{\frac14}\right)$ respectively in these spaces.

From the strong convergence of the sequence $\sqrt{\rho_n}$ to $\sqrt{\rho}$ in $\mathcal{C}^0([0,\infty);L^2)$, we can pass to the limit in the compatibility condition to obtain that 
\begin{equation*}
	\sqrt{q} \mathbf{S}_{\rm K}  =  \frac{1}{\rho_m}\left[\sqrt{\rho}D^2(\sqrt{\rho}) - \nabla(\sqrt{\rho}) \otimes \nabla (\sqrt{\rho}) + \frac\rho{2\sigma^2} {\rm I}_d\right] 
\end{equation*}
in the space $L^1_{\mu_m}$. 

In order to prove that the compatibility condition for $\mathbf{T}_{\rm NS}$ is satisfied, we can follow once again \cite{Lacroix-Violet_Vasseur_JMPA18} (see Lemma 5.1 and the proof of part (4) of Theorem 2.1 page 207). Let us state the last property we will use.
\begin{Lemma}\label{lemma:LV-V}
	For any $H \in L^{\infty}(\mathbb{R}^d)$, $(\sqrt{q_n}H(u_n))_{n\geq 1}$ converges strongly to $\sqrt{q}H(u)$ in $L^2_{\rm loc}(0,\infty;L^2_{\mu_m})$ where $u$ is defined a.e. in $\mathbb{R}^d$ by $u = \frac{\sqrt{q}u}{\sqrt{q}} {\bf 1}_{\{\sqrt{q} >0\}}$, $(\sqrt{q}u)$ is the limit of the sequence $(\sqrt{q_n}u_n)_{n\geq 1}$ and $\sqrt{q}$ is the limit of the sequence $(\sqrt{q_n})_{n\geq 1}$ obtained above.
\end{Lemma}
\begin{proof}
	The proof of this property can be adapted from \cite[Lemma 5.1 (4)]{Lacroix-Violet_Vasseur_JMPA18}.
\end{proof}

Now, for any $\Theta \in W^{2,\infty}(\mathbb{R}^d)$, for any $i,j,k =1, \ldots, d$,   using Lemma \ref{lemma:LV-V} for $H = \frac{\partial \Theta}{\partial x_i}$, the strong convergence of $(\sqrt{q_n})_{n\geq1}$ to $\sqrt{q}$ in $\mathcal{C}^0([0,\infty);L^2_{\mu_m})$, the weak convergence of the sequence $(\mathbf{T}_{{\rm NS},n})_{n\geq 1}$ to $\mathbf{T}_{\rm NS}$ in $L^2_{\rm loc}(0,\infty;[L^2_{\mu_m}]^{d\times d})$, and the compactness of the bounded sequence of measures $(\{(\mathfrak{m}_\Theta^2)_{n}\}_{i,j,k})_{n\geq 1}$, we can pass to the limit in the compatibility condition 
\begin{equation}\label{eq:RenormalizedWeakSolution_CompatibiltyTNS_convergence}
\sqrt{q_n} \frac{\partial \Theta}{\partial x_i}(u_n) [\{\mathbf{T}_{\rm NS} \}_n]_{j,k} = \partial_k\left(q_n \frac{\partial \Theta}{\partial x_i}(u_n) \{u_n\}_j\right) - 2 \sqrt{q_n} \{u_n\}_j \frac{\partial \Theta}{\partial x_i}(u_n) \partial_k \sqrt{q_n} + \{(\mathfrak{m}_\Theta^2)_n\}_{i,j,k}.
\end{equation}
Indeed, first $(\{(\mathfrak{m}_\Theta^2)_n\}_{i,j,k})_n$ is bounded independently of $n$ thanks to the definition of renormalized weak solution and the definition of the sequences of initial data (which allow uniform in $n$ bounds of the initial energy and BD entropy), then up to a subsequence, the sequence $(\{(\mathfrak{m}_\Theta^2)_n\}_{i,j,k})_n$ converges in $\mathcal{M}((0,\infty)\times \mathbb{R}^d;\mathbb{R})$ to an element (denoted $\{\mathfrak{m}_\Theta^2\}_{i,j,k}$ and satisfying the same bound as the sequence) and second
\begin{equation*}
\begin{array}{rcll} 
\sqrt{q_n} \frac{\partial \Theta}{\partial x_i}(u_n) [\{\mathbf{T}_{\rm NS} \}_n]_{j,k} & \underset{n\to \infty}{\longrightarrow} & \sqrt{q} \frac{\partial \Theta}{\partial x_i}(u) [\mathbf{T}_{\rm NS}]_{j,k} &\mbox{in } L^1_{\rm loc}(0,\infty,L^1_{\mu_m}) - w, \\
2 \sqrt{q_n} \{u_n\}_j \frac{\partial \Theta}{\partial x_i}(u_n) \partial_k \sqrt{q_n} & \underset{n\to \infty}{\longrightarrow} & 2 \sqrt{q} u_j \frac{\partial \Theta}{\partial x_i}(u) \partial_k \sqrt{q} &\mbox{in } L^1_{\rm loc}(0,\infty,L^1_{\mu_m}) - w, \\
\partial_k\left(q_n \frac{\partial \Theta}{\partial x_i}(u_n) \{u_n\}_j\right) & \underset{n\to \infty}{\longrightarrow}& \partial_k\left(q \frac{\partial \Theta}{\partial x_i}(u)u_j\right) & \mbox{in } \mathcal{M}((0,\infty)\times \mathbb{R}^d;\mathbb{R}).
\end{array} 
\end{equation*}
The last convergence is a consequence of the identity \eqref{eq:RenormalizedWeakSolution_CompatibiltyTNS_convergence} satisfied for every $n$ and the convergences of the others terms in this space. The convergence occurs a priori in a subspace of distribution of order~$1$ (in space) as the sequence $(q_n \frac{\partial \Theta}{\partial x_i}(u_n) \{u_n\}_j)_{n\geq 1}$ converges weakly in $L^1_{\rm loc}(0,\infty;L^1_{\mu_m})$ by the arguments above. Thus, for all $\Phi \in \mathcal{D}([0,\infty)\times \mathbb{R}^d)$,
\begin{eqnarray*}
\left\langle \partial_k\left(q_n \frac{\partial \Theta}{\partial x_i}(u_n) \{u_n\}_j\right), \Phi \right\rangle & = & - \left\langle q_n \frac{\partial \Theta}{\partial x_i}(u_n) \{u_n\}_j, \partial_k\Phi \right\rangle \\
& \underset{n\to \infty}{\longrightarrow} & - \left\langle q \frac{\partial \Theta}{\partial x_i}(u) u_j, \partial_k\Phi \right\rangle = \left\langle \partial_k\left(q \frac{\partial \Theta}{\partial x_i}(u)u_j\right), \Phi \right\rangle.
\end{eqnarray*}

The same arguments allow us to pass to the limit in Equation \eqref{eq:CNSK_system_m_RenormalizedWeakSolution-velocity} (written for $\sqrt{q_n}$, $\sqrt{q_n}u_n$) to obtain the corresponding equation satisfied by $(\sqrt{q},\sqrt{q}u)$.

This proves that indeed, up to subsequences, $((\sqrt{q_n},\sqrt{q_n}u_n))_n$ converges to $(\sqrt{q},\sqrt{q}u)$, renormalized weak solution to \eqref{eq:CNSK_system_m_drag-term_WeakSolution} with $r_0 = 0$, $r_1 = 0$, $r_4 =0$.

This concludes the proof of Theorem \ref{th:solutionwithoutdragterms}.

\section{The case without potential force: global weak solutions for a rescaled QNS system similar to \cite{Carles_Carrapatoso_Hillairet_AIF2022}} \label{sect:CCH}

We would like in this section to show that our strategy of proof can be adapted to the Quantum Navier-Stokes system without potential force in $\mathbb{R}^d$ (for $d\geq 1$ an integer), that is the system studied in \cite{Carles_Carrapatoso_Hillairet_AIF2022}, which corresponds to \eqref{eq:CNSK_system} with $\lambda = 0$ in this paper, namely
\begin{subequations}\label{eq:CNSK_system_0}
	\begin{eqnarray}
	\partial_t\rho + \dv(\rho u) & = & 0, \label{eq:CNSK_system_0_density} \\
	\partial_t(\rho u) + \dv(\rho u \otimes u) & = & \dv(2\nu \rho D(u))  + 2\kappa^2 \rho\nabla\left(\frac{\Delta(\sqrt{\rho})}{\sqrt{\rho}}\right) - a \nabla \rho. \label{eq:CNSK_system_0_velocity}
	\end{eqnarray}
\end{subequations}

As in \cite{Carles_Carrapatoso_Hillairet_AIF2022}, we can consider an equivalent problem obtained after a time-dependent space rescaling $y = \frac{x}{\tau(t)}$ where $\tau$ solves a nonlinear ODE.
Let us introduce $R$ and $U$ defined by (recall that $m_F=1$)
\begin{equation}\label{def:autosimilaire}
\rho(t,x) = \frac{1}{(\tau(t))^d} R \left(t, \frac{x}{\tau(t)}\right), \qquad u(t,x) = \frac{1}{\tau(t)} U\left(t,\frac{x}{\tau(t)}\right) + \frac{\dot{\tau}(t)}{\tau(t)}x
\end{equation}
where $\tau$ is a function of time that will be determined later on. The functions $R$ and $U$ will be the new unknowns.

Putting the ansatz \eqref{def:autosimilaire} into  \eqref{eq:CNSK_system_0}, and denoting $y = \frac{x}{\tau(t)}$ the new space variable, one can show that the new unknowns $(R,U)$ satisfy
	\begin{eqnarray*}
	\partial_t R + \frac{1}{\tau^2}{\dv}_y(R U ) & = & 0,\\ 
	\partial_t(R U ) + \frac{1}{\tau^2}{\dv}_y( RU \otimes U) + R \ddot{\tau}\tau y & = & \frac{1}{ \tau^2}{\dv}_y(2\nu R D_y(U)) + \frac{2\kappa^2}{ \tau^2} R \nabla_y\left(\frac{\Delta_y(\sqrt{R})}{\sqrt{R}}\right)  \\
	& &- a \nabla_y R + 2\nu \frac{\dot \tau}{\tau} \nabla_y R. 
	\end{eqnarray*}

We here look for $\tau$ solution to the ODE (which exists, is unique and belongs to $\mathcal{C}^2([0,\infty);(0,\infty)$),
\begin{equation}\label{eq:ODE_tau}
\ddot{\tau} = \frac{a}{\tau} + \frac{\kappa^2}{\tau^3} - 2\nu \frac{\dot \tau}{\tau^2}, \qquad \tau(0) = 1 \quad \mbox{and} \quad \dot{\tau}(0) = 0.
\end{equation}
 With such an expression for $\ddot{\tau}$, we obtain a system that we will rewrite in the variables $Q = \frac{R}{R_m}$ and $U$, denoting in this section the operators $\dvm$ and $\Delta_m$ with respect to the density $R_m$ given in Figure \ref{figure} (and which corresponds to $\rho_m$ when $\sigma=1$), that is
 \begin{equation*}
 \dvm = \frac{1}{R_m}\dv(R_m \cdot ) \qquad \mbox{and} \qquad \Delta_m = \dvm(\nabla \cdot) = \frac{1}{R_m}\dv(R_m \nabla(\cdot)).
 \end{equation*}
 
 The new system reads
 \begin{subequations}\label{eq:CNSK_system_tau}
 	\begin{eqnarray}
 	\partial_t Q + \frac{1}{\tau^2}\dvm(Q U ) & = & 0, \label{eq:CNSK_system_tau_density} \\
 	\partial_t(Q U ) + \frac{1}{\tau^2}\dvm( QU \otimes U) & = & \frac{1}{\tau^2}\dvm(2\nu Q D_y(U)) + \frac{2\kappa^2}{ \tau^2} Q \nabla_y\left(\frac{\Delta_m(\sqrt{Q})}{\sqrt{Q}}\right)  \label{eq:CNSK_system_tau_velocity} \\
 	& & - \left(a - \frac{2\nu \dot \tau}{\tau}\right) Q \nabla_y \ln(Q). \notag
 	\end{eqnarray}
 \end{subequations}

\begin{Remark}
	In \cite{Carles_Carrapatoso_Hillairet_AIF2022}, the parameter $\tau$ satisfies, instead of \eqref{eq:ODE_tau}, the following ODE (written in our setting)
	\begin{equation}\label{eq:ODE_zeta}
			\ddot \tau = \frac{a}{\tau},  \qquad \tau(0) = 1 \quad \mbox{and} \quad  \dot{\tau}(0) = 0.
	\end{equation}
Note that the large time behavior of the solution to \eqref{eq:ODE_zeta} and \eqref{eq:ODE_tau} are the same (and does not depend on the initial data, see \cite[Section 3, in particular Lemma 3.2]{Carles_Carrapatoso_Hillairet_AHL2018} for details). Here
\begin{equation*}
	\tau(t) \underset{t\to \infty}{\sim} t\sqrt{2a\ln(t)} \qquad \dot{\tau}(t) \underset{t\to \infty}{\sim} \sqrt{2a\ln(t)}.
\end{equation*}
\end{Remark}

Let us summarize in Figure \ref{figure} the differences between our system and the one studied in \cite{Carles_Carrapatoso_Hillairet_AIF2022}.
\begin{center} 
	\begin{figure}[!h]
		\begin{tabular}{|l||l|l|}
			\hline 
			Quantum Navier-Stokes equations & Equation \eqref{eq:CNSK_system_0} & Equation (1.1) in \cite{Carles_Carrapatoso_Hillairet_AIF2022} 	\\
			\hline 
			\hline
			Pressure coefficient & $a$ & $1$ \\
			\hline
			
			Quantum coefficient & $2\kappa^2$ & $\frac{\epsilon^2}{2}$ \\ 
			
			\hline
			
			Viscosity coefficient & $2\nu$ & $\nu$ \\
			
			\hline 
			Autosimilar parameter $\tau$ ODE & Equation \eqref{eq:ODE_tau} & $\ddot \tau = \frac{2}{\tau}, \; \tau(0)=1, \; \dot \tau(0) = 0$ \\
			\hline 
			
			Stationnary density & $R_m : y \mapsto (2\pi)^{-\frac{d}{2}}\exp\left(-\frac12 |y|^2\right)$  & $\Gamma(y) = \exp(-|y|^2)$  \\
			\hline 
			System studied in the variables $(R,U)$ & Equation \eqref{eq:CNSK_system_tau} & Equation (1.5) in \cite{Carles_Carrapatoso_Hillairet_AIF2022} \\
			\hline 
		\end{tabular}
		\caption{Differences in notation between this paper and \cite{Carles_Carrapatoso_Hillairet_AIF2022}}  \label{figure}
	\end{figure}
\end{center}

Solutions to system \eqref{eq:CNSK_system_tau} satisfy a BD entropy identity. One can recover this identity by considering the system satisfied by
$W = U+2\nu\nabla \ln(Q)$ the effective velocity, namely (denoting $\mathbf{A}(U) = \frac12\left(\nabla_y U - \nabla_yU^\top\right)$),
\begin{subequations}\label{eq:CNSK_system_tau_BD}
	\begin{eqnarray}
	\partial_t Q + \frac{1}{\tau^2}\dvm(Q U ) & = & 0, \label{eq:CNSK_system_tau_BD_density} \\
	\partial_t(Q W) + \frac{1}{\tau^2}\dvm(W \otimes QU) & = & \frac{1}{\tau^2}\dvm(2\nu Q \mathbf{A}(U)) + \frac{2\kappa^2}{ \tau^2} Q \nabla_y\left(\frac{\Delta_m(\sqrt{Q})}{\sqrt{Q}}\right)  \label{eq:CNSK_system_tau_BD_velocity} \\
	& & - \left(a - \frac{2\nu \dot \tau}{\tau}\right) Q \nabla_y \ln(Q). \notag
	\end{eqnarray}
\end{subequations}

\paragraph{Energy identity:}
		\begin{equation*}
		\frac{{\rm d}}{{\rm d}t}\mathcal{E}^\tau(Q,U) + \mathcal{D}^\tau(Q,U) = \frac{\dot \tau }{\tau^3} \int_{\mathbb{R}^d} QU \cdot \left[2\nu \nabla\ln(Q)\right] \;{\rm d}\mu_m,
		\end{equation*}
		where the energy and dissipation are given respectively by
		\begin{eqnarray*}
			\mathcal{E}^\tau(Q,U) & = & \frac{1}{2\tau^2}\left(\int_{\mathbb{R}^d} Q|U|^2 \;{\rm d}\mu_m + 4\kappa^2\int_{\mathbb{R}^d} \left|\nabla\left(\sqrt{Q}\right)\right|^2 \;{\rm d}\mu_m\right) + a \int_{\mathbb{R}^d} Q\ln(Q) \;{\rm d}\mu_m, \\
			\mathcal{D}^\tau(Q,U) & = & \frac{\dot{\tau}}{\tau^3}\left(\int_{\mathbb{R}^d} Q|U|^2 \;{\rm d}\mu_m + 4\kappa^2\int_{\mathbb{R}^d} \left|\nabla\left(\sqrt{Q}\right)\right|^2 \;{\rm d}\mu_m\right) + \frac{2\nu}{\tau^4} \int_{\mathbb{R}^d} Q|D(U)|^2    \;{\rm d}\mu_m	.
		\end{eqnarray*}
		
\paragraph{BD entropy identity:}	
		\begin{equation*}
		\frac{{\rm d}}{{\rm d}t}\mathcal{E}_{{\rm BD}}^\tau(Q,U) + \mathcal{D}_{{\rm BD}}^\tau(Q,U) = -\frac{\dot \tau }{\tau^3} \int_{\mathbb{R}^d} QU \cdot \left[2\nu \nabla\ln(Q)\right] \;{\rm d}\mu_m,
		\end{equation*}
		where the BD entropy and corresponding dissipation are given respectively by 
		\begin{eqnarray*}
			\mathcal{E}_{{\rm BD}}^\tau(Q,U) & = & \frac{1}{2\tau^2}\left(\int_{\mathbb{R}^d} Q|W|^2 \;{\rm d}\mu_m + 4\kappa^2\int_{\mathbb{R}^d} \left|\nabla\left(\sqrt{Q}\right)\right|^2 \;{\rm d}\mu_m\right) + a \int_{\mathbb{R}^d} Q\ln(Q) \;{\rm d}\mu_m ,\\
			\mathcal{D}_{{\rm BD}}^\tau(Q,U) & = & \frac{\dot{\tau}}{\tau^3}\left(\int_{\mathbb{R}^d} Q|U|^2 \;{\rm d}\mu_m + 4\kappa^2\int_{\mathbb{R}^d} \left|\nabla\left(\sqrt{Q}\right)\right|^2 \;{\rm d}\mu_m\right) + \frac{2\nu}{\tau^4} \int_{\mathbb{R}^d} Q|\mathbf{A}(U)|^2    \;{\rm d}\mu_m \\
			& & 	+\frac{2\nu \kappa^2}{\tau^4}\int_{\mathbb{R}^d} Q|D^2(\ln(Q))|^2  \;{\rm d}\mu_m + 2\nu \left(\frac{a}{\tau^2}+\frac{\kappa^2}{\tau^4}\right)\int_{\mathbb{R}^d} Q|\nabla \ln(Q)|^2  \;{\rm d}\mu_m.
		\end{eqnarray*}

This leads in particular to the following identity by summation:
\begin{equation*}
\frac{{\rm d}}{{\rm d}t}\left(\mathcal{E}^\tau(Q,U) +\mathcal{E}_{{\rm BD}}^\tau(Q,U)\right) +  \mathcal{D}^\tau(Q,U)+ \mathcal{D}_{{\rm BD}}^\tau(Q,U) = 0.
\end{equation*}

It is easy to see that \eqref{eq:CNSK_system_tau} has the same structure as  \eqref{eq:CNSK_system_m} (with coefficients depending on $\tau$ (and $\dot{\tau}$)) and corresponds to the system (1.5) in \cite{Carles_Carrapatoso_Hillairet_AIF2022}. Therefore, we can follow the strategy of proof developed in the present paper, using the same energy $\mathcal{E}^\tau(Q,U)$ and BD entropy $\mathcal{E}_{{\rm BD}}^\tau(Q,U)$ as in \cite{Carles_Carrapatoso_Hillairet_AHL2018}. This leads to the existence of global weak solutions to \eqref{eq:CNSK_system_tau} in the sense of the following definition. 

\begin{Definition}\label{def:WeakSolution_tau}
	Let $(\sqrt{Q^0},\Lambda^0) \in L^2_{\mu_m} \times [L^2_{\mu_m}]^d$. A pair $(Q,U)$ is a weak solution to \eqref{eq:CNSK_system_tau} associated to the initial data $(\sqrt{Q^0},\Lambda^0)$ if there exists a quadruplet $(\sqrt{Q},\Lambda,\mathbf{S}_{\rm K},\mathbf{T}_{\rm NS})$ such that 
	\begin{enumerate}

	\item The following regularities are satisfied
		\begin{equation*}
		\begin{array}{rclcrcl}
		\sqrt{Q}   &\in&  L_{\rm loc}^\infty(0,\infty,L^2_{\mu_m}), &  & \nabla\sqrt{Q}, \Lambda &\in&  L_{\rm loc}^\infty(0,\infty,[L^2_{\mu_m}]^d), \\
		\mathbf{T}_{\rm NS}, \;\mathbf{S}_{\rm K} &\in & L_{\rm loc}^2(0,\infty,[L^2_{\mu_m}]^{d\times d}), & & Q\ln(Q)&\in& L^\infty_{\rm loc}(0,\infty;L^1_{\mu_m}),
		\end{array}
		\end{equation*}
		and, for $R = Q R_m$,
		\begin{equation*}
		\begin{array}{rlcrl}
		{D^2(\sqrt{R})} 
		&\in  L_{\rm loc}^2(0,\infty,[L^2]^{d\times d}), 
		& \quad &  \nabla \left(R^{\frac14}\right) &\in  L_{\rm loc}^4(0,\infty,[L^4]^d), 
		\end{array}
		\end{equation*}
		with the compatibility conditions
		\begin{equation*}
		\sqrt{Q} \geq 0 \quad\mbox{a.e. on } (0,\infty)\times \mathbb{R}^d, \qquad \Lambda = 0 \quad\mbox{a.e. on } \left\{ \sqrt{Q} = 0\right\}.
		\end{equation*}
		The velocity $U$ is then defined by 
		\begin{equation*}
			U = \frac{\Lambda}{\sqrt{Q}}{\bf 1}_{\{\sqrt{Q} >0\}}.
		\end{equation*}

		\item The following equations are satisfied in $\mathcal{D}'((0,\infty)\times \mathbb{R}^d)$, 
		\begin{eqnarray*}
			\partial_t \sqrt{Q} + \frac{1}{\tau^2}{\dv}_m (\sqrt{Q}U) & = &\frac{1}{2\tau^2} \tr(\mathbf{T}_{\rm NS}) + \frac{1}{2\tau^2} \sqrt{Q}U \cdot \nabla_y(\ln(R_m)), \\
			\partial_t(\sqrt{Q}\sqrt{Q}U) + \frac{1}{\tau^2}\dvm(\sqrt{Q}U \otimes \sqrt{Q}U) + a\nabla_y\left(|\sqrt{Q}|^2\right) & = & \frac{1}{\tau^2}\dvm(2\nu \sqrt{Q}\mathbf{S}_{\rm NS} + 2\kappa^2\sqrt{Q}\mathbf{S}_{\rm K}) \\
			& & + \frac{2\nu\dot{\tau}}{\tau}\nabla_y\left(|\sqrt{Q}|^2\right)
		\end{eqnarray*}
		with
		\begin{equation*}
		\mathbf{S}_{\rm NS} = \frac{\mathbf{T}_{\rm NS} + \mathbf{T}_{\rm NS}^\top}{2}
		\end{equation*}
		and the compatibility conditions
		\begin{eqnarray*}
			\sqrt{Q}\mathbf{T}_{\rm NS} & = & \nabla(\sqrt{Q}\sqrt{Q}U) - 2\sqrt{Q}U \otimes \nabla(\sqrt{q}), \\
			\sqrt{Q}\mathbf{S}_{\rm K} & = & \sqrt{Q}D^2(\sqrt{Q})-  \nabla(\sqrt{Q}) \otimes \nabla(\sqrt{Q}).
		\end{eqnarray*}
		
		\item For any $\psi \in \mathcal{D}(\mathbb{R}^d)$,
		\begin{eqnarray*}
			\underset{t\to 0^+}{\lim} \int_{\mathbb{R}^d} \sqrt{Q}(t,x) \psi(x)\,{\rm d}x & = & \int_{\mathbb{R}^d} \sqrt{Q^0}(x)\psi(x) \,{\rm d}x, \\
			\underset{t\to 0^+}{\lim} \int_{\mathbb{R}^d} \sqrt{Q}(t,x) (\sqrt{Q}U)(t,x)\psi(x)\,{\rm d}x & = & \int_{\mathbb{R}^d} \sqrt{Q^0}(x)\Lambda^0(x)\psi(x) \,{\rm d}x.
		\end{eqnarray*}
	\end{enumerate}
\end{Definition}

One can then prove the following theorem by adapting the method developed in Sections \ref{sec:2} and \ref{sec:3} of this paper.

\begin{Theorem}
	Let $\nu >0$ and $\kappa > 0$. Let $(\sqrt{Q^0},\sqrt{Q^0}U^0) \in L^2_{\mu_m} \times [L^2_{\mu_m}]^d$ 
	satisfying furthermore $\mathcal{E}^\tau(Q^0,U^0) < \infty$ and $\mathcal{E}^\tau_{\rm BD}(Q^0,U^0) < \infty$ and the compatibility conditions
	\begin{equation*}
	\sqrt{Q^0} \geq 0 \quad \mbox{a.e. on } \mathbb{R}^d, \qquad  \sqrt{Q^0}U^0=0 \quad \mbox{a.e. on } \{ \sqrt{Q^0} =0\}.
	\end{equation*} 
	Then, there exists a global weak solution $(\sqrt{Q},\sqrt{Q}U)$ to \eqref{eq:CNSK_system_tau} in the sense of Definition \ref{def:WeakSolution_tau} satisfying the energy and entropy inequalities, namely there exist absolute constants $C$ and $C_{\rm BD}$ such that, for almost all $t\geq 0$,
	\begin{eqnarray*}
		\mathcal{E}^\tau(Q(t),U(t)) + \int_0^t \mathcal{D}^\tau(Q(t'),U(t')) \,{\rm d}t' & \leq & C(\mathcal{E}^\tau(Q^0,U^0)), \\
		\mathcal{E}_{\rm BD}^\tau(Q(t),U(t)) + \int_0^t \mathcal{D}_{\rm BD}^\tau(Q(t'),U(t')) \,{\rm d}t' & \leq & C_{\rm BD}(\mathcal{E}^\tau(Q^0,U^0), \mathcal{E}_{\rm BD}^\tau(Q^0,U^0)).
	\end{eqnarray*}

\end{Theorem}

\appendix

\section{From the original system \eqref{eq:CNSK_system} to the new system \eqref{eq:CNSK_system_m} and vice versa}
\label{appendixA}
\paragraph{Minimal density $\rho_m$.} Consider a small perturbation of the energy $E(\rho,u)$ given by \eqref{eq:def_energy_intro}, through a small variation of $(\rho,u)$:
\begin{align*}
E(\rho+\delta\rho,u+\delta u) &= E(\rho,u) + \int_{\mathbb{R}^d} \rho u \cdot \delta u + \int_{\mathbb{R}^d}\left(\frac12 |u|^2 + a\ln(\rho) + \kappa^2\left(\frac{\Delta\rho}{\rho} - \frac{|\nabla\rho|^2}{2\rho^2}\right)+\frac12\lambda|x|^2\right)\delta\rho \\
&\quad+ o\left(\|(\delta\rho,\delta u)\|^2\right).
\end{align*}
Therefore, for $(\rho_m,u_m)$ satisfying $\int_{\mathbb{R}^d} \rho_m = 1$, $\rho_m \geq 0$ and $E(\rho_m,u_m) >0$, to be minimizer of the energy, this couple has to satisfy
\begin{equation*}
\rho_mu_m = 0 \quad \mbox{and therefore} \quad a\ln(\rho_m) + \kappa^2\left(\frac{\Delta\rho_m}{\rho_m} - \frac{|\nabla\rho_m|^2}{2\rho_m^2}\right)+\frac12\lambda|x|^2 = c,
\end{equation*}
where $c$ is a constant associated to the mass condition for $\rho_m$ (Lagrange multiplier). This implies that $\rho_m$ is given by \eqref{def:rho_m_u_m}.

\paragraph{Potential force (Korteweg, pressure and potential).}

First, let us note that the Korteweg tensor, when $K(\rho)$ is given by \eqref{def:K(rho)quantum}, satisfies the Bohm identity
\begin{equation*}
2\kappa^2 \rho \nabla \left(\frac{\Delta\sqrt{\rho}}{\sqrt{\rho}}\right) = \kappa^2 \rho\nabla\left(\Delta \ln(\rho)+\frac12|\nabla\ln(\rho)|^2\right) = \kappa^2\dv\left(\rho D^2(\ln(\rho))\right).
\end{equation*}
Recall that $q = \frac{\rho}{\rho_m}$. A simple computation, using  $\dvm(v) = \frac{1}{\rho_m}\dv(\rho_m v) = \dv(v) + \nabla \ln(\rho_m) \cdot v$ and the explicit expression of $\nabla \ln(\rho_m) = -\frac{x}{\sigma^2}$, leads to

\begin{equation*}
2\kappa^2\frac{\Delta \sqrt{\rho}}{\sqrt{\rho}}  = 2\kappa^2\frac{\Delta_m(\sqrt{q})}{\sqrt{q}} + [V +a \ln(\rho_m)] - E(\rho_m,0),
\end{equation*} 
where $E(\rho_m,0)$ is given by \eqref{eq:def_energy_minimal}.
Therefore, the last three terms in the right-hand side of the Navier-Stokes Equation \eqref{eq:CNSK_system_velocity} become
\begin{equation}\label{eq:potential_force_m}
\rho\nabla\left(-a\ln(\rho) + 2\kappa^2\frac{\Delta \sqrt{\rho}}{\sqrt{\rho}} - V \right) = \rho\nabla\left(-a\ln(q) + 2\kappa^2\frac{\Delta_m(\sqrt{q})}{\sqrt{q}} \right).
\end{equation}

On the other hand, a direct calculation leads to: 
\begin{eqnarray*}
\kappa^2\dvm\left(q D^2(\ln(q))\right) - \frac{\kappa^2}{\sigma^2}q \nabla \ln(q) & = & \frac{\kappa^2}{\rho_m}\left[\dv\left(\rho D^2(\ln(q))\right) - \frac1{\sigma^4}\rho x\right] \\
&=& 2\kappa^2 q \nabla\left[\frac{\Delta \sqrt{\rho}}{\sqrt{\rho}}\right] - \frac{\kappa^2}{\sigma^4}q \nabla\left(\frac12|x|^2\right) \\
& = & q \nabla \left[2\kappa^2 \frac{\Delta \sqrt{\rho}}{\sqrt{\rho}}-V\right] + a q\nabla \left(\frac12\frac{|x|^2}{\sigma^2}\right) = 2\kappa^2 q \nabla \left(\frac{\Delta_m(\sqrt{q})}{\sqrt{q}}\right),
\end{eqnarray*}
where we used the definition of $\sigma$, see \eqref{def:sigma}, in the second line and \eqref{eq:potential_force_m} in the last one.

These two different expressions of the Korteweg tensor allow us to derive the energy and BD entropy identities for system \eqref{eq:CNSK_system_m}. 
Indeed, for the energy, we have
\begin{eqnarray*}
\int_{\mathbb{R}^d} 2q \nabla \left(\frac{\Delta_m(\sqrt{q})}{\sqrt{q}}\right) \cdot u \;{\rm d}\mu_m & = & 2 \int_{\mathbb{R}^d}  \left(\frac{\Delta_m(\sqrt{q})}{\sqrt{q}}\right) \frac{\partial q}{\partial t}   \;{\rm d}\mu_m = 4\int_{\mathbb{R}^d} \Delta_m(\sqrt{q}) \frac{\partial \sqrt{q}}{\partial t} \;{\rm d}\mu_m \\
& = & - 2\frac{\rm d}{{\rm d}t} \int_{\mathbb{R}^d} |\nabla \sqrt{q}|^2 \;{\rm d}\mu_m = - \frac12\frac{\rm d}{{\rm d}t} \int_{\mathbb{R}^d} q |\nabla \ln(q)|^2 \;{\rm d}\mu_m.
\end{eqnarray*}
In the original variable $\rho = q \rho_m$, we have (recall that $\int_{\mathbb{R}^d} q \;{\rm d}\mu_m = \int_{\mathbb{R}^d} \rho \;{\rm d}x = 1$)
\begin{eqnarray*}
	 \int_{\mathbb{R}^d} q|\nabla \ln(q)|^2 \;{\rm d}\mu_m & = & \int_{\mathbb{R}^d} \rho \left(|\nabla \ln(\rho)|^2 - 2 \nabla\ln(\rho) \cdot \nabla\ln(\rho_m) + |\nabla \ln(\rho_m)|^2 \right)    \;{\rm d}x \\
	& = & \int_{\mathbb{R}^d} \rho |\nabla \ln(\rho)|^2\;{\rm d}x - \frac{2d}{\sigma^2}  + \frac1{\sigma^2}\int_{\mathbb{R}^d}\rho \left|\frac{x}{\sigma}\right|^2    \;{\rm d}x.
\end{eqnarray*}
Using the expression of $\widetilde{I_2}(q)$ defined in \eqref{def:tildeI2}, we have
\begin{equation*}
\int_{\mathbb{R}^d} \rho |\nabla \ln(\rho)|^2\;{\rm d}x - \frac{d}{\sigma^2}  = \int_{\mathbb{R}^d} q|\nabla \ln(q)|^2 \;{\rm d}\mu_m - \frac{1}{\sigma^2}\widetilde{I_2}(q).
\end{equation*}

For the BD entropy, we have
\begin{equation*}
- \int_{\mathbb{R}^d} \left[\dvm(qD^2(\ln(q))) - \frac{1}{\sigma^2}\nabla q \right] \cdot \nabla\ln(q)  \;{\rm d}\mu_m =   \int_{\mathbb{R}^d} q |D^2(\ln(q))|^2 \;{\rm d}\mu_m + \frac{1}{\sigma^2} \int_{\mathbb{R}^d} q|\nabla \ln(q)|^2  \;{\rm d}\mu_m.
\end{equation*}
This formula can be expressed in terms of the original density $\rho = q \rho_m$ as follows
\begin{eqnarray*}
	& &  \int_{\mathbb{R}^d} q |D^2(\ln(q))|^2 \;{\rm d}\mu_m + \frac{1}{\sigma^2} \int_{\mathbb{R}^d} q|\nabla \ln(q)|^2  \;{\rm d}\mu_m \\
	&  & = \int_{\mathbb{R}^d} \rho|D^2(\ln(\rho))|^2 \;{\rm d}x + \frac{1}{\sigma^4} \widetilde{I_2}(q) - \frac{1}{\sigma^2}\int_{\mathbb{R}^d} \rho|\nabla\ln(\rho)|^2 \;{\rm d}x, 
\end{eqnarray*}
that is
\begin{equation*}
\int_{\mathbb{R}^d} \rho|D^2\ln(\rho)|^2   \;{\rm d}\mu_m - \frac{d}{\sigma^4} = \int_{\mathbb{R}^d} q |D^2(\ln(q))|^2 \;{\rm d}\mu_m + \frac{2}{\sigma^2} \int_{\mathbb{R}^d} q|\nabla \ln(q)|^2  \;{\rm d}\mu_m - \frac{2}{\sigma^4} \widetilde{I_2}(q).
\end{equation*}

These two calculations lead to the following result.

	\begin{Lemma}[Link between $\rho$ and $q$]\label{lemma:Appendix_rho_q1}
		For $q : \mathbb{R}^d \mapsto \mathbb{R}$ nonnegative, we define $\rho = q \rho_m$ and vice versa.
		
		Then, we have the following:
		
		\begin{equation*}
			\sqrt{q} \in L^2_{\mu_m} \mbox{ s.t. } \int_{\mathbb{R}^d} q \;{\rm d}\mu_m = 1 \quad \Longleftrightarrow \quad \sqrt{\rho} \in L^2 \mbox{ s.t. } \int_{\mathbb{R}^d} \rho \;{\rm d}x = 1.
		\end{equation*}
		Furthermore
		\begin{equation*}
			\sqrt{q} \in H^1_{\mu_m} \mbox{ s.t. } \int_{\mathbb{R}^d} q \;{\rm d}\mu_m = 1 \quad \Longleftrightarrow \quad \sqrt{\rho} \in H^1 \mbox{ s.t. } \int_{\mathbb{R}^d} \rho \;{\rm d}x = 1 \mbox{ and } \int_{\mathbb{R}^d} \rho |x|^2 < \infty.
		\end{equation*}
		Finally,
		\begin{eqnarray*}
		&  & \sqrt{q} \in H^1_{\mu_m} \mbox{ s.t. } \int_{\mathbb{R}^d} q \;{\rm d}\mu_m = 1  \mbox{ and } \int_{\mathbb{R}^d}q |D^2(\ln(q))|^2\;{\rm d}\mu_m < \infty \\
		& \Longleftrightarrow & \sqrt{\rho} \in H^1 \mbox{ s.t. } \int_{\mathbb{R}^d} \rho \;{\rm d}x = 1, \quad \int_{\mathbb{R}^d} \rho |x|^2 < \infty  \mbox{ and } \int_{\mathbb{R}^d} \rho |D^2(\ln(\rho))|^2 \;{\rm d}x < \infty.
		\end{eqnarray*}		
	\end{Lemma}

We have in addition the following useful identities.

\begin{Lemma}\label{lemma:Appendix_rho_q2}
	For $q : \mathbb{R}^d \mapsto \mathbb{R}$ nonnegative, we define $\rho = q \rho_m$ and vice versa. Then 
	\begin{equation*}
		D^2(\sqrt{q}) - \frac{\nabla(\sqrt{q}) \otimes \nabla (\sqrt{q})}{\sqrt{q}} = \frac{1}{\sqrt{\rho_m}}\left[D^2(\sqrt{\rho}) - \frac{\nabla(\sqrt{\rho}) \otimes \nabla (\sqrt{\rho})}{\sqrt{\rho}} + \frac{\sqrt{\rho}}{2\sigma^2}{\rm I}_d\right].
	\end{equation*}
	In particular, assuming that $\sqrt{q} \in L^2_{\mu_m}$ (for instance such that $\int_{\mathbb{R}^d} q \;{\rm d}\mu_m = 1$), we have
	\begin{equation*}
		\sqrt{q}D^2(\sqrt{q}) - \nabla(\sqrt{q}) \otimes \nabla (\sqrt{q}) \in L^1_{\mu_m} \quad \Longleftrightarrow \quad \sqrt{\rho}D^2(\sqrt{\rho}) - \nabla(\sqrt{\rho}) \otimes \nabla (\sqrt{\rho}) \in L^1.
	\end{equation*}
\end{Lemma}

\begin{proof}
	The proof of the lemma is straightforward by direct computations. The formulae used in the proof are given in the following remark.
\end{proof}

\begin{Remark}\label{rem:AppendixA-rho-rho_m}
	For smooth positive $q$ and $\rho = q \rho_m$, we have 
	\begin{equation*}
	\sqrt{q}D^2(\ln(q)) = 2 \left[D^2(\sqrt{q}) - \frac{\nabla(\sqrt{q}) \otimes \nabla (\sqrt{q})}{\sqrt{q}}\right] \quad \mbox{and} \quad \sqrt{\rho}D^2(\ln(\rho)) = 2 \left[D^2(\sqrt{\rho}) - \frac{\nabla(\sqrt{\rho}) \otimes \nabla (\sqrt{\rho})}{\sqrt{\rho}}\right]
	\end{equation*}
	but 
	\begin{equation*}
		D^2(\sqrt{q}) = \frac{1}{\sqrt{\rho_m}}\left[D^2(\sqrt{\rho}) + \nabla (\sqrt{\rho}) \overset{{\rm sym}}{\otimes} \frac{x}{\sigma^2} +\left[\frac1{4\sigma^4} x \otimes x + \frac1{2\sigma^2} {\rm I}_d\right]\sqrt{\rho} \right]
	\end{equation*}
	and 
	\begin{equation*}
		\frac{\nabla(\sqrt{q}) \otimes \nabla (\sqrt{q})}{\sqrt{q}} = \frac{1}{\sqrt{\rho_m}}\left[\frac{\nabla(\sqrt{\rho}) \otimes \nabla (\sqrt{\rho})}{\sqrt{\rho}} + \nabla (\sqrt{\rho}) \overset{{\rm sym}}{\otimes} \frac{x}{\sigma^2} +\frac{\sqrt{\rho}}{4\sigma^4} x \otimes x \right].
	\end{equation*}
		We used here the notation $\overset{{\rm sym}}{\otimes}$ to simplify the formulae, it reads, for $a =(a_k)_{k=1,\ldots,d}$, $b=(b_k)_{k=1,\ldots,d}$ in $\mathbb{R}^d$, as follows
	\begin{equation*}
	a \overset{{\rm sym}}{\otimes} b = \frac12 \left(a\otimes b + b \otimes a\right) = \frac{1}{2}(a_kb_l+a_lb_k)_{k,l}.
	\end{equation*}
	It corresponds to the symmetric part of the matrix $a\otimes b$.
	
	From theses computations, it seems that we cannot control $D^2(\sqrt{q})$ \big(resp. $	\frac{\nabla(\sqrt{q}) \otimes \nabla (\sqrt{q})}{\sqrt{q}}$\big) in $L^2_{\mu_m}$ from a control of $D^2(\sqrt{\rho})$ \big(resp. $\frac{\nabla(\sqrt{\rho}) \otimes \nabla (\sqrt{\rho})}{\sqrt{\rho}}$\big) in $L^2(\mathbb{R}^d)$  only and vice versa. Some bound on $I_4(q)$ must also be involved.	
\end{Remark}

Finally the following estimates were used several times in the main text.

\begin{Lemma}\label{lem:Hessian_estimate}
	There exists a positive constant $C_* >0$ (depending only on the dimension) such that for any positive smooth function $\rho : \mathbb{R}^d \mapsto \mathbb{R}$, we have
	\begin{equation*}
	\int_{\mathbb{R}^d} |D^2(\sqrt{\rho})|^2 \;{\rm d}x + \int_{\mathbb{R}^d}  |\nabla(\rho^{\frac14})|^4 \;{\rm d}x \leq C_*	\int_{\mathbb{R}^d} \rho |D^2(\ln(\rho))|^2 \;{\rm d}x.
	\end{equation*}

	There exists a positive constant $C_{m,*} >0$ (depending only on the dimension) such that for any positive smooth function $q : \mathbb{R}^d \mapsto \mathbb{R}$, we have
	\begin{equation*}
	\int_{\mathbb{R}^d}|D^2(\sqrt{q})|^2 {\rm d}\mu_m + \int_{\mathbb{R}^d}|\nabla({q}^{\frac14})|^4 {\rm d}\mu_m \leq C_{m,*}\left( \int_{\mathbb{R}^d}q|D^2(\ln({q}))|^2 {\rm d}\mu_m + \frac{1}{\sigma^4} \int_{\mathbb{R}^d} q\left|\frac{x}{\sigma}\right|^4{\rm d}\mu_m \right).
	\end{equation*}
\end{Lemma}

\begin{proof}
	The proof follows (exactly for the first inequality and with the additional term from the density $\rho_m$ for the second) the different steps of the proof of \cite[Lemma 2.1]{Vasseur_Yu_SIMA2016}.
	We prove here only the second one. 
	Denoting (following notations in \cite{Vasseur_Yu_SIMA2016})
	\begin{equation}\label{def:Lemma_Vasseur_Yu_ABD}
	A = \int_{\mathbb{R}^d} |D^2(\sqrt{q})|^2 {\rm d}\mu_m, \quad B = \int_{\mathbb{R}^d} |2\nabla (q^{\frac14})|^4 {\rm d}\mu_m, \quad D = \int_{\mathbb{R}^d} q |D^2(\ln(\sqrt{q}))|^2 {\rm d}\mu_m, \quad I_4(q) = \int_{\mathbb{R}^d} q \left|\frac{x}{\sigma}\right|^4 {\rm d}\mu_m,	
	\end{equation}
	we obtain after some integrations by parts and by using $|\tr(M)|^2 \leq d \|M\|^2$ for any $M \in \mathcal{M}_{d \times d}(\mathbb{R})$ that
	\begin{equation*}
	A+B \leq D + \sqrt{3BD} + \frac1\sigma (I_4(q))^{\frac14}B^{\frac34}.
	\end{equation*}
	This leads in particular to 
	\begin{equation*}
	A + \frac12 B \leq 4 D + \frac3{4\sigma^4} I_4(q).
	\end{equation*}
\end{proof}

\section{Inequalities in a Gaussian setting}\label{AppendixB}

\subsection{Logarithmic Sobolev inequality}

Let us denote ${\rm d}\gamma_d$ the normalized Gaussian measure on $\mathbb{R}^d$. Let $f \in H_{\gamma_d}^1$, then (see \cite{Gross75}) $|f|^2 \ln(|f|) \in L_{\gamma_d}^1$ and furthermore
\begin{equation*}
\int_{\mathbb{R}^d} |f(x)|^2 \ln(|f(x)|) \,{\rm d}\gamma_d \leq \int_{\mathbb{R}^d} |\nabla f(x)|^2 \;{\rm d}\gamma_d + \|f\|_{L_{\gamma_d}^2}^2\ln(\|f\|_{L_{\gamma_d}^2}).
\end{equation*} 
In particular, for $f = \sqrt{q}$ for which $\|f\|_{L_{\mu_m}^2}^2 = \int_{\mathbb{R}^d} q \,{\rm d}\mu_m = 1$, we have (after change of variables)
\begin{equation}\label{eq:LogarithmicSobolevInequality}
\int_{\mathbb{R}^d} q \ln(q)  \;{\rm d}\mu_m \leq \frac{2}{\sigma^2} \int_{\mathbb{R}^d} |\nabla \sqrt{q}|^2  \;{\rm d}\mu_m.
\end{equation}

\subsection{Strong Poincaré Inequality}

\begin{Proposition}\label{prop:SP_inequality}
There exists a constant $C_{\rm SP} >0$ such that, for all $f \in H_{\mu_m}^1$,
\begin{equation*}
\left\|\sqrt{1+|x|^2} \left(f - \int_{\mathbb{R}^d} f \;{\rm d}\mu_m\right)\right\|_{L_{\mu_m}^2} \leq C_{\rm SP} \|\nabla f\|_{[L_{\mu_m}^2]^d}.
\end{equation*}
\end{Proposition}

See \cite[Proposition 5]{Carrapatoso_Dolbeault_Herau_Mischler_Mouhot_ARMA22}. In particular, for $f = \sqrt{q}$, we obtain, using the Cauchy-Schwarz inequality to bound the mean value
\begin{equation*}
\int_{\mathbb{R}^d} \sqrt{q} \;{\rm d}\mu_m \leq 1
\end{equation*}
and
\begin{equation*}
\left\|\sqrt{1+\frac{|x|^2}{\sigma^2}}\right\|_{L_{\mu_m}^2} = \sqrt{1+d},
\end{equation*}
that
\begin{equation*}
\sqrt{\int_{\mathbb{R}^d} \left(1+\frac{|x|^2}{\sigma^2}\right)q  \;{\rm d}\mu_m} \leq \sqrt{C_{\rm SP}} \|\nabla \sqrt{q}\|_{[L^2_{\mu_m}]^d} + \sqrt{1+d}.
\end{equation*}

\subsection{Strong Poincaré-Korn Inequality}

\begin{Proposition}
	There exists a constant $C_{\rm SPK} >0$ such that, for all $u \in [H_{\mu_m}^1]^d$,
	\begin{equation*}
	\left\|\sqrt{1+|x|^2} \left(u - \int_{\mathbb{R}^d} u \;{\rm d}\mu_m - \Proj(u) \right)\right\|_{[L_{\mu_m}^2]^d} \leq C_{\rm SPK} \|D(u),
	\|_{[L_{\mu_m}^2]^{d\times d}}
	\end{equation*}
	where $\Proj(u)$ is the orthogonal projection from $[L_{\mu_m}^2]^d$ onto  $\mathcal{R}$ (the space of all the infinitesimal rotations) given by
	\begin{equation*}
	\mathcal{R} = \left\{ R : x \in \mathbb{R}^d \mapsto Ax \in \mathbb{R}^d \; \mbox{with } A \in M_{d\times d}(\mathbb{R}) \mbox{ s.t. } A^\top = -A  \right\}.
	\end{equation*}
\end{Proposition}

See \cite[Theorem 1]{Carrapatoso_Dolbeault_Herau_Mischler_Mouhot_ARMA22}. The projection $\Proj$ can be computed, see \cite[Appendix B]{Carrapatoso_Dolbeault_Herau_Mischler_Mouhot_ARMA22}, it is a continous operator from $[L_{\mu_m}^2]^d$ into $[H_{\mu_m}^1]^d$.

\subsection{Consequences of this inequalities.}

\begin{Proposition}\label{prop:StrongPoincare-consequences}$ $
	\begin{enumerate}
		\item There exists a constant $C_{\rm SP,m} >0$ such that, if $r \in H_{\mu_m}^1$, then $r x \in [L_{\mu_m}^2]^d$ and
			\begin{equation*}
				\|xr\|_{L_{\mu_m}^2} \leq C_{\rm SP} \|\nabla r\|_{[L_{\mu_m}^2]^d} + C_{\rm SP,m}\|r\|_{L_{\mu_m}^2}.
			\end{equation*}
		\item There exists a constant $C_{\rm SPK,m} >0$ such that, if $u \in [H_{\mu_m}^1]^d$, then $x \cdot u \in L_{\mu_m}^2$ and
		\begin{equation*}
		\|x\cdot u\|_{L_{\mu_m}^2} \leq C_{\rm SPK} \|D(u)\|_{[L_{\mu_m}^2]^{d\times d}} + C_{\rm SPK,m}\|u\|_{[L_{\mu_m}^2]^d}.
		\end{equation*}
		\item If $r \in H_{\mu_m}^2$, then $r x \in [H_{\mu_m}^1]^d$ and $|x|^2r \in L_{\mu_m}^2$ and
		\begin{equation*}
			\||x|^2 r\|_{L_{\mu_m}^2} + \|x r\|_{[H_{\mu_m}^1]^d} \leq C_{\rm SPK} \|D^2(r)\|_{[L_{\mu_m}^2]^{d\times d}} + C_{m}\|r\|_{L_{\mu_m}^2}.
		\end{equation*}  
	\end{enumerate} 
\end{Proposition}

\bibliographystyle{abbrv}

\vspace*{\fill}

\begin{small}	
\begin{itemize}
\item[1] Universit\'e de Lorraine, CNRS, IECL, F-57000 Metz, France.
\item[2] Universit\'e de Lorraine, CNRS, IECL, F-54000 Nancy, France.
\item[3] Universit\'e de Lorraine, CNRS, Inria, IECL, F-57000 Metz, France.
\item[] E-mail:  jeremy.faupin@univ-lorraine.fr, ingrid.lacroix@univ-lorraine.fr, julien.lequeurre@univ-lorraine.fr
\end{itemize}
\end{small}

\end{document}